\documentclass[11pt]{article}
\usepackage[dvipdfmx]{graphicx}
\usepackage[dvipdfmx]{color}
\usepackage{amsmath,amsthm,amssymb,amscd,ascmac,amsfonts,multicol,tabularx,url,stmaryrd, cases}

\usepackage[all]{xy}

\usepackage{here}
\usepackage{longtable}
\usepackage{multirow}

\makeatletter
\@addtoreset{equation}{section}

\makeatother

\newtheorem{main thm}{Theorem}
\newtheorem{df}{\quad Definition}[subsection]
\newtheorem{thm}[df]{Theorem}
\newtheorem{lem}[df]{Lemma}
\newtheorem{cor}[df]{Corollary}
\newtheorem{prop}[df]{Proposition}
\newtheorem{rem}[df]{Remark}
\newtheorem{eg}[df]{Example}

\newcommand{\Z}{\mathbb{Z}}
\newcommand{\Q}{\mathbb{Q}}

\newcommand{\Zp}{\mathbb{Z}_p}

\newcommand{\Fp}{\mathbb{F}_p}

\newcommand{\p}{\mathfrak{p}}
\newcommand{\q}{\mathfrak{q}}
\renewcommand{\P}{\mathfrak{P}}

\newcommand{\alp}{\alpha}

\newcommand{\gam}{\gamma}
\newcommand{\Gam}{\Gamma}
\newcommand{\del}{\delta}
\newcommand{\Del}{\Delta}
\newcommand{\ep}{\varepsilon}

\newcommand{\Gal}{{\rm Gal}}
\newcommand{\Hom}{{\rm Hom}}
\newcommand{\Ker}{{\rm Ker}}
\newcommand{\Coker}{{\rm Coker}}

\newcommand{\rank}{{\rm rank}}
\newcommand{\longto}{\longrightarrow}

\newcommand{\ord}{{\rm ord}}
\newcommand{\isom}{\simeq}

\newcommand{\ki}{k_{\infty}}

\newcommand{\Fi}{F_{\infty}}

\newcommand{\kcyc}{k_{\rm cyc}}

\newcommand{\x}{\times}
\newcommand{\op}{\oplus}

\newcommand{\ox}{\otimes}
\newcommand{\inj}{\hookrightarrow}
\newcommand{\surj}{\twoheadrightarrow}

\renewcommand{\tilde}{\widetilde}

\DeclareMathOperator*{\zetaprod}{{\displaystyle{\prod\kern-1.45em\coprod}}}

\newcommand{\D}[1]{D_{#1}}
\newcommand{\Do}[1]{D_{\overline{#1}}}
\newcommand{\I}[1]{I_{#1}}
\newcommand{\Io}[1]{I_{\overline{#1}}}

\begin{document}
\begin{center}
{\Large {On the triviality of the unramified Iwasawa modules 
\\
of the maximal multiple $\mathbb{Z}_p$-extensions}
}\\
\medskip{Keiji OKANO
\footnote{
Department of Teacher Education,
Tsuru University,
3-8-1 Tahara, Tsuru-shi, Yamanashi 402-0054, Japan.
email:
{\tt 
okano@tsuru.ac.jp
}
}
}
\end{center}


\begin{abstract}
For a number field $k$ and an odd prime number $p$, we consider the maximal multiple $\mathbb{Z}_p$-extension $\tilde{k}$ of $k$ and the unramified Iwasawa module $X(\tilde{k})$, 
which is the Galois group of the maximal unramified abelian $p$-extension of $\tilde{k}$. 
In this article, we classify the CM-fields $k$ in which $p$ splits completely and for which $X(\tilde{k}) = 0$.
To this end, we provide a relation between the trivialities of $X(\tilde{k})$ and that of the $p$-split Iwasawa modules associated with certain anticyclotomic $\mathbb{Z}_p$-extensions.
In addition, we provide an alternative proof of the sufficient condition for $X(\tilde{k})=0$, based on 
the study of the generalized Greenberg conjecture.
\end{abstract}


\section{Introduction}\label{section-intro}
Let $p$ be an odd prime number, and let $k$ be a finite algebraic number field.
A $\Zp$-extension $\ki$ of $k$ is a Galois extension of $k$ whose Galois group $\Gal(\ki/k)$ is isomorphic to the additive group of $p$-adic integers $\Zp$.
The maximal unramified abelian $p$-extension of $\ki$ is one of the central objects in classical Iwasawa theory.
As a generalization, we consider the maximal multiple $\Zp$-extension $\tilde{k}$ of $k$ and the Galois group $X(\tilde{k})$ of the maximal unramified abelian $p$-extension of $\tilde{k}$.
We refer to $X(\tilde{k})$ as the unramified Iwasawa module of $\tilde{k}$.
It is known that $X(\tilde{k})$ is a finitely generated torsion $\Zp \llbracket \Gal(\tilde{k}/k) \rrbracket$-module,
where $\Zp \llbracket \Gal(\tilde{k}/k) \rrbracket$ denotes the completed group ring.
Greenberg conjectured in \cite{Greenberg2001} that $X(\tilde{k})$ is a pseudo-null $\Zp \llbracket \Gal(\tilde{k}/k) \rrbracket$-module.
In other words, the conjecture states that the annihilator ideal of $X(\tilde{k})$ is not contained in any height-one ideal of $\Zp \llbracket \Gal(\tilde{k}/k) \rrbracket$.
This conjecture is usually referred to as the Generalized Greenberg Conjecture (GGC, for short).
Roughly speaking, it suggests that $X(\tilde{k})$ is ``small'' in a certain sense. 
In fact, if we specialize $k$ to totally real fields that satisfy Leopoldt's conjecture for $p$, then $\tilde{k}$ coincides with the cyclotomic $\Zp$-extension $\kcyc$ of $k$, and therefore
the conjecture implies that the unramified Iwasawa module of $\kcyc$ should be finite
(this is the original Greenberg conjecture proposed in \cite{Greenberg1976}).

There are many results that support the original and the generalized Greenberg conjecture, 
as well as significant results that follow from these conjectures.
Here, we review some results concerning the GGC.
Minardi \cite{Minardi}, Itoh \cite{Itoh2011}, and Fujii \cite{Fujii2015} provided conditions under which $X(\tilde{k})$ is pseudo-null for CM-fields $k$ in which $p$ splits completely.
(The ideas in their work will be used in \S\ref{another proof in case non-Gal deg 4} and \S\ref{Proof in the case deg 6 |K| notequiv 0} of this paper.)
Moreover, as an application, Fujii \cite{Fujii2022} constructed an infinite family of fields $k$ satisfying both the GGC and the condition $X(\tilde{k}) \neq 0$.
A weak version of the GGC was also proposed by Nguyen Quang Do \cite{Nguyen Quang Do2006, Nguyen Quang Do2017} and Wingberg \cite{Wingberg Free pro-p extensions}.
It asserts that if $X(\tilde{k}) \neq 0$, then it contains a nontrivial pseudo-null submodule. 
Fujii \cite{Fujii2010} and Kleine \cite{Kleine2022} investigated the relationship between the weak GGC and the GGC.
In addition, \cite{Fujii2010}, Shirakawa \cite{Shirakawa}, and Murakami \cite{Murakami2023} provided many examples that satisfy the weak GGC and the GGC.
On the other hand, Kurihara (unpublished) and Kataoka \cite{Kataoka2022} proved that if $k$ is an imaginary quadratic field in which $p$ splits, then $X(\tilde{k})$ has no nontrivial finite submodule.
These results suggest that the unramified Iwasawa modules $X(\tilde{k})$ may not be particularly small.

These studies on the size of the unramified Iwasawa module $X(\tilde{k})$ naturally lead to the following question: 
When $X(\tilde{k})$ is trivial? 
In the case where $k$ is totally real and satisfies Leopoldt's conjecture for $p$ (so that $\tilde{k}=\kcyc$), 
there are known conditions under which the unramified Iwasawa module $X(\kcyc)$ is trivial.
For example, Iwasawa's criterion \cite{Iwasawa1956} (see also Washington \cite[Theorem 10.1]{Was}) provides a necessary and sufficient condition for $X(\kcyc)=0$ under certain assumptions.
Fukuda's criterion \cite{Fukuda1994} also provides a sufficient condition for $X(\kcyc)=0$.
Yamamoto \cite{Yamamoto2000} determined all absolutely abelian $p$-extensions $k$ for which $X(\kcyc)$ vanishes.
Moreover, there are several results concerning the density of real quadratic fields $k$ for which $X(\kcyc)$ vanishes (see Taya and Yamamoto \cite{Taya-Yamamoto2010} and the references therein).
We note that in the case $p=2$, 
Mouhib and Movahhedi \cite{Mouhib-Movahhedi2011} 
listed real quadratic fields $k$ for which $X(\kcyc)$ either vanishes or is cyclic.
In addition, Taya and Yamamoto \cite{Taya-Yamamoto2006} determined all real abelian $2$-extensions $k$ for which $X(\kcyc)$ vanishes.

In contrast to the totally real case discussed above, 
few results are known when $k$ is a CM-field.
Indeed, in this case, the known results may be the only ones that are readily accessible.
For example, when $k$ is an imaginary quadratic field, we refer the reader to Theorem~\ref{main thm}(I) below and to \cite[Theorem 1.6]{MOMO2021}.
We note that in both the totally real and CM cases where $k$ is completely decomposed at $p$, Jaulent \cite{Jaulent2024} provided necessary conditions for $X(\tilde{k})$ to be trivial.
Therefore, we aim to determine the CM-fields $k$ for which $X(\tilde{k})=0$.
In this paper, we give a complete classification of CM-fields $k$ in which $p$ splits completely and for which $X(\tilde{k})=0$
(Theorem \ref{main thm}).
To this end, we provide a relation between the trivialities of $X(\tilde{k})$ and that of the $p$-split Iwasawa modules associated with certain anticyclotomic $\mathbb{Z}_p$-extensions
(Theorem \ref{main prop}).
We note that for such CM-fields $k$, the triviality of $X(\tilde{k})$ implies 
that the maximal unramified $p$-extension of $\kcyc$ is abelian.
Finally, we remark that 
all the main results of the paper are proved under Leopoldt's conjecture
and that 
our results do not require the assumption that the GGC holds; rather, they provide a new example satisfying the GGC.

\subsection{Main theorems}
\label{subsection main thm}

Let $p$ be an odd prime number, and $k$ a CM-field in which $p$ splits completely.
Assume that Leopoldt's conjecture holds for $k$ and $p$.
Let $A(k)$ denote the $p$-Sylow subgroup of the ideal class group of $k$.
Let $\kcyc$ be the cyclotomic $\Zp$-extension of $k$ and $\tilde{k}$ the maximal multiple $\Zp$-extension of $k$.
We denote by $X(\kcyc)$ and $X(\tilde{k})$ the Galois groups of the maximal unramified abelian $p$-extensions of $\kcyc$ and $\tilde{k}$, respectively.
Moreover, for a $\Zp$-extension $\ki$ of $k$, we define the $p$-split Iwasawa module $X'(\ki)$ of $\ki$ as the Galois group of the maximal unramified abelian $p$-extension of $\ki$ in which every prime lying above $p$ splits completely.
Let $k^+$ denote the maximal totally real subfield of $k$.
A $\Zp$-extension $\ki$ of $k$ is called an anticyclotomic $\Zp$-extension if $\ki/k^+$ is a non-abelian Galois extension,
in other words, 
the complex conjugation acts on $\Gal(\ki/k)$ by inversion.
Our first main theorem is as follows.

\begin{thm}\label{main prop}
Let $p$ be an odd prime number, and let $k$ be a CM-field in which $p$ splits completely and which satisfies Leopoldt's conjecture for $k$ and $p$.
Assume that $X(\tilde{k})=0$.
Then, for the cyclotomic $\Zp$-extension $\kcyc$ of $k$ and any anticyclotomic $\Zp$-extension $\ki$ of $k$ in which every prime ideal of $k$ lying above $p$ is non-split and ramified in $\ki/k$, the $p$-split Iwasawa modules $X'(\kcyc)$ and $X'(\ki)$ are trivial.
\end{thm}

\begin{rem}
{\rm 
In fact, after proving the next theorem, we will see that the converse of the claim in Theorem \ref{main prop} holds if $[k:\Q] \le 6$.
If $[k:\Q] > 6$, the author does not know whether the converse of the claim in Theorem \ref{main prop} holds
(see Question (1) at the end of this paper).
}
\end{rem}

For a prime ideal $\P$ of $k$ lying above $p$, denote its complex conjugate by $\overline{\P}$.
Moreover, for a set $T$ of prime ideals of $k$ lying above $p$, 
denote by $M_p^T(k)$ the maximal abelian $p$-extension of $k$ that is unramified outside $p$ and in which every prime ideal in $T$ splits completely.
Our next main theorem is as follows.

\begin{thm}\label{main thm}
Let $p$ be an odd prime number, and $k$ a CM-field in which $p$ splits completely and which satisfies Leopoldt's conjecture for $k$ and $p$.
If $X(\tilde{k})=0$, then $[k:\Q] \le 6$.
Furthermore, depending on the degree of $k$, the following hold:
\begin{itemize}
\setlength{\parskip}{0pt} 
\setlength{\itemsep}{0pt} 
\item[{\rm (I)}]
In the case where $[k:\Q]=2${\rm :}
$X(\tilde{k})=0$ if and only if $X(\kcyc) \isom \Zp$.
\item[{\rm (II)}]
In the case where $[k:\Q]=4${\rm :}
Let $(p)=\p \q \overline{\p} \overline{\q}$ be the prime decomposition of $p$ in $k$,
and define $T=\{\p, \overline{\p} \}$.
Then $X(\tilde{k})=0$ if and only if all of the following conditions {\rm (i)}, {\rm (ii)}, and {\rm (iii)} hold:
\begin{itemize}
\setlength{\parskip}{0pt} 
\setlength{\itemsep}{0pt} 
\item[{\rm (i)}]
$A(k)=0$ or $A(k)$ is cyclic.
\item[{\rm (ii)}]
$X(\kcyc) \isom \Zp^{\op [k:\Q]/2}$ as a $\Zp \llbracket \Gal(\kcyc/k) \rrbracket$-module;
in other words, $L(\kcyc)=\tilde{k}$.
\item[{\rm (iii)}]
$M_p^T(k)=k$.
\end{itemize}
In fact, when $A(k)$ is cyclic, if conditions {\rm (i)} and {\rm (ii)} hold, then condition {\rm (iii)} also holds.
Therefore, in this case, condition {\rm (iii)} can be omitted.
\item[{\rm (III)}]
In the case where $[k:\Q]=6${\rm :}
Let $(p)=\p_1 \p_2 \p_3 \overline{\p_1}\; \overline{\p_2}\; \overline{\p_3} $ be the prime decomposition of $p$ in $k$, and define $T_1=\{\p_1, \overline{\p_1} \}$ and $T_2=\{ \p_2, \overline{\p_2} \}$.
Then $X(\tilde{k})=0$ if and only if all of the following conditions {\rm (i)}, {\rm (ii)}, and {\rm (iii)} hold:
\begin{itemize}
\setlength{\parskip}{0pt} 
\setlength{\itemsep}{0pt} 
\item[{\rm (i)}]
$A(k)=0$ or $A(k)$ is cyclic.
\item[{\rm (ii)}]
$X(\kcyc) \isom \Zp^{\op [k:\Q]/2}$ as $\Zp \llbracket \Gal(\kcyc/k) \rrbracket$-modules;
in other words, $L(\kcyc)=\tilde{k}$.
\item[{\rm (iii)}]
In the compositum $M_p^{T_1}(k)M_p^{T_2}(k)/k$, the prime ideal $\p_3$ does not split at all.
\end{itemize}
\end{itemize}
\end{thm}

\begin{rem}\label{remark of main thm}
{\rm 
Case (I) is well known.
In addition, Jaulent \cite{Jaulent2024} provides a computational method for detecting when $X(\kcyc) \isom \Zp$.

In conditions (ii) in both cases (II) and (III), the isomorphism 
$X(\kcyc) \isom \Zp^{\op [k:\Q]/2}$ 
as $\Zp \llbracket \Gal(\kcyc/k) \rrbracket$-modules is equivalent to $L(\kcyc)=\tilde{k}$.

Conditions (iii) in both cases (II) and (III) are equivalent to the statement that Scholz's number knot $\mathcal{K}(\Gal(L(\kcyc)/k))$ of $\Gal(L(\kcyc)/k)$, which appears in central class field theory and is defined in Proposition~\ref{central}, 
is trivial
(see Lemma~\ref{mathcal{K} in case non-Gal deg 4} and Lemma~\ref{|K| not cong 0 iff}(i)).
We can check whether condition (iii) in case (II) holds by computing ray class groups.

In case (III), note that $M_p^{T_1}(k)M_p^{T_2}(k)$ is not equal to $M_p^{ T_1 \cup T_2 }(k)$.
When $p=3$, $[k:\Q]=6$, and $A(k)$ is cyclic, 
if conditions (i) and (ii) hold, then condition (iii) in case (III) also holds.
Therefore, similarly to the case where $[k:\Q]=4$ and $A(k)$ is cyclic, condition (iii) can be omitted in this case
(see Remark~\ref{in case deg 6 dim=1 p=3}).
In \cite{Okano2012-1}, using Kida's formula in Iwasawa theory, 
the special case of Theorem~\ref{main thm}(III) with $p = 3$ and $k$ a totally imaginary abelian extension of degree $6$ was proved.
%
}
\end{rem}

\begin{rem}\label{remark of main thm 2}
{\rm 
Recently, Yasushi Mizusawa and the author obtained another computational method for checking whether $X(\tilde{k})=0$ using arithmetic linking numbers, 
as described in \cite{Mizusawa-Okano2026}.
}
\end{rem}


\subsection{Organization of the paper}\label{Organization of the paper}

\S\ref{notation} introduces the notation and basic facts frequently used throughout this paper. 
In \S\ref{Theorems in central p-class field theory}, we present central class field theory,
which is used to prove the sufficient condition in Theorem \ref{main thm}.
In \S\ref{section anticyclotomic}, 
we prove Theorem \ref{main prop}, which is used to prove the necessary condition in Theorem~\ref{main thm}.
In \S\ref{Narrowing-down target}, we narrow down the number of cases to be considered.
In \S\ref{section deg 4} and \S\ref{section deg 6}, we prove the main results for the cases of degrees $4$ and $6$, respectively
(see Figure \ref{stragegy}).
Although the results in the case of degree $4$ are not used in the degree $6$ case,
the basic idea presented in \S\ref{section deg 4} is also applied in \S\ref{section deg 6}.
In \S\ref{another proof in case non-Gal deg 4} and \S\ref{Proof in the case deg 6 |K| notequiv 0}, we provide alternative proofs of the sufficient conditions for the main results in the cases of degrees $4$ and $6$, respectively, whose ideas are based on those of the GGC.
In \S\ref{Examples}, we give some examples.
\begin{figure}[H]
$$ 
\begin{xy}
\xymatrix{ 
*+[F-:<3pt>]{
X(\tilde{k}) = 0
}
\ar@{=>}[rrr]^{\S \ref{section anticyclotomic}} 
\ar@{<=}[d]_{ \S \ref{Theorems in central p-class field theory}}
\ar@{<=}[rrrd]^{\S \ref{another proof in case non-Gal deg 4}}
_{\S \ref{Proof in the case deg 6 |K| notequiv 0}}
&&&
*+[F-:<3pt>]{
\text{
Theorem \ref{main prop}
}
}
\ar@{=>}[d]^{
\S \ref{proof in case non-Gal deg 4}
}
_{
\S \ref{Pf of in case deg 6}
}
 \\
*+[F-:<3pt>]{
\text{
$\mathcal{K}(\Gal(L(\kcyc)/k))=0$
}
}
\ar@{<=}[rrr]^{
\hspace*{-40pt}
\text{
Lemma \ref{mathcal{K} in case non-Gal deg 4}
}
}
_{
\hspace*{-35pt}
\text{
Lemma \ref{|K| not cong 0 iff}(i)
}
}
&&&
*+[F-:<3pt>]{
\text{
 Conditions (i)(ii)(iii) in Theorem \ref{main thm}
}}
}
\end{xy}
$$
\caption{}
\label{stragegy}
\end{figure}


\section{
Preliminaries 
and the proof of Theorem \ref{main prop}
}\label{section-prelim}

\subsection{Notation}\label{notation}

Throughout the remainder of this paper, let $p$ denote an odd prime number.
We begin by introducing notation for groups and modules.
For elements $x,y, \ldots$ in a $\Zp$-module $M$, we write  
$\langle x,y,\ldots \rangle$ 
for the $\Zp$-submodule of $M$ generated by $x,y, \ldots$.
The same notation is used to denote the $\Zp$-submodule generated by elements and submodules.
For example, if $x, y \in M$ and $N \subset M$ is a submodule, 
then $\langle x,y,N \rangle$ 
denotes the $\Zp$-submodule generated by $x$, $y$, and all elements of $N$.
For an abelian group $G$ and a $G$-module $M$, 
let $M^G$ denote the maximal submodule on which $G$ acts trivially, 
and let $M_G$ denote the maximal quotient on which  $G$ acts trivially.
In particular, if $G$ is generated by complex conjugation, we denote $M^+:=M^G$ and $M^-:=M/M^+$.
Under the assumption that $p$ is odd,
$M^-$ is isomorphic to the maximal submodule of $M$ on which $G$ acts as multiplication by $-1$.

As in \S\ref{subsection main thm}, for any finite extension $F$ of $\Q$, 
let $F_{\rm cyc}$ and $\tilde{F}$ denote the cyclotomic $\Zp$-extension and the maximal multiple $\Zp$-extension of $F$, respectively.
Note that every prime ideal of $F$ lying above $p$ ramifies in $F_{\rm cyc}$
(see Neukirch, Schmidt, and Wingberg \cite[Proposition 11.1.1(ii)]{NSW}).
Let $A(F)$ be the $p$-Sylow subgroup of the ideal class group of $F$,
and $D(F)$ the subgroup of $A(F)$ consisting of ideal classes containing a power of a prime ideal lying above $p$.
Let $L(F)$ denote the maximal unramified abelian $p$-extension of $F$, and set $X(F):=\Gal(L(F)/F)$.
The same notation $L(F)$ and $X(F)$ is used when $F$ is an infinite extension of $\Q$.

Fix a topological generator $\overline{\gam} \in \Gal(F_{\rm cyc}/F)$.
Then $\Gal(F_{\rm cyc}/F)$ acts on $X(F_{\rm cyc})$ by inner automorphisms, 
as shown in Lemma \ref{G/H act on H}.
It is well known that $X(F_{\rm cyc})$ is a finitely generated torsion module over the completed group ring
$\Zp \llbracket \Gal(F_{\rm cyc}/F) \rrbracket$
(see \cite[Proposition 11.1.4]{NSW} or \cite[Lemma~13.18]{Was}).
For further properties of $X(F_{\rm cyc})$,
see \cite[Chapter XI]{NSW} or \cite[\S~13]{Was}.
Moreover, for a $\Zp$-extension $\Fi$, we define the $p$-split Iwasawa module $X'(\Fi)$ of $\Fi$ as the Galois group of the maximal unramified abelian $p$-extension of $\Fi$ in which every prime lying above $p$ splits completely.

Let $r_1(F)$ and $r_2(F)$ denote the number of real primes and the number of pairs of complex primes of $F$, respectively.
By class field theory, 
$$
\Gal(\tilde{F}/F) \isom \Zp^{\op (r_2(F)+1)},
$$ 
if Leopoldt's conjecture for $F$ and $p$ holds
(see \cite[Theorem 11.1.2]{NSW} or \cite[Theorem~13.4]{Was}).
For convenience, we introduce the following notation.

\begin{df}\label{leftrightarrow}
For a closed subgroup $H$ of $\Gal(\tilde{F}/F)$ and an intermediate field $M$ of $\tilde{F}/F$, we write
$M \leftrightarrow H$
to mean that $M$ is the fixed field of $H$.
\end{df}

We now introduce the following lemmas, 
which will be used frequently in this paper.

\begin{lem}\label{G/H act on H}
Let $G$ be a profinite group with a closed normal abelian subgroup $H$.
Then $G/H$ acts on $H$ by inner automorphisms.
Furthermore, suppose that $G$ is a pro-$p$ group and that $G/H$ is isomorphic to $\Zp$.
Then, $G$ is abelian if and only if this action is trivial.
\end{lem}

\begin{proof}
We define the action of $G/H$ on $H$ as follows.
Let $\overline{g} \in G/H$, and choose an extension $g \in G$ of $\overline{g}$.
Then, for any $x \in H$, we define $\overline{g}(x):=gxg^{-1} \in H$.
This is independent of the choice of the extension $g$, since $H$ is abelian.
Moreover, this implies that 
$\overline{g}(\overline{h}(x))=
(\overline{g}\overline{h})(x)$ 
for any $\overline{g}, \overline{h} \in G/H$ and $x \in H$.
Therefore, this defines an action of $G/H$ on $H$.

Suppose that $G$ is a pro-$p$ group and that $G/H$ is isomorphic to $\Zp$.
We can take a topological generator $\bar{g} \in G/H$.
Then, the action is trivial if and only if $\bar{g}(x)=x$ for all $x \in H$, which is equivalent to $G$ being abelian.
This completes the proof.
\end{proof}

\begin{lem}\label{tilde{F} in L(F_inf)}
Assume that $p$ splits completely in $F$. 
Let $\Fi$ be a $\Zp$-extension of $F$ in which every prime ideal of $F$ lying above $p$ ramifies,
and let $M_p(F)$ be the maximal abelian $p$-extension of $F$ that is unramified outside $p$.
Then $M_p(F)/\Fi$ is an unramified abelian $p$-extension.
In other words, $M_p(F)$ is contained in $L(\Fi)$.
In particular, $\tilde{F}$  is contained in $L(\Fi)$.
\end{lem}

\begin{proof}
The proof is similar to that at the beginning of \S~2 in Fujii \cite{Fujii2011}.
Since $p$ splits completely in $F$, the inertia group $I$ in $\Gal(M_p(F)/F)$ of a prime ideal of $F$ lying above $p$ is isomorphic to $\Zp$.
By the assumption on $\Fi$, $I$ maps to $\Gal(\Fi/F)$ injectively to finite cokernel. 
This implies that $I \cap \Gal(M_p(F)/\Fi)$ is trivial, and hence, $M_p(F)/\Fi$ is unramified at all prime ideals of $F$ lying above $p$.
Combining the definition of $M_p(F)$ with this, 
we see that $M_p(F)/\Fi$ is an unramified abelian $p$-extension.
Since it is known that $\tilde{F} \subset M_p(F)$, 
we have $\tilde{F} \subset L(\Fi)$.
This completes the proof.
\end{proof}

\begin{cor}\label{X surj Zp^r2}
Under the notation and assumptions of Lemma~{\rm \ref{tilde{F} in L(F_inf)}}, there exists a surjective homomorphism 
$X(\Fi) \surj \Zp^{\op r_2(F)}$
of $\Zp \llbracket \Gal(\Fi/F) \rrbracket$-modules.
\end{cor}

\begin{proof}
By Lemma~\ref{tilde{F} in L(F_inf)}, there exists a natural surjection
$X(\Fi) \surj \Gal(\tilde{F}/\Fi)$.
Moreover, by Lemma~\ref{G/H act on H}, the group $\Gal(\Fi/F)$ acts trivially on $\Gal(\tilde{F}/\Fi)$.
Since $\rank_{\Zp} \Gal(\tilde{F}/\Fi) \ge r_2(F)$, the assertion follows.
\end{proof}

We note that the above lemma and corollary do not require the validity of Leopoldt's conjecture for $F$ and $p$,
and that $\Gal(\tilde{F}/\Fi) \isom \Zp^{\op r_2(F)}$ as 
$\Zp \llbracket \Gal(\Fi/F) \rrbracket$-modules 
if the conjecture holds.

The following lemma will be used in \S\ref{Proof in the case deg 6 |K| notequiv 0}.

\begin{lem}\label{Galois L/F}
Let $K/F$ be a Galois extension, possibly of infinite degree, 
and let $H$ be a normal subgroup of $\Gal(K/F)$.
Denote by $L$ the subfield in $L(K)/K$ such that $\Gal(L/K) \isom X(K)_H$.
Then $L/F$ is a Galois extension.
\end{lem}

\begin{proof}
Since $K/F$ is a Galois extension and $L(K)$ is the maximal abelian extension of $K$ unramified at all primes, the extension $L(K)/F$ is also Galois.
Therefore, it suffices to show that $\Gal(L(K)/L)$ is a normal subgroup of $\Gal(L(K)/F)$.
By definition, 
$$
\Gal(L(K)/L)=\left[ X(K), \Gal( L(K)/K^H ) \right].
$$ 
Here, for subgroups $H_1$, $H_2$ of a pro-$p$ group, the notation $[H_1, H_2]$ denotes the closed normal subgroup generated by elements of the form 
$[h_1,h_2]:=h_1 h_2 h_1^{-1} h_2^{-1}$ 
($h_i \in H_i$). 
Note that $\Gal(L(K)/K^H)$ is a normal subgroup of $\Gal(L(K)/F)$ since $K^H/F$ is Galois.
For any $x \in X(K)$, $\sigma \in \Gal( L(K)/K^H )$, and $\tau \in \Gal(L(K)/F)$, we have
$
\tau [x,\sigma] \tau^{-1}
=
[\tau x \tau^{-1}, \tau \sigma \tau^{-1}].
$
Since $\tau x \tau^{-1} \in X(K)$ and $\tau \sigma \tau^{-1} \in \Gal( L(K)/K^H )$, the proof is complete.
\end{proof}


\subsection{Theorems in central class field theory}
\label{Theorems in central p-class field theory}

Let $p$ be an odd prime number, 
let $F$ be a finite algebraic number field, 
and let $K/F$ be a finite abelian $p$-extension.
Set $G:=\Gal(L(K)/F)$.
Then $G$ acts on $X(L(K))$ via inner automorphisms, since $X(L(K))$ is abelian.
We define the central $p$-class field $\mathcal{C}_{L(K)/F}$ associated with $L(K)/F$ as 
the subfield of $L(L(K))$ over $L(K)$ such that 
\begin{eqnarray}\label{central galois}
\Gal(\mathcal{C}_{L(K)/F}/L(K)) \isom 
X(L(K))_G.
\end{eqnarray}
Note that $X(L(K))=0$ if and only if 
$X(L(K))_G=0$,
by Nakayama's lemma.
For a prime ideal $\mathfrak{l}$ of $F$ that is ramified in $K/F$, we fix a prime ideal of $L(K)$ lying above $\mathfrak{l}$ 
and denote its decomposition group in $G$ by $D_\mathfrak{l}$. 
Then we have the following proposition from central class field theory (see Fr\"{o}hlich~\cite[Proposition 3.6, (3.24), and Theorem 3.11]{Frohlich}).

\begin{prop}\label{central}
With the above notation, we define Scholz's number knot associated with $G$ as follows:
$$
\mathcal{K}(G):=\Coker \left(
\prod_{\mathfrak{l}} H_2(D_{\mathfrak{l}},\Zp)
\to
H_2(G,\Zp)
\right),
$$
where the map is induced by the canonical inclusion $D_{\mathfrak{l}} \to G$, 
and the product runs over all prime ideals of $F$ that are ramified in $K$.
Then we obtain the following exact sequence:
\begin{eqnarray}\label{central ex}
\hspace{-40pt}
&&
(E(F) \cap N_{L(K)/F}\mathbb{A}^\x_{L(K)})\ox_{\Z}\Zp
	\to
\mathcal{K}(G)
	\to
X(L(K))_G
	\to
0,
\end{eqnarray}
where $\mathbb{A}^\x_{L(K)}$ denotes the id\`{e}le group of $L(K)$, 
and $N_{L(K)/F}$ denotes the norm map.
\end{prop}

\begin{rem}
{\rm 
Since we assume that $K/F$ is abelian, the genus field associated with $L(K)/F$, defined in \cite[\S2,~page 13]{Frohlich}, coincides with $L(K)$.
}
\end{rem}

If either $K$ or $F$ is an infinite extension, 
we obtain an exact sequence similar to (\ref{central ex})
by taking the projective limit over its finite subextensions.
On the other hand, suppose that $K=F$.
Then we have 
$
E(F) \cap N_{L(F)/F}\mathbb{A}^\x_{L(F)}=E(F),
$
and
$
H_2(X(F),\Zp) \isom A(F) \wedge_{\Zp} A(F).
$
Hence, if $L(L(F))= L(F)$,
i.e., $X(L(F))=0$, 
then the induced map $E(F) \to A(F)\wedge_{\Zp}A(F)$ is surjective.
Therefore, we obtain the following.

\begin{cor}\label{nec of finite}
If $X(L(F))=0$, then
$$
\dim_{\Fp}A(F)/p \le \frac{1+\sqrt{1+8(r_1(F)+r_2(F)+\nu-1)}}{2},
$$
where $\nu=1$ or $0$, according to whether a primitive $p$th root of unity is contained in $F$ or not.
\end{cor}

If $\mathcal{K}(G)=0$, then $X(L(K))_G=0$, and hence $X(L(K))=0$ by Nakayama's lemma.
The module
$\mathcal{K}(G)$ can be computed as follows.
We consider the commutative diagram arising from the minimal presentations
$$
\begin{CD}
1 @>>> \mathcal{R} @>>> \mathcal{F} @>>> G @>>> 1\\
  & & @AAA @AAA @AAA \\
1 @>>> \mathcal{R}_{\mathfrak{l}} @>>> \mathcal{F}_{\mathfrak{l}} @>>> D_{\mathfrak{l}} @>>> 1\\
\end{CD}
$$
of $G$ and the decomposition groups $D_{\mathfrak{l}}$, presented by the free pro-$p$ groups $\mathcal{F}$ and $\mathcal{F}_{\mathfrak{l}}$, respectively.
Since the Hochschild--Serre spectral sequence yields the following isomorphisms,
\begin{eqnarray}\label{homology isom commutator}
H_2(G,\Zp) 
\isom 
\mathcal{R} \cap [F,F]/[R,F],
\ \ 
H_2(D_\mathfrak{l},\Zp) 
\isom 
\mathcal{R}_{\mathfrak{l}} \cap [\mathcal{F}_{\mathfrak{l}},\mathcal{F}_{\mathfrak{l}}]/[\mathcal{R}_{\mathfrak{l}},\mathcal{F}_{\mathfrak{l}}],
\end{eqnarray}
we conclude that $\mathcal{K}(G)=0$ if and only if
\begin{eqnarray}\label{judge}
\Phi\colon
\prod_{\mathfrak{l}}
{\mathcal{R}_{\mathfrak{l}} \cap [\mathcal{F}_{\mathfrak{l}},\mathcal{F}_{\mathfrak{l}}]}/
{[\mathcal{R}_{\mathfrak{l}},\mathcal{F}_{\mathfrak{l}}]}
\longto
{R \cap [F,F]}/{[R,F] }
\end{eqnarray}
is surjective.


\subsection{$\Zp$-extensions with the action of the complex conjugation}
\label{section anticyclotomic}

Hereafter, we assume that $k$ is a CM-field.
Let $k^+$ denote the maximal totally real subfield of $k$.

\begin{lem}\label{cokernel D(k_n) to X(k_n)^{Gam}}
Let $p$ be an odd prime number, $k$ a CM-field that does not contain all primitive $p$th roots of unity, and $\ki$ a $\Zp$-extension of $k$.
Let $k_n$ denote the intermediate field of $\ki/k$ of degree $p^n$, 
and set $\Gam:=\Gal(\ki/k)$ and $\Gam_n:=\Gal(k_n/k)$.
Suppose that $\Gal(k/k^+)$ acts on $\Gam$ via inner automorphisms, 
and that every prime ideal of $k$ lying above $p$ is non-split and ramified in $\ki/k$.
Then the induced action of $\Gal(k/k^+)$ on the cokernel
$$
\Coker
\left(
\varprojlim_{\rm norm} D(k_n) \inj X(\ki)^\Gam
\right)
$$
is trivial.
\end{lem}

\begin{proof}
The proof is essentially the same as that of \cite[Lemma 4.1]{Okano2006}.
By assumption that $\Gal(k/k^+)$ acts on $\Gam$ via inner automorphisms, $\ki/k^+$ is an abelian or a dihedral extension.
Note that $\Gam$ acts trivially on $D(k_n)$ because every prime ideal of $k$ lying above $p$ does not split in $\ki/k$.
Also note that $\Gal(k/k^+)$ acts on $X(\ki)^\Gam$  via inner automorphisms.
Let $E(k_n)$, $P(k_n)$, and $I(k_n)$ denote the unit group, the principal ideal group, and the group of fractional ideals of $k_n$, respectively.
Consider the following two exact sequences of $\Gam_n$-modules:
$$
0 \to E(k_n) \to k_n^\x \to P(k_n) \to 0,
\ \ 
0 \to P(k_n) \to I(k_n) \to I(k_n)/P(k_n) \to 0.
$$
Taking the cohomological long exact sequences of the above two short exact sequences,
we obtain the following exact sequence:
\begin{eqnarray}\label{i(A(k))D(k_n) to A(k_n)^{Gam} to unit}
0 \to i_n(A(k))D(k_n) \to A(k_n)^{\Gam} \to 
\dfrac{E(k) \cap N_{\Gam_n} k_n^\x}{N_{\Gam_n} E(k_n)} \to 0,
\end{eqnarray}
where $i_n \colon A(k) \to A(k_n)$ is the lifting map, 
and $N_{\Gam_n}$ is the norm operator.
Since we assume that $k$ does not contain all primitive $p$th roots of unity, the group index $[E(k):E(k^+)]$ is coprime to the odd prime number $p$
(see \cite[Theorem 4.12]{Was}).
Therefore, $\Gal(k/k^+)$ acts trivially on the right-hand term in (\ref{i(A(k))D(k_n) to A(k_n)^{Gam} to unit}).
Moreover, 
since $\Gal(\ki/k^+)$ is abelian or dihedral, it follows that the norm from $k_n$ to $k_{n-1}$ on each term in 
(\ref{i(A(k))D(k_n) to A(k_n)^{Gam} to unit}) 
commutes with the action of $\Gal(k/k^+)$.
Note that $\varprojlim i_n(A(k))=0$.
Therefore, by taking the projective limit with respect to the norm maps of (\ref{i(A(k))D(k_n) to A(k_n)^{Gam} to unit}), we obtain an injective morphism
$
\varprojlim D(k_n) \inj X(\ki)^\Gam
$
of $\Gal(k/k^+)$-modules,
whose cokernel has a trivial $\Gal(k/k^+)$-action.
\end{proof}

\begin{cor}\label{cor cyclotomic case}
With the notation and assumptions of Lemma~{\rm \ref{cokernel D(k_n) to X(k_n)^{Gam}}}, 
let $\ki=\kcyc$ and $\Gam=\Gal(\kcyc/k)$.
Then, for the $p$-split Iwasawa module $X'(\kcyc)$ of $\kcyc$, 
we have
$$
X'(\kcyc)^- \isom 
\left( X(\kcyc)/X(\kcyc)^\Gam \right)^-.
$$
\end{cor}

\begin{proof}
We have $X'(\kcyc) \isom \varprojlim_{\rm norm}(A(k_n)/D(k_n))$.
Note that $\Gal(k/k^+)$ acts on $X(\kcyc)$, 
since $\Gal(\kcyc/k^+) \isom \Gam \x \Gal(k/k^+)$.
By Lemma~\ref{cokernel D(k_n) to X(k_n)^{Gam}}, we obtain 
$
\left(\varprojlim D(k_n)\right)^-
\isom
\left(X(\kcyc)^\Gam \right)^-.
$
Hence, we have 
$$
X'(\kcyc)^- 
\isom 
\left(\varprojlim A(k_n) \right)^-/\left(\varprojlim D(k_n) \right)^-
\isom 
X(\kcyc)^-/\left(X(\kcyc)^\Gam \right)^-.
$$
This completes the proof.
\end{proof}

\begin{rem}
{\rm 
For CM-fields $k$ in which $p$ splits completely,
Jaulent~\cite{Jaulent2024} used the same result as in Corollary~\ref{cor cyclotomic case} to provide a sufficient condition for $X(\tilde{k}) \neq 0$.
On the other hand, we will use Theorem \ref{main prop} to prove that $X(\tilde{k}) \neq 0$.
}
\end{rem}

We now prove Theorem \ref{main prop}.
Under the assumptions of Theorem \ref{main prop}, assume that $X(\tilde{k})=0$.
Let $\ki$ be either an anticyclotomic $\Zp$-extension of $k$ in which every prime ideal of $k$ lying above $p$ is non-split and ramified in $\ki/k$, or the cyclotomic $\Zp$-extension of $k$.
Then,  we have $\ki \subset \tilde{k} \subset L(\ki)$ by Lemma~\ref{tilde{F} in L(F_inf)}.
Moreover, since $X(\tilde{k}) = 0$, we obtain $\tilde{k} = L(\ki)$.
Therefore, 
$X(\ki)=X(\ki)^\Gam$.
Hence, it follows that the action of $\Gal(k/k^+)$ on $X(\ki)$ is naturally induced, 
and 
we have the following exact sequence:
$$
0 \to \varprojlim D(k_n) \to X(\ki)^\Gam \to X'(\ki) \to 0.
$$
Thus, $X'(\ki)$ also carries an action of $\Gal(k/k^+)$.
By Lemma~\ref{cokernel D(k_n) to X(k_n)^{Gam}}, this action is trivial, 
i.e., $X'(\ki)^-=0$.
Therefore, it suffices to show that $X'(\ki)^+=0$.

Suppose that $\ki=\kcyc$.
Then,  since $\tilde{k} = L(\kcyc)$ and Leopoldt's conjecture holds for $k$ and $p$,  we have
\begin{eqnarray}\label{X(kcyc) isom Zp^{op [k:Q]/2} in proof}
X(\kcyc) \isom \Zp^{\op [k:\Q]/2}
\ \ 
\text{
as $\llbracket \Gal(\kcyc/k) \rrbracket$-modules.
}
\end{eqnarray}
This isomorphism yields 
\begin{eqnarray}\label{X(kcyc)^+=0}
X(\kcyc)^+=0.
\end{eqnarray}
Indeed, let $M_p(k^+)$ denote the maximal abelian $p$-extension of $k^+$ that is unramified outside $p$.
Then, by Lemma \ref{tilde{F} in L(F_inf)}, $M_p(k^+) \subset L(\kcyc^+)$.
This implies that $M_p(k^+)$ is the maximal subfield of $L(\kcyc^+)$ that is abelian over $k^+$.
Therefore,  
\begin{eqnarray}\label{X(kcyc^+)_{Gal(kcyc^+/k^+)} isom Gal(M_p(k^+)/kcyc^+)}
X(\kcyc^+)_{\Gal(\kcyc^+/k^+)} \isom \Gal(M_p(k^+)/\kcyc^+).
\end{eqnarray}
Since Leopoldt's conjecture holds for $k^+$ and $p$, 
the right-hand side of 
(\ref{X(kcyc^+)_{Gal(kcyc^+/k^+)} isom Gal(M_p(k^+)/kcyc^+)}) 
is finite. 
By (\ref{X(kcyc) isom Zp^{op [k:Q]/2} in proof}), the left-hand side of 
(\ref{X(kcyc^+)_{Gal(kcyc^+/k^+)} isom Gal(M_p(k^+)/kcyc^+)}) 
is isomorphic to $X(\kcyc)^+$, which is a $\Zp$-free module.
Therefore, $X(\kcyc^+)_{\Gal(\kcyc^+/k^+)}=0$, and  $X(\kcyc^+)=0$ by Nakayama's lemma.
Thus, we obtain $X'(\kcyc)^+=0$.

On the other hand, suppose that $\ki$ is an anticyclotomic $\Zp$-extension of $k$ in which every prime ideal of $k$ lying above $p$ is non-split and ramified in $\ki/k$.
We regard $X(\ki)$ as a $\Gal(k/k^+)$-submodule of $\Gal(\tilde{k}/k)$.
Since $X(\kcyc)^+=0$, we see that $\Gal(\tilde{k}/\kcyc)$ is the maximal submodule of $\Gal(\tilde{k}/k)$ on which the complex conjugation acts by multiplication by $-1$.
Let $M$ denote the fixed field of the maximal submodule of $X(\ki)$ on which the complex conjugation acts by multiplication by $-1$.
Then, we have the following isomorphisms as $\Gal(k/k^+)$-modules:
$$
\Gal(\kcyc/k) 
\isom 
\Gal(\tilde{k}/k)^+
=
X(\ki)^+
\isom
\Gal(M/\ki)
=\Gal(M/k)^+.
$$
Since $\Gal(M/k)^- =\Gal(\ki/k)$, we obtain
$$
\Gal(M/k) \isom \Gal(\kcyc/k)  \op \Gal(\ki/k)
\isom \Zp^{\op 2}.
$$
We define the field $L'(\ki)$ by $\Gal(L'(\ki)/\ki) =X'(\ki)$.
Since $X'(\ki)^-=0$, taking the plus-parts of the natural surjection
$X(\ki) \to X'(\ki)$,
we obtain a surjective map 
$\Gal(M/\ki) \to X'(\ki)$.
This implies 
$$
\ki \subset L'(\ki) \subset M.
$$
Assume that $X'(\ki) \neq 0$.
Then, since no prime ideal of $k$ lying above $p$ splits in $\kcyc/k$, every such prime ideal must split in $M/\kcyc$.
Since $M/\kcyc$ is a $\Zp$-extension, this implies that $X'(\kcyc) \neq 0$, which is a contradiction. 
This completes the proof of Theorem \ref{main prop}.




\subsection{Narrowing down the number of cases}
\label{Narrowing-down target}

In this subsection, we introduce the following lemma, which follows from well-known results.

\begin{lem}\label{target}
Let $p$ be an odd prime number, and let $k$ be a CM-field satisfying Leopoldt's conjecture for $k$ and $p$.
Suppose that $p$ splits completely in $k$.
If $X(\tilde{k})=0$, then the following conditions hold{\rm :}
\begin{itemize}
\setlength{\parskip}{0pt} 
\setlength{\itemsep}{0pt} 
\item[{\rm (i)}]
$[k:\Q] \le 6$.
\item[{\rm (ii)}]
$\dim_{\Fp}A(k)/pA(k) \le 2$.
\item[{\rm (iii)}]
$X(\kcyc) \isom \Zp^{\op [k:\Q]/2}$ as a $\Zp\llbracket \Gal(\kcyc/k) \rrbracket$-module;
in other words, $L(\kcyc)=\tilde{k}$.
\end{itemize}
\end{lem}

\begin{proof}
As in the proof of Theorem \ref{main prop}, (iii) follows from Lemma~\ref{tilde{F} in L(F_inf)} and the assumption $X(\tilde{k})=0$.
Condition (i) is the same as that in \cite[Lemma 3.9]{Fujii2022} or \cite[Th\'{e}or\`{e}me 10(i)]{Jaulent2024}.
Finally, we show that condition (ii) holds.
Since $X(\tilde{k})=0$,  it follows that $X(L(k))=0$.
Indeed, if $X(L(k)) \neq 0$, then 
$L(L(k)) \neq L(k)$.
Since $L(L(k))$ is not abelian over $k$, it is not contained in $\tilde{k}$.
Hence, $L(L(k))\tilde{k}/\tilde{k}$ is a nontrivial unramified $p$-extension.
This contradicts the assumption that $X(\tilde{k})=0$.
Thus, we obtain the inequality
$$
\dim_{\Fp}A(k)/pA(k) \le \frac{1+\sqrt{4[k:\Q]-7}}{2} <3
$$
by Corollary~\ref{nec of finite} and condition (i).
\end{proof}

It is known that, when $k$ is an imaginary quadratic field, $X(\tilde{k})=0$ if and only if $X(\kcyc) \isom \Zp$.
Indeed, if $X(\kcyc) \isom \Zp$, then we have $L(\kcyc)=\tilde{k}$ and 
$H_2(X(\kcyc),\Zp) \isom X(\kcyc) \wedge_{\Zp} X(\kcyc)=0$.
Therefore, $\mathcal{K}(X(\kcyc))$ in Proposition~\ref{central} is trivial,
which implies $X(\tilde{k})=0$ by Proposition~\ref{central} and Nakayama's lemma.
Conversely, if $X(\tilde{k})=0$, then $L(\kcyc)=\tilde{k}$,
which implies $X(\kcyc) \isom \Zp$.
Therefore, to prove Theorem~\ref{main thm}, 
it suffices to consider CM-fields $k$ of degree $4$ or $6$ that satisfy the following conditions:
\begin{itemize}
\setlength{\parskip}{0pt} 
\setlength{\itemsep}{0pt} 
\item[(I)]
Leopoldt's conjecture for $k$ and $p$ holds.
\item[(II)]
$p$ splits completely in $k$.
\item[(III)]
$\dim_{\Fp}A(k)/pA(k) \le 2$ and  
$X(\kcyc) \isom \Zp^{\op [k:\Q]/2}$
as a $\Zp\llbracket \Gal(\kcyc/k) \rrbracket$-module.
\end{itemize}
In particular, in this case, it is implicitly assumed that 
``Iwasawa's $\mu(\kcyc/k)=0$ conjecture'' holds (meaning that $X(\kcyc)$ is finitely generated as a $\Zp$-module) and that $X(\kcyc^+)=0$.
Indeed, condition (iii) implies that $\mu(\kcyc/k)=0$.

\begin{lem}\label{A(k)=D(k)}
Suppose that $k$ satisfies the above conditions {\rm (I)},  {\rm (II)}, and  {\rm (III)}.
Then 
$A(k)=D(k)$, $X(\kcyc)^+=0$, and $X'(\kcyc)=0$.  
\end{lem}

\begin{proof}
Similar to the argument used to obtain (\ref{X(kcyc)^+=0}) from (\ref{X(kcyc) isom Zp^{op [k:Q]/2} in proof}) in the proof of Theorem \ref{main prop}, we obtain $X(\kcyc)^+ \isom X(\kcyc^+)=0$.
In particular, $X'(\kcyc)^+$ is also trivial.
Hence, by Corollary~\ref{cor cyclotomic case}, we have 
$
X'(\kcyc)
\isom
\left( X(\kcyc)/X(\kcyc)^{\Gal(\kcyc/k)} \right)^-
$.
Combining this with condition (III), we obtain the triviality of $X'(\kcyc)$.
This means that the maximal abelian $p$-extension of $\kcyc$ in which every prime lying above $p$ splits completely coincides with $\kcyc$.
Hence, we obtain $A(k)=D(k)$ by class field theory, since every prime ideal in $k$ above $p$ is totally ramified in $\kcyc/k$. 
\end{proof}

Hereinafter, we write the elements of $\Gal(\tilde{k}/k)$ multiplicatively.


\section{The case of degree $4$}\label{section deg 4}


\subsection{The non-Galois or cyclic case}\label{case: non-Gal deg 4}
\subsubsection{Setting and preparation}

Let $k$ be either a non-Galois CM-field of degree $4$ over $\Q$ or a totally imaginary cyclic extension of degree $4$ over $\Q$.
Suppose that $k$ satisfies conditions (I), (II), and (III) given in \S\ref{Narrowing-down target}.
Let $F/\Q$ be the Galois closure of $k/\Q$, and denote its Galois group by $\Delta:=\Gal(F/\Q)$.
Then it is known that 
$$
\Delta \isom 
\begin{cases}
\Z/4\Z 
=\{ \sigma \,|\, \sigma^4=1 \}
&
\text{if $k$ is cyclic},
\\
D_8
:=\{ \sigma,\tau \,|\, \sigma^4=\tau^2=1,\ \tau \sigma \tau^{-1}=\sigma^{-1} \}
&
\text{if $k$ is non-Galois}.
\end{cases}
$$
We identify $\Del$ with the corresponding group accordingly.
The center of $\Del$ is $\langle \sigma^2 \rangle$ when $\Del=D_8$.
By \cite[Corollary~1.5]{MilneCM06}, the field $F/\Q$ is also a CM-field.
It follows that $F^+/\Q$ is Galois, 
and therefore $\sigma^2$ is the complex conjugation;
that is, 
$\Gal(F/F^+)=\langle \sigma^2 \rangle$.
Fix a prime ideal $\P$ of $F$ lying above $p$.
Since $p$ splits completely in $k/\Q$, 
it also splits completely in $F/\Q$.
Therefore, the prime decomposition of $p$ in $F$ is given by
$
(p)=\prod_{g \in \Delta}g(\P).
$
Define 
$\mathfrak{Q}:=\sigma(\P)$, $\overline{\P}:=\sigma^2(\P)$, and $\overline{\mathfrak{Q}}:=\sigma^3(\P)$.
Then $\overline{\P}$ and $\overline{\mathfrak{Q}}$ are the complex conjugates of $\P$ and $\mathfrak{Q}$, respectively.
Define 
$$
\p:=N_{F/k}\P, 
\ \ 
\q:=N_{F/k}\mathfrak{Q}, 
\ \ 
\overline{\p}:=N_{F/k} \overline{\P}, 
\ \ 
\overline{\q}:=N_{F/k} \overline{\mathfrak{Q}}.
$$

\begin{lem}\label{primes are distinct in the case deg 4}
The four prime ideals $\p$, $\q$, $\overline{\p}$, and $\overline{\q}$ are distinct.
\end{lem}

\begin{proof}
In the case where $\Delta=\Z/4\Z$, there is nothing to prove.
Therefore, we assume that $\Delta=D_8$.
We prove only that $\p$ differs from the other prime ideals.
If $\p=\overline{\p}$, then 
$N_{\Gal(F/k)}(\P)= N_{\Gal(F/k)} \circ \sigma^2(\P)$.
By the uniqueness of the prime ideal factorization of $N_{\Gal(F/k)}(\P)$, 
we have $\sigma^2 \in \Gal(F/k)$, which implies that $k$ is totally real.
This is a contradiction.
Assume that $\p$ coincides with another prime ideal other than $\overline{\p}$.
Then, similarly, either $\sigma$ or $\sigma^3$ is contained in $\Gal(F/k)$.
This implies that $\Gal(F/k)$ has an element of order $4$.
However, the order of $\Gal(F/k)$ is $2$.
This leads to a contradiction.
\end{proof}

We denote by $D_{\p}$, $D_\q$, $D_{\overline{\p}}$, and $D_{\overline{\q}}$ the decomposition groups of 
$\p$, $\q$, $\overline{\p}$, and $\overline{\q}$ in $\Gal(\tilde{k}/k)$, respectively.
Similarly, we denote by $I_{\p}$, $I_\q$, $I_{\overline{\p}}$, and $I_{\overline{\q}}$ the inertia groups of $\p$, $\q$, $\overline{\p}$, and $\overline{\q}$ in $\Gal(\tilde{k}/k)$, respectively.
Note that these inertia groups are isomorphic to $\Zp$, since $p$ splits completely in $k$.

\begin{lem}\label{complex embed in case non-Gal deg 4}
Under the above notation and assumptions, the restriction $\sigma |_k \in \Hom(k,F)$ of $\sigma$ to $k$ acts on $\Gal(\tilde{k}/k)$ and on $X(\kcyc)$ via inner automorphisms.
Moreover,
$$
\sigma|_k (D_\p)=D_\q,
\ \ 
\sigma|_k (D_\q)=D_{\overline{\p}},
\ \ 
\sigma|_k (D_{\overline{\p}})=D_{\overline{\q}}.
$$
\end{lem}

\begin{proof}
Note that $\Gal(\tilde{F}/F)$ is endowed with an action of $\sigma$.
Let $D_\P$, $D_{\mathfrak{Q}}$, $D_{\overline{\P}}$, and $D_{\overline{\mathfrak{Q}}}$ denote the decomposition groups in $\Gal(\tilde{F}/F)$ of $\P$, $\mathfrak{Q}$, $\overline{\P}$, and $\overline{\mathfrak{Q}}$, respectively.
Then, by the definition of these prime ideals, we have 
$\sigma(D_\P)=D_{\mathfrak{Q}}$, $\sigma(D_{\mathfrak{Q}})=D_{\overline{\P}}$, and $\sigma(D_{\overline{\P}})=D_{\overline{\mathfrak{Q}}}$. 
Since $p$ splits completely in $F/\Q$, the canonical projection $\pi_\P \colon D_\P \to D_\p$ is an isomorphism.
From this, it follows that 
$$
\sigma|_k (D_\p)=D_\q
$$ 
is induced.
Indeed, using another isomorphism 
$\pi_{\mathfrak{Q}} \colon D_{\mathfrak{Q}} \to D_\q$, 
the map $\sigma|_k \colon D_\p \to D_\q$ can be written as
$\sigma|_k=\pi_{\mathfrak{Q}} \circ \sigma \circ \pi_{\mathfrak{P}}^{-1}$.
Similarly, $\sigma|_k (D_\q)=D_{\overline{\p}}$ and $\sigma|_k (D_{\overline{\p}})=D_{\overline{\q}}$ are induced.
Note that any two of these maps agree on the intersection of their domains of definition.
Since 
\begin{eqnarray}\label{Gal(tilde{k}/k) is generated by these four decomposition groups}
\Gal(\tilde{k}/k)=\langle D_\p, D_\q,D_{\overline{\p}}, D_{\overline{\q}} \rangle
\end{eqnarray}
by the condition $A(k)=D(k)$ in Lemma \ref{A(k)=D(k)},
$\sigma$ induces an action on $\Gal(\tilde{k}/k)$.
Via the canonical isomorphism $\Gal(F_{\rm cyc}/F) \isom \Gal(\kcyc/k)$, 
we observe that $\sigma|_k$ acts on $\Gal(\kcyc/k)$
and that the action is trivial.
Hence, $\sigma|_k$ also acts on $X(\kcyc)=\Ker(\Gal(\tilde{k}/k) \to \Gal(\kcyc/k))$.
This completes the proof.
\end{proof}

We now show that the decomposition and inertia groups can be described using certain parameters, as stated in the following lemma.
This result serves as a key step in the proof of our main theorem.

\begin{lem}\label{Z in case non-Gal deg 4}
Let $\gam \in \Gal(\tilde{k}/k)$ be a topological generator of $I_{\p}$.
Then there exist $x, y \in X(\kcyc)$ and $a,b \in \Zp$ such that 
$$
\begin{cases}
\Gal(\tilde{k}/k)=\langle \gam, x,y \rangle,
&
X(\kcyc)=\langle x,y \rangle,
\\
D_\p=\langle \gam \rangle \x \langle x \rangle,
&
D_\q=\langle \gam x^ay^b \rangle \x \langle y \rangle,
\\
D_{\overline \p}=\langle \gam x^{a-b}y^{a+b} \rangle \x \langle x^{-1} \rangle,
&
D_{\overline \q}=\langle \gam x^{-b}y^a\rangle \x \langle y^{-1} \rangle.
\end{cases}
$$
Here, the left factor of each direct product is the inertia group corresponding to the respective prime ideal.
\end{lem}

\begin{proof}
Let $\sigma|_k$ be as in Lemma \ref{complex embed in case non-Gal deg 4}.
By Lemma \ref{A(k)=D(k)}, every prime ideal lying above $p$ is inert in the extension $\tilde{k}/\kcyc$,
and $\Gal(\tilde{k}/k)$ is written as (\ref{Gal(tilde{k}/k) is generated by these four decomposition groups}).
Choose $x \in X(\kcyc)$ such that $D_{\p}=\langle \gam \rangle \x \langle x \rangle$, 
and set $y:=\sigma|_k(x)$.
Note that $y \in X(\kcyc)$ by Lemma \ref{complex embed in case non-Gal deg 4}, 
and that 
$(\sigma|_k)^2$ is the complex conjugation, which acts on
$X(\kcyc)=X(\kcyc)^-$ as inverse.
Applying $\sigma|_k$ repeatedly to each of $D_\p$ and $I_\p$, we obtain
\begin{eqnarray}\label{decomp group in proof in the case non-Gal deg 4}
D_\q=\langle \sigma|_k(\gam) \rangle \x \langle y \rangle,
\ \ 
D_{\overline \p}=\langle (\sigma|_k)^2(\gam) \rangle \x \langle x^{-1} \rangle,
\ \ 
D_{\overline \q}=\langle (\sigma|_k)^3(\gam) \rangle \x \langle y^{-1} \rangle
\end{eqnarray}
again by Lemma \ref{complex embed in case non-Gal deg 4}.
Here, the left factor of each direct product is the inertia group corresponding to the respective prime ideal.
Since $X'(\kcyc)=0$ by Lemma \ref{A(k)=D(k)} again, we have
$$
X(\kcyc)=
\left\langle
D_{\p} \cap X(\kcyc),\; D_{\q} \cap X(\kcyc),\; D_{\overline \p} \cap X(\kcyc),\; D_{\overline \q} \cap X(\kcyc)
\right\rangle
=
\langle
x,y
\rangle.
$$
Since $\sigma|_k$ acts trivially on $\Gal(\kcyc/k)$,
it follows that
$\sigma|_k(\gam) \equiv \gam$ modulo $X(\kcyc)$.
Therefore, there exist $a,b \in \Zp$ such that 
$$
\sigma|_k(\gam)=\gam x^a y^b.
$$
Combining this with (\ref{decomp group in proof in the case non-Gal deg 4}), we obtain the desired expressions for the decomposition groups.
Finally, 
by (\ref{Gal(tilde{k}/k) is generated by these four decomposition groups}),
we have $\Gal(\tilde{k}/k)=\langle \gam, x,y \rangle$.
\end{proof}

\begin{rem}\label{a neq 0, b neq 0, a+b neq 0}
{\rm
(i) 
In the remainder of \S \ref{case: non-Gal deg 4}, 
we will not use the Galois action of $\Gal(k/\Q)$, except for the complex conjugation $J \in \Gal(k/k^+)$.
\\
(ii)
We have
$$
a \neq 0,\ b \neq 0,\ a+b \neq 0.
$$
Indeed, suppose that $a=0$.
Then $D_\q=\langle \gam, y \rangle \supset \langle I_\p, I_\q \rangle$.
This implies that the fixed field of $D_\q$ is a $\Zp$-extension of $k$ that is unramified outside $\{ \overline{\p}, \overline{\q}\}$ and is infinitely decomposed at $\q$.
This contradicts \cite[Lemma 3]{Fujii2015}.
Similarly, we have $b \neq 0$.
The last inequality, $a+b \neq 0$, also follows in the same way:
if $a+b=0$, then the fixed field of $D_\p=D_{\overline{\p}}=\langle \gam, x \rangle$ is a $\Zp$-extension of $k$, unramified outside $\{ \q, \overline{\q} \}$ and is infinitely decomposed at $\p$ and $\overline{\p}$.
This again contradicts the same lemma in \cite{Fujii2015}.
}
\end{rem}

\begin{cor}\label{A(k) in case non-Gal deg 4}
Depending on the number of generators of $A(k)$, the following conditions on $a$ and $b$ hold.
\begin{itemize}
\setlength{\parskip}{0pt} 
\setlength{\itemsep}{0pt} 
\item[{\rm (i)}]
$A(k)=0$ $\iff$ $a^2+b^2 \not\equiv 0 \pmod p$.
\item[{\rm (ii)}]
$A(k)$ is cyclic $\iff$ $a \not\equiv 0$, $b \not\equiv 0$, $a^2+b^2 \equiv 0 \pmod p$.
\item[{\rm (iii)}]
$\dim_{\Fp}A(k)/pA(k)=2$ $\iff$ $a \equiv b \equiv 0 \pmod p$.
\end{itemize}
\end{cor}

\begin{proof}
By Lemma \ref{Z in case non-Gal deg 4}, we have
$
A(k) \isom 
\langle \gam, x, y \rangle/\langle I_\p, I_\q, I_{\overline{\p}},I_{\overline{\q}} \rangle
=
\langle \gam, x, y \rangle/\langle \gam, x^ay^b,x^{-b}y^a \rangle.
$
Hence, we have a presentation of $A(k)$ with the following representation matrix:
$$
A:=
\begin{bmatrix}
1 & 0 & 0
\\
0 & a & b
\\
0 & -b & a
\end{bmatrix}.
$$
In other words, there exists a linear map $\langle \gam, x, y \rangle \to \langle \gam, x, y \rangle$ whose cokernel is $A(k)$ and whose representation matrix is $A$.
Therefore, it suffices to consider the rank of $\overline{A}:=A \bmod p$:
$$
\left\{
\begin{array}{ccc}
A(k)=0 & \iff & \rank\; \overline{A} =3 
\\
\text{$A(k)$ is cyclic} & \iff & \rank\; \overline{A} =2 
\\
\dim_{\Fp}A(k)/pA(k)=2 & \iff & \rank\; \overline{A} =1.
\end{array}
\right.
$$
The corollary follows directly from this.
\end{proof}

We define $G:=\Gal(L(\kcyc)/k)$ and compute $\mathcal{K}(G)$.
Note that the equivalence between (i) and (ii) in the following lemma is used in \S \ref{proof in case non-Gal deg 4}, 
but not in \S \ref{another proof in case non-Gal deg 4}.

\begin{lem}\label{mathcal{K} in case non-Gal deg 4}
Set $T:=\{ \p, \overline{\p} \}$ and let $M_p^T(k)$ be the maximal abelian $p$-extension of $k$ that is unramified outside $p$ and completely decomposes at every prime ideal in $T$.
Also set $G:=\Gal(L(\kcyc)/k)$ and let $\mathcal{K}(G)$ be the module defined in Proposition {\rm \ref{central}}.
Then the following statements are equivalent:
\\
{\rm (i)} $M_p^T(k)=k$.
\ \ 
{\rm (ii)} $\mathcal{K}(G)=0$, and therefore $X(\tilde{k})=0$.
\ \ 
{\rm (iii)} $a+b \not\equiv 0 \pmod p$.
\end{lem}

\begin{proof}
Since $\langle D_\p, D_{\overline{\p}}\rangle=\langle \gam,x,y^{a+b} \rangle$ by Lemma \ref{Z in case non-Gal deg 4}, 
conditions (i) and (iii) are equivalent.
Let $F$ be a free pro-$p$ group generated by $\gam$, $x$, and $y$, 
and let $R$ be the closed normal subgroup of $F$ generated by $[\gam,x]$, $[\gam,y]$, $[x,y]$, and their conjugates.
Then, by Lemma \ref{Z in case non-Gal deg 4}, we have a minimal presentation
$
1 \to R \to F \to \Gal(L(\kcyc)/k) \to 1
$
of $\Gal(L(\kcyc)/k)$, 
and 
$$
H_2(\Gal(L(\kcyc)/k),\Zp) \isom R \cap [F,F]/[R,F]=\langle [\gam,x], [\gam,y], [x,y] \rangle [R,F]/[R,F]
$$
by (\ref{homology isom commutator}).
Similarly, we see that each image of $H_2(D_\p,\Zp)$, $H_2(D_\q,\Zp)$, $H_2(D_{\overline{\p}},\Zp)$, and $H_2(D_{\overline{\q}},\Zp)$ under the map (\ref{judge}) is generated by
$$
[\gam,x],
\ \ 
[\gam x^ay^b, y],
\ \ 
[\gam x^{a-b}y^{a+b},x^{-1}],
\ \ 
[\gam x^{-b}y^a, y^{-1}],
$$
respectively.
Modulo $[R, F]$, 
we have
\begin{eqnarray*}
&&
[\gam x^ay^b, y] \equiv [\gam,y][x^a,y][y^b,y] \equiv [\gam,y][x,y]^a,
\\ 
&&
[\gam x^{a-b}y^{a+b},x^{-1}] \equiv [\gam,x]^{-1}[x,y]^{a+b},
\\
&&[\gam x^{-b}y^a, y^{-1}] \equiv [\gam,y]^{-1}[x,y]^b.
\end{eqnarray*}
Fixing the basis $\{ [\gam,x], [\gam,y], [x,y]\}$ of the $\Zp$-module $R \cap [F,F]/[R,F]$, 
we obtain a presentation of 
$\mathcal{K}(G)$ with the following representation matrix: 
$$
K:=
\begin{bmatrix}
1 & 0 & 0
\\
0 & 1 & a
\\
-1 & 0 & a+b
\\
0 & -1 & b
\end{bmatrix}.
$$
Hence, $\mathcal{K}(G)=0$ if and only if 
the matrix $K \bmod p$ has rank $3$, which is equivalent to $a + b \not\equiv 0 \pmod p$.
\end{proof}

\subsubsection{Proof of the non-Galois or cyclic case}
\label{proof in case non-Gal deg 4}

To prove the main theorem in this case, it suffices to show the following proposition.

\begin{prop}\label{main thm in case non-Gal deg 4}
Let $k$ be either a non-Galois CM-field or a totally imaginary cyclic extension of degree $4$.
Suppose that $k$ satisfies conditions {\rm (I)}, {\rm (II)}, and {\rm (III)} given in {\rm \S\ref{Narrowing-down target}}.
Then $X(\tilde{k})$ is trivial if and only if $a+b \not\equiv 0 \pmod p$.
\end{prop}

We now explain how the main theorem in this case follows from Proposition \ref{main thm in case non-Gal deg 4}.
By Lemma \ref{mathcal{K} in case non-Gal deg 4}, when $\dim_{\Fp}A(k)/pA(k) \le 1$, 
the proposition is simply a restatement of the main theorem in different terms.
Furthermore, if $A(k)$ is cyclic, then $a+b \not\equiv 0 \pmod p$ can be easily verified from Corollary \ref{A(k) in case non-Gal deg 4}.
Therefore, the proposition implies that, 
when $A(k)$ is cyclic, 
$X(\tilde{k})$ is trivial if and only if $k$ satisfies conditions {\rm (I)}, {\rm (II)}, and {\rm (III)} given in {\rm \S\ref{Narrowing-down target}}.
Similarly, when $\dim_{\Fp}A(k)/pA(k)=2$, we have $a+b \equiv 0 \pmod p$, and thus $X(\tilde{k})$ is nontrivial.

To prove Proposition \ref{main thm in case non-Gal deg 4}, we need the following lemma.

\begin{lem}\label{ki in case non-Gal deg 4}
Define 
$\ki$ by $\ki \leftrightarrow \langle \gam x^a,xy^{-1} \rangle$,
where the notation $\leftrightarrow$ is as in Definition {\rm \ref{leftrightarrow}}.
Then $\ki/k$ is 
an anticyclotomic $\Zp$-extension 
in which $p$ is non-split and ramified.
\end{lem}

\begin{proof}
Note that 
$$
J(\gam x^a)
=
\gam x^{a-b} y^{a+b} \cdot x^{-a}
=
\gam x^a \cdot x^{-a-b}y^{a+b} \in \langle \gam x^a, xy^{-1} \rangle.
$$
Since 
$\Gal(\ki/k)
=\langle \gam,x,y \rangle/\langle \gam x^a,xy^{-1} \rangle 
=\langle x, \gam x^a,xy^{-1} \rangle/\langle \gam x^a,xy^{-1} \rangle\isom \Zp$, 
it follows that $\ki/k$ is a $\Zp$-extension and that $J$ acts on $\Gal(\ki/k)$ as inverse,
i.e., 
$\ki/k$ is an anticyclotomic $\Zp$-extension.
On the other hand, by Lemma \ref{Z in case non-Gal deg 4} and Remark \ref{a neq 0, b neq 0, a+b neq 0}(ii), we have  
$$
\begin{cases}
\langle \gam x^a,xy^{-1}, I_\p \rangle
=
\langle \gam x^a,xy^{-1}, I_{\overline{\p}} \rangle
=
\langle \gam, x^a, xy^{-1} \rangle
\neq
\langle \gam x^a, xy^{-1} \rangle,
\\
\langle \gam x^a,xy^{-1}, I_\q \rangle
=
\langle \gam x^a,xy^{-1}, I_{\overline{\q}} \rangle
=
\langle \gam x^a, y^b, xy^{-1} \rangle
\neq
\langle \gam x^a, xy^{-1} \rangle.
\end{cases}
$$
Similarly, all four of the groups
$\langle \gam x^a,xy^{-1}, D_\p \rangle$,
$\langle \gam x^a,xy^{-1}, D_{\overline{\p}} \rangle$,
$\langle \gam x^a,xy^{-1}, D_\q \rangle$,
and
$\langle \gam x^a,xy^{-1}, D_{\overline{\q}} \rangle$
coincide with $\langle \gam, x,y \rangle$.
These observations imply that every prime ideal lying above $p$ is non-split and ramified in $\ki/k$.
\end{proof}

We now prove Proposition \ref{main thm in case non-Gal deg 4}.
If $a+b \not\equiv 0 \pmod p$, then $X(\tilde{k})$ is trivial by Lemma \ref{mathcal{K} in case non-Gal deg 4} and Proposition \ref{central}
(for an alternative proof, see \S\ref{another proof in case non-Gal deg 4}).
To prove the converse,
suppose that $a+b \equiv 0 \pmod p$.
Set
$n:=\ord_p(a+b)$.
Note that $n <\infty$ by Remark \ref{a neq 0, b neq 0, a+b neq 0}(ii).
%
Define $P_\infty$, $Q_\infty$, $P_n$, and $Q_n$ by
$$
\begin{cases}
P_\infty \leftrightarrow D_\p=\langle \gam,x \rangle,
&
Q_\infty \leftrightarrow D_\q=\langle \gam x^a,y \rangle,
\\
P_n \leftrightarrow \langle D_\p, D_{\overline{\p}} \rangle=\langle \gam, x, y^{a+b} \rangle,
&
Q_n \leftrightarrow \langle D_\q, D_{\overline{\q}} \rangle=\langle \gam x^a,x^{a+b},y \rangle.
\end{cases}
$$
Then it follows that both $P_\infty$ and $Q_\infty$ are $\Zp$-extensions of $k$, 
and that each of $P_n$ and $Q_n$ is contained in $P_\infty$ and $Q_\infty$, respectively.
Let $\ki$ be the field defined in Lemma \ref{ki in case non-Gal deg 4}.
Then we have $\ki \subset P_\infty Q_\infty$, since 
$P_\infty Q_\infty \leftrightarrow \langle \gam x^a \rangle$.
Moreover, the extension $P_\infty Q_\infty/\ki$ is a $\Zp$-extension and, in particular, is cyclic
(see Figure \ref{X neq 0 galois diagram in case non-Gal deg 4}).
\begin{figure}[H]
$$ 
\begin{xy}
(0,0) *{\text{$k$}}="k",
(-15,10) *{\text{$P_n$}}="P",
(-30,20) *{\text{$P_\infty$}}="Pinf",
(15,10) *{\text{$Q_n$}}="Q",
(30,20) *{\text{$Q_\infty$}}="Qinf",
(0,20) *{\text{$k_\infty$}}="kinf",
(0,40) *{\text{$P_\infty Q_\infty$}}="PinfQinf",
(0,28) *{\text{$L'$}}="L",
(75,0) *{\text{$\langle \gam,x,y \rangle$}}="gk",
(60,10) *{\text{$\langle \gam,x,y^{a+b} \rangle$}}="gP",
(45,20) *{\text{$\langle \gam,x \rangle$}}="gPinf",
(90,10) *{\text{$\langle \gam x^a,x^{a+b}, y \rangle$}}="gQ",
(105,20) *{\text{$\langle \gam x^a,y \rangle$}}="gQinf",
(75,20) *{\text{$\langle \gam x^a,xy^{-1} \rangle$}}="gkinf",
(75,40) *{\text{$\langle \gam x^a \rangle$}}="gPinfQinf",
(75,28) *{\text{$\langle \gam x^a,(xy^{-1})^{a+b} \rangle$}}="gL",
\ar@{-} "k";"P"^{\text{$\p, \overline{\p}$: split}}
\ar@{-} "P";"Pinf"
\ar@{-} "k";"Q"_{\text{$\q, \overline{\q}$: split}}
\ar@{-} "Q";"Qinf"
\ar@{-} "k";"kinf"
\ar@{-} "kinf";"L"^{\text{$p$: split}}
\ar@{-} "L";"PinfQinf"
\ar@{-} "Pinf";"PinfQinf"
\ar@{-} "Qinf";"PinfQinf"
\ar@{-} "gk";"gP"
\ar@{-} "gP";"gPinf"
\ar@{-} "gk";"gQ"
\ar@{-} "gQ";"gQinf"
\ar@{-} "gk";"gkinf"
\ar@{-} "gkinf";"gL"
\ar@{-} "gL";"gPinfQinf"
\ar@{-} "gPinf";"gPinfQinf"
\ar@{-} "gQinf";"gPinfQinf"
\end{xy}
$$
\caption{}
\label{X neq 0 galois diagram in case non-Gal deg 4}
\end{figure}
Note that the prime $p$ does not split in the extension $\ki/k$ by Lemma \ref{ki in case non-Gal deg 4}.
Since the prime ideals $\p$, $\overline{\p}$ and $\q$, $\overline{\q}$ split completely in the extensions $P_n/k$ and $Q_n/k$, respectively,
there exists a cyclic extension $L'$ of $\ki$ of degree $p^n$ such that every prime lying above $p$ splits completely in $L'/\ki$.
Moreover, the extension $L'/\ki$ is unramified, again by Lemma \ref{ki in case non-Gal deg 4}.
Therefore, there exists a natural surjective homomorphism
$$
X'(\ki) \surj \Gal(L'/\ki) =\langle \gam x^a, xy^{-1} \rangle/\langle \gam x^a, (xy^{-1})^{a+b} \rangle.
$$
This implies that $X'(\ki)$ is nontrivial, since $a+b \equiv 0 \pmod p$.
Hence, by Theorem \ref{main prop}, $X(\tilde{k})$ is also nontrivial.
Therefore, we conclude that if $X(\tilde{k})$ is trivial, then $a+b \not\equiv 0 \pmod p$.

\subsubsection{An alternative proof of the sufficient condition}
\label{another proof in case non-Gal deg 4}

We present an alternative proof of the sufficient condition stated in Proposition \ref{main thm in case non-Gal deg 4}.
The approach is based on the ideas of Minardi \cite{Minardi}, Itoh \cite{Itoh2011}, and Fujii \cite{Fujii2015}.
In other words, we show that if $a+b \not\equiv 0 \pmod p$, then $X(\tilde{k})$ is trivial, without relying on Proposition \ref{central}.

Suppose that $a+b \not\equiv 0 \pmod p$ (i.e., $M_p^{ \{ \p, {\overline{\p}} \}}(k)=k$).
Then, by Corollary \ref{A(k) in case non-Gal deg 4}, we have $\dim_{\Fp}A(k)/pA(k) \le 1$.
Define the fields $k^{(1)}$, $k^{(2)}$, $N_\p$, 
and their conjugates $\overline{k^{(1)}}$, $\overline{k^{(2)}}$, and $N_{\overline{\p}}$ 
by
$$
\left\{
\begin{array}{lll}
k^{(1)} \leftrightarrow \langle I_\p, I_\q \rangle=\langle \gam, x^ay^b \rangle,
&
k^{(2)} \leftrightarrow I_\p=\langle \gam \rangle,
&
N_\p \leftrightarrow D_\p=\langle \gam, x \rangle,
\\
\overline{k^{(1)}} \leftrightarrow 
\langle I_{\overline{\p}}, I_{\overline{\q}} \rangle
=\langle \gam x^{a-b}y^{a+b}, x^{a}y^{b} \rangle,
&
\overline{k^{(2)}} \leftrightarrow 
I_{\overline{\p}}
=\langle \gam x^{a-b}y^{a+b} \rangle,
&
N_{\overline{\p}} \leftrightarrow 
D_{\overline{\p}}
=\langle \gam y^{a+b}, x \rangle.
\end{array}
\right.
$$
We may assume, without loss of generality, that $a \not\equiv 0 \pmod p$ in the following argument.
Otherwise 
(which occurs when $A(k)=0$, and consequently $b \not\equiv 0 \pmod p$), 
by replacing the role of $\q$ with $\overline{\q}$, 
for example, 
$k^{(1)} \leftrightarrow \langle I_\p, I_{\overline{\q}} \rangle
=\langle \gam , x^{-b}y^a \rangle$,
we arrive at the same conclusion.

\begin{lem}\label{X^1=0 in case non-Gal deg 4}
The module $X(k^{(1)})$ is trivial.
\end{lem}

\begin{proof}
Note that when $A(k)=0$, the claim follows directly from \cite[Corollary 3.3]{Itoh2011}.
However, this fact is not necessary for our argument.
Let $L_0$ denote the maximal subfield of $L(k^{(1)})$ that is abelian over $k$.
Then, since $\Gal(k^{(1)}/k) \isom \Zp$, the group $\Gal(L_0/k^{(1)})$ is isomorphic to $X(k^{(1)})_{\Gal(k^{(1)}/k)}$. 
Since $\tilde{k}$ is the maximal abelian $p$-extension of $k$ unramified outside $p$, it follows that $k^{(2)}$ is the maximal abelian $p$-extension of $k$ unramified outside $\{ \q,\overline{\p}, \overline{\q} \}$.
Therefore, since $L_0/k$ is an abelian $p$-extension that is unramified outside $\{ \q,\overline{\p}, \overline{\q} \}$, we have $L_0 \subset k^{(2)} \cap L(k^{(1)})$.
On the other hand, by the definition of $k^{(2)}$, all primes of $k^{(1)}$ lying above $\q$ are totally ramified in the extension $k^{(2)} \cap L(k^{(1)})/k^{(1)}$.
Therefore, $k^{(2)} \cap L(k^{(1)})=k^{(1)}$, which implies that $L_0=k^{(1)}$.
It follows that $X(k^{(1)})_{\Gal(k^{(1)}/k)}$ is trivial, and hence $X(k^{(1)})$ is also trivial by Nakayama's lemma. 
\end{proof}

\begin{lem}\label{X^2=0 in case non-Gal deg 4}
The module $X(k^{(2)})$ is trivial.
\end{lem}

\begin{proof}
From the assumption that $a \not\equiv 0 \pmod p$, we obtain 
$\langle I_\p, D_\q \rangle=\langle \gam, x^a,y \rangle=\langle \gam,x,y \rangle$,
which implies that $\q$ is totally inert in the extension $k^{(1)}/k$.
Hence, there exists a unique prime of $k^{(1)}$ lying above $\q$.
Therefore, in the extension $k^{(2)}/k^{(1)}$, exactly one prime is ramified, and this prime is totally ramified.
Combining this with Lemma \ref{X^1=0 in case non-Gal deg 4},
we conclude that $X(k^{(2)})$ is trivial by Iwasawa's criterion \cite{Iwasawa1956}.
\end{proof}

Let $L_2$ denote the maximal subfield of $L(\tilde{k})$ that is abelian over $k^{(2)}$.
Then, since $\Gal(\tilde{k}/k^{(2)}) \isom \Zp$, the group $\Gal(L_2/k^{(2)})$ is isomorphic to $X(\tilde{k})_{\Gal(\tilde{k}/k^{(2)})}$.

\begin{lem}\label{L_2/N_p is abelian in case non-Gal deg 4}
With the notation above, the extension $L_2/N_\p$ is abelian.
\end{lem}

\begin{proof}
Our proof closely follows the arguments in \cite{Minardi}, \cite{Itoh2011}, and \cite{Fujii2015}.
For a prime $\P^{(2)}$ of $k^{(2)}$ lying above $\p$, let $\mathbb{I}_{\P^{(2)}}$ denote its inertia subgroup in $\Gal(L_2/k^{(2)})$.
By the definitions of $k^{(2)}$ and $N_\p$, any prime of $N_\p$ lying above $\p$ is totally inert in the extension $k^{(2)}/N_\p$.
Therefore, for each prime $\P^{(2)}$ of $k^{(2)}$ lying above $\p$, the inertia subgroup $\mathbb{I}_{\P^{(2)}}$ naturally has the structure of a $\Zp \llbracket \Gal(k^{(2)}/N_\p) \rrbracket$-module.
Since $X(k^{(2)})$ is trivial by Lemma \ref{X^2=0 in case non-Gal deg 4}, and since $L_2/k^{(2)}$ is unramified outside $\p$, we obtain the isomorphism
\begin{eqnarray}\label{MIF in case non-Gal deg 4}
\Gal(L_2/k^{(2)}) = \sum_{\P^{(2)} \mid \p} \mathbb{I}_{\P^{(2)}}
\end{eqnarray}
as a module over $\Zp \llbracket \Gal(k^{(2)}/N_\p) \rrbracket$.
We now claim that $\Gal(k^{(2)}/N_\p)$ acts trivially on $\mathbb{I}_{\P^{(2)}}$.
Since $L_2/\tilde{k}$ is unramified, it follows that 
$\mathbb{I}_{\P^{(2)}} \cap \Gal(L_2/\tilde{k})=1$.
This implies that the restriction map 
$\mathbb{I}_{\P^{(2)}} \to \Gal(\tilde{k}/k^{(2)})$ 
is injective.
Because $\tilde{k}/N_\p$ is abelian, 
the action of $\Gal(k^{(2)}/N_\p)$ on $\Gal(\tilde{k}/k^{(2)})$ is trivial by Lemma \ref{G/H act on H},
and consequently, its action on $\mathbb{I}_{\P^{(2)}}$ is also trivial.
From (\ref{MIF in case non-Gal deg 4}), it follows that $\Gal(k^{(2)}/N_\p)$ acts trivially on $\Gal(L_2/k^{(2)})$.
Hence, applying Lemma \ref{G/H act on H} once more, we conclude that $L_2/N_\p$ is abelian.
\end{proof}

By interchanging the roles of $\p$ and $\q$ with $\overline{\p}$ and $\overline{\q}$, respectively, in the above argument, 
we similarly deduce that $\overline{L_2}/N_{\overline{\p}}$ is abelian, where $\overline{L_2}$ denotes the maximal subfield of $L(\tilde{k})$ that is abelian over $\overline{k^{(2)}}$.
To prove that $X(\tilde{k})$ is trivial, it suffices to show that $X(\tilde{k})_{\Gal(\tilde{k}/k)}$ is trivial.
Let $\mathcal{C}_{\tilde{k}/k}$ denote the subfield of $L(\tilde{k})$ containing $\tilde{k}$ such that
$
\Gal(\mathcal{C}_{\tilde{k}/k}/\tilde{k}) \isom X(\tilde{k})_{\Gal(\tilde{k}/k)}.
$
Note that $k=N_\p \cap N_{\overline{\p}}$, since the subgroup corresponding to 
$N_\p \cap N_{\overline{\p}} \leftrightarrow 
\langle D_{\p}, D_{\overline{\p}} \rangle
=
\langle \gam,x,y^{a+b} \rangle,
$
and
$a+b \not\equiv 0 \pmod p$.
Combining this with the correspondence $N_\p N_{\overline{\p}} \leftrightarrow \langle x \rangle$, we obtain the inclusion
\begin{eqnarray}\label{k subset kcyc subset NN in case non-Gal deg 4}
k=N_\p \cap N_{\overline{\p}} \subset \kcyc \subset N_\p N_{\overline{\p}}
\end{eqnarray}
as illustrated in Figure~\ref{X = 0 galois diagram in case non-Gal deg 4}.

\begin{figure}[H]
$$ 
\begin{xy}
(0,12) *{\text{$k_{{\rm cyc}}$}}="kcyc",
(0,0) *{\text{$k$}}="k",
(13,0) *{\text{$\leftrightarrow  \langle D_{\p},  D_{\overline{\p}} \rangle$}}="Z",
(-20,12) *{\text{$N_{\p}$}}="Np",
(-20,24) *{\text{$k^{(2)}$}}="k2",
(0,36) *{\text{$\tilde{k}$}}="tildek",
(0,45) *{\text{$\mathcal{C}_{\tilde{k}/k}$}}="central/k",
(-15,50) *{\text{$L_2$}}="L2",
(20,12) *{\text{$N_{\overline{\p}}$}}="Np-conj",
(20,24) *{\text{$\overline{k^{(2)}}$}}="k2-conj",
(0,24) *{\text{$N_{\p}N_{\overline{\p}}$}}="NpNp-conj",
%
(15,50) *{\text{$\overline{L_2}$}}="L2-conj",
\ar@{-} "k";"Np",
\ar@{-} "k2";"tildek",
\ar@{-} "k2";"Np",
\ar@{-} "tildek";"central/k",
\ar@{-} "central/k";"L2",
%
\ar@{-} "k";"Np",
\ar@{-} "k2";"tildek",
\ar@{-} "k2";"Np",
\ar@{-} "tildek";"central/k",
\ar@{-} "central/k";"L2",
\ar@{-} "k";"Np-conj",
\ar@{-} "k2-conj";"tildek",
\ar@{-} "k2-conj";"Np-conj",
\ar@{-} "Np";"NpNp-conj",
\ar@{-} "Np-conj";"NpNp-conj",
\ar@{-} "central/k";"L2-conj",
\ar@{-} "NpNp-conj";"tildek",
\ar@{-} "k";"kcyc",
\ar@{-} "kcyc";"NpNp-conj",
\ar@/^10mm/ @{.} "Np";"L2"^{\rm abel}
\ar@/_10mm/ @{.} "Np-conj";"L2-conj"_{\rm abel}
\end{xy}
$$
\caption{}
\label{X = 0 galois diagram in case non-Gal deg 4}
\end{figure}

On the other hand, by considering the natural surjections
$\Gal(L_2/\tilde{k}) \isom X(\tilde{k})_{\Gal(\tilde{k}/k^{(2)})} \surj X(\tilde{k})_{\Gal(\tilde{k}/k)}$
and
$\Gal(\overline{L_2}/\tilde{k}) \isom X(\tilde{k})_{\Gal(\tilde{k}/\overline{k^{(2)}})} \surj X(\tilde{k})_{\Gal(\tilde{k}/k)}$,
we obtain the inclusions 
$$
\tilde{k} \subset \mathcal{C}_{\tilde{k}/k} \subset L_2 \cap \overline{L_2}.
$$
Since $\tilde{k}/\kcyc$ is unramified and $\tilde{k} \subset \mathcal{C}_{\tilde{k}/k} \subset L(\tilde{k})$, 
it follows that the (possibly non-abelian) $p$-extension $\mathcal{C}_{\tilde{k}/k}/\kcyc$ is also unramified.
We now claim that the extension $\mathcal{C}_{\tilde{k}/k}/\kcyc$ is abelian.
Indeed, since $\mathcal{C}_{\tilde{k}/k}/N_\p$ 
is abelian by Lemma \ref{L_2/N_p is abelian in case non-Gal deg 4}, 
the group $\Gal(N_\p N_{\overline{\p}}/N_\p) \isom \Zp$ acts trivially on $\Gal(\mathcal{C}_{\tilde{k}/k}/N_\p N_{\overline{\p}})$ by Lemma \ref{G/H act on H}.
Similarly, $\Gal(N_\p N_{\overline{\p}}/N_{\overline{\p}})$ also acts trivially on $\Gal(\mathcal{C}_{\tilde{k}/k}/N_\p N_{\overline{\p}})$.
Since every element of $\Gal(N_\p N_{\overline{\p}}/\kcyc)$ can be expressed as a product of elements of $\Gal(N_\p N_{\overline{\p}}/N_\p)$ and $\Gal(N_\p N_{\overline{\p}}/N_{\overline{\p}})$ by (\ref{k subset kcyc subset NN in case non-Gal deg 4}), 
it follows that the group $\Gal(N_\p N_{\overline{\p}}/\kcyc) \isom \Zp$ also acts trivially on $\Gal(\mathcal{C}_{\tilde{k}/k}/N_\p N_{\overline{\p}})$.
Therefore, applying Lemma \ref{G/H act on H} once again, we conclude that the extension $\mathcal{C}_{\tilde{k}/k}/\kcyc$ is abelian (and unramified).
Since $\tilde{k}=L(\kcyc)$, it follows that $\mathcal{C}_{\tilde{k}/k}=\tilde{k}$.
By the definition of $\mathcal{C}_{\tilde{k}/k}$, we deduce that $X(\tilde{k})_{\Gal(\tilde{k}/k)}$ is trivial, and hence, by Nakayama's lemma, $X(\tilde{k})$ itself is also trivial.
This completes the proof.


\subsection{The case of totally imaginary $(2,2)$-extensions}\label{case: (2,2) of degree 4}

In this subsection, let $k$ be a totally imaginary $(2,2)$-extension of $\Q$.
Suppose that $k$ satisfies conditions (I), (II), and (III) given in \S\ref{Narrowing-down target}.
Choose two generators $\sigma$ and $\tau$ of $\Gal(k/\Q)$ such that 
$J:=\sigma \tau$ is the complex conjugation.
Then the element $\sigma \tau$ acts on $X(\kcyc)=X(\kcyc)^-$ as inverse.
Fix a prime ideal $\p$ of $k$ lying above $p$, and define the other prime ideals $\overline{\p}$, $\q$, and $\overline{\q}$ of $k$ lying above $p$ by 
\begin{eqnarray}\label{prime decomp in case (2,2)}
\overline{\p}:=\sigma\tau(\p),
\ \ 
\q:=\sigma(\p),
\ \ 
\overline{\q}:=\tau(\p).
\end{eqnarray}
Here, the notation $\overline{\; \cdot\; }$ denotes the complex conjugate.
We denote by $D_{\p}$, $D_\q$, $D_{\overline{\p}}$, and $D_{\overline{\q}}$ the decomposition groups of 
$\p$, $\q$, $\overline{\p}$, and $\overline{\q}$ in $\Gal(\tilde{k}/k)$, respectively.
Similarly, we denote by $I_{\p}$, $I_\q$, $I_{\overline{\p}}$, and $I_{\overline{\q}}$ the inertia groups of $\p$, $\q$, $\overline{\p}$, and $\overline{\q}$ in $\Gal(\tilde{k}/k)$, respectively.
We now obtain a result similar to Lemma \ref{Z in case non-Gal deg 4}:

\begin{lem}\label{Z in case (2,2)}
Let $\gam \in \Gal(\tilde{k}/k)$ be a topological generator of $I_{\p}$.
Then there exist elements $x, y \in X(\kcyc)$ and $a,b \in \Zp$ such that 
$$
\begin{cases}
\Gal(\tilde{k}/k)=\langle \gam, x,y \rangle,
&
X(\kcyc)=\langle x,y \rangle,
\\
D_\p=\langle \gam \rangle \x \langle x \rangle,
&
D_\q=\langle \gam x^ay^{-a} \rangle \x \langle y \rangle,
\\
D_{\overline \p}=\langle \gam x^{a+b}y^{b-a} \rangle \x \langle x^{-1} \rangle,
&
D_{\overline \q}=\langle \gam x^b y^b \rangle \x \langle y^{-1} \rangle.
\end{cases}
$$
Here, the left factor of each direct product is the inertia group corresponding to the respective prime ideal.
\end{lem}

\begin{proof}
The proof proceeds in a manner similar to that of Lemma \ref{Z in case non-Gal deg 4}.
By Lemma \ref{A(k)=D(k)}, every prime ideal lying above $p$ is inert in the extension $\tilde{k}/\kcyc$ and
$
\Gal(\tilde{k}/k)=\langle D_\p, D_\q,D_{\overline{\p}}, D_{\overline{\q}} \rangle.
$
We take an element $x \in X(\kcyc)$ such that 
$D_{\p}=\langle \gam \rangle \x \langle x \rangle$, 
and define $y:=\sigma(x)$.
Note that $\tau(x)=y^{-1}$ since $\sigma(\tau(x))=x^{-1}$.
Applying the automorphisms $\sigma$, $\sigma \tau$, and $\tau$ to both $I_\p$ and $D_\p$, we obtain 
\begin{eqnarray}\label{decomp group in proof in the case (2,2)}
D_\q=\langle \sigma(\gam) \rangle \x \langle y \rangle,
\ \ 
D_{\overline \p}=\langle \sigma \tau(\gam) \rangle \x \langle x^{-1} \rangle,
\ \ 
D_{\overline \q}=\langle \tau(\gam) \rangle \x \langle y^{-1} \rangle
\end{eqnarray}
by (\ref{prime decomp in case (2,2)}).
Here, the left factor of each direct product is the inertia group corresponding to the respective prime ideal.
Similarly to the proof of Lemma \ref{Z in case non-Gal deg 4}, we have 
$$
X(\kcyc)=
\left\langle
D_{\p} \cap X(\kcyc),\; D_{\q} \cap X(\kcyc),\; D_{\overline \p} \cap X(\kcyc),\; D_{\overline \q} \cap X(\kcyc)
\right\rangle
=
\langle
x,y
\rangle,
$$
%
and 
$\sigma(\gam) \equiv \gam \pmod {X(\kcyc)}$.
Hence, there exist $a,b,c,d \in \Zp$ such that $\sigma(\gam)=\gam x^a y^c$ and $\tau(\gam)=\gam x^b y^d$.
Since both $\sigma^2$ and $\tau^2$ are the identity automorphisms, we have
$
\gam=\sigma^2(\gam)=\gam x^{a+c} y^{a+c}
$ 
and
$
\gam=\tau^2(\gam)=\gam x^{b-d} y^{d-b}
$,
which imply that 
$c=-a$ and $d=b$.
Therefore,
$$
\sigma(\gam)=\gam x^a y^{-a}, 
\ \ 
\tau(\gam)=\gam x^b y^b.
$$
Combining these expressions with (\ref{decomp group in proof in the case (2,2)}), we obtain the representations of the decomposition groups.
Finally, 
we conclude that
$\Gal(\tilde{k}/k)=\langle \gam, x,y \rangle$.
\end{proof}

\begin{rem}\label{a neq 0, b neq 0, a neq b}
{\rm
For the same reason as in Remark \ref{a neq 0, b neq 0, a+b neq 0}, we have
$
a \neq 0,\ b \neq 0,\ a-b \neq 0.
$
}
\end{rem}

In a manner similar to \S \ref{case: non-Gal deg 4},
we can prove the main theorem for this case.
Table \ref{table deg4} lists the corresponding objects (all congruences are taken modulo $p$).

\begin{longtable}{|c|c|c|}
\hline
     & \S \ref{case: non-Gal deg 4} & \S \ref{case: (2,2) of degree 4} 
\\
\hline
\hline
    $A(k)=0$ \hfill $\iff$ & $a^2+b^2 \not\equiv 0$ & $a \not\equiv 0$, $b \not\equiv 0$ 
\\
    $A(k)$: cyclic \hfill $\iff$ & $a,b \not\equiv 0$, $a^2+b^2 \equiv 0$ & $ab \equiv 0$ and $a-b \not\equiv 0$
\\
    $\dim_{\Fp}A(k)/pA(k)=2$ \hfill $\iff$ & $a \equiv b \equiv 0$ & $a \equiv b \equiv 0$
\\
\hline
$\mathcal{K}(G)=0$ \hfill $\iff$ & $n:=\ord_p(a+b)$ is $0$ & $n:=\ord_p(a-b)$ is $0$
\\
\hline
$\ki$  & $\langle \gam x^a, xy^{-1} \rangle$ & $\langle \gam x^a, xy \rangle$
\\
\hline
$P_\infty$ & $D_\p=\langle \gam, x \rangle$ & $D_\p=\langle \gam, x \rangle$ 
\\
$Q_\infty$  & $D_\q=\langle \gam x^a, y \rangle$ & $D_\q=\langle \gam x^a, y \rangle$
\\
$P_n$ & $\langle D_\p, D_{\overline{\p}} \rangle =\langle \gam, x, y^{a+b} \rangle$ & $\langle D_\p, D_{\overline{\p}} \rangle =\langle \gam, x, y^{a-b} \rangle$
\\
$Q_n$ & $\langle D_\q, D_{\overline{\q}} \rangle =\langle \gam, x^{a+b}, y \rangle$ & $\langle D_\q, D_{\overline{\q}} \rangle =\langle \gam, x^{a-b}, y \rangle$
\\
\hline
$L'$ & $\langle \gam x^a, (xy^{-1})^{a+b} \rangle$ & $\langle \gam x^a, (xy)^{a-b} \rangle$
\\
\hline
$k^{(1)}$ & $\langle I_\p, I_\q \rangle =\langle \gam, x^ay^b \rangle$ & $\langle I_\p, I_\q \rangle =\langle \gam, x^ay^{-a} \rangle$
\\
$k^{(2)}$ and $N_\p$ & $I_\p=\langle \gam \rangle$ and $D_\p=\langle \gam, x \rangle$ & $I_\p=\langle \gam \rangle$ and $D_\p=\langle \gam, x \rangle$
\\
\hline
The goal: $X(\tilde{k})=0$ \hfill $\iff$  & $a+b \not\equiv 0$ & $a-b \not\equiv 0$
\\
\hline
\caption{Summary of relations in \S \ref{case: non-Gal deg 4} and \S \ref{case: (2,2) of degree 4} .}
\label{table deg4}
\end{longtable}


\section{The case of degree $6$}
\label{section deg 6}

\subsection{Setting and preparations}\label{case: non-Gal deg 6}

Suppose that $k$ is a CM-field of degree $6$
satisfying conditions (I), (II), and (III) given in \S\ref{Narrowing-down target}.
Let $F/\Q$ denote the Galois closure of $k/\Q$, 
and let $\Delta:=\Gal(F/\Q)$ be its Galois group.
The following classification of the group structure of $\Delta$ is known (see, for example, Dodson \cite[Theorem in \S 5.1.2]{Dodson1984} or Bouw et al. \cite[Proposition 2.1]{Bouw et al.2015}).

\begin{lem}\label{Gal of deg 6}
Let $k$ denote a CM-field of degree $6$, and let $\Delta$ denote the Galois group of the Galois closure of $k/\Q$.
Then $\Delta$ is isomorphic to one of the following groups:
\begin{itemize}
\setlength{\parskip}{0pt} 
\setlength{\itemsep}{0pt} 
\item[{\rm (i)}]
$\Z/6\Z =\{ \sigma \,|\, \sigma^6=1 \}$,
\item[{\rm (ii)}]
$
D_{12}
=\{ \sigma,\tau \,|\, \sigma^6=\tau^2=1,\ \tau \sigma \tau^{-1}=\sigma^{-1} \}
$,
\item[{\rm (iii)}]
$\displaystyle{
(\Z/2\Z)^3 \rtimes \Z/3\Z
=\left\{ 
a,b,c,x \,\Big|\, 
\begin{array}{l}
a^2=b^2=c^2=x^3=1,\ ab=ba,\ bc=cb,\ ca=ac,
\\
ax=xc,\ bx=xa,\ cx=xb
\end{array}
\right\}
}$,
\item[{\rm (iv)}]
$\displaystyle{
(\Z/2\Z)^3 \rtimes \mathcal{S}_3
=
\left\{ 
a,b,c,x,y \,\Big|\, 
\begin{array}{l}
a^2=b^2=c^2=x^3=y^2=1,\ ab=ba,\ bc=cb,\ ca=ac,
\\
ax=xc,\ bx=xa,\ cx=xb,\ 
ya = by,\ yb = ay,\ yc = cy,
\\
yx = x^2 y
\end{array}
\right\}
}$.
\end{itemize}
Here, $\mathcal{S}_3$ denotes the symmetric group of degree $3$, 
and $\rtimes$ denotes the semi-direct product of groups.
Thus, $\Z/3\Z$ in {\rm (iii)} and $\mathcal{S}_3$ in {\rm (iv)} act by permutations on the three copies of $\Z/2\Z$.
In particular, if $k/\Q$ is Galois, then the Galois group $\Delta$ is cyclic of order $6$.
\end{lem}

We identify $\Del$ with each group appearing in Lemma \ref{Gal of deg 6}.
In cases (iii) and (iv), set
$$
\sigma:=abcx.
$$
Then, in each of the cases from {\rm (ii)} to {\rm (iv)}, we verify that $\sigma$ has order $6$ and that $\langle \sigma^3 \rangle$ is the center of $\Gal(F/\Q)$.
By \cite[Corollary 1.5]{MilneCM06}, the field $F$ is also a CM-field.
This implies that $F^+/\Q$ is Galois, and hence $\Gal(F/F^+)=\langle \sigma^3 \rangle$,
i.e., $\sigma^3$ is the complex conjugation.
Fix a prime ideal $\P_1$ of $F$ lying above $p$.
Since $p$ splits completely in $k/\Q$, it also splits completely in $F/\Q$.
Therefore, the prime ideals $g(\P_1)$ for $g \in \Delta$ are distinct.
Define 
\begin{eqnarray}\label{def of primes in F in case non-Gal deg 6}
&&
\begin{array}{lll}
&
\overline{\P_2}:=\sigma(\P_1), 
&
\P_3:=\sigma^2(\P_1),
\\[3pt]
\overline{\P_1}:=\sigma^3(\P_1), 
&
\P_2:=\sigma^4(\P_1),
&
\overline{\P_3}:=\sigma^5(\P_1).
\end{array}
\end{eqnarray}
Then, for $i=1,2,3$, each $\overline{\P_i}$ is the complex conjugate of $\P_i$.
Define the prime ideals of $k$ lying above $p$ by 
$$
\p_i:=N_{F/k}\P_i, 
\ \ 
\overline{\p_i}:=N_{F/k} \overline{\P_i}
\ \ 
(i=1,2,3).
$$

\begin{lem}
The six prime ideals $\p_i$ and $\overline{\p_i}$ {\rm ($i=1,2,3$)} in $k$ are distinct.
\end{lem}

\begin{proof}
The proof is similar to that of Lemma \ref{primes are distinct in the case deg 4}.
In the case $\Delta=\Z/6\Z$, there is nothing to prove, 
so we consider the remaining cases.
We prove only that $\p_1$ is distinct from the other prime ideals.
Suppose that $\p_1=\overline{\p_1}$.
Then 
$N_{\Gal(F/k)}(\P_1)= N_{\Gal(F/k)} \circ \sigma^3(\P_1)$.
By the uniqueness of the prime ideal factorization of $N_{\Gal(F/k)}(\P_1)$, it follows that $\sigma^3 \in \Gal(F/k)$, which implies that $k$ is totally real.
This is a contradiction.
Assume that $\p_1$ coincides with a prime ideal other than $\overline{\p_1}$.
Then, similarly, there exists $1 \le j \le 5$ with $j \neq 3$ such that $\sigma^j \in \Gal(F/k)$.
This implies that $\Gal(F/k)$ contains an element $\rho$ of order $3$.
Hence, the fixed field of $\langle \rho \rangle$ contains $k$ and has degree a power of $2$ over $\Q$.
This contradicts the assumptions.
\end{proof}

For each $i=1,2,3$, we denote by $\D{i}$ and $\Do{i}$ the decomposition groups of 
$\p_i$ and $\overline{\p_i}$ in $\Gal(\tilde{k}/k)$, respectively.
Similarly, we denote by $\I{i}$ and $\Io{i}$ the inertia groups of $\p_i$ and $\overline{\p_i}$ in $\Gal(\tilde{k}/k)$, respectively.
We now state a lemma analogous to Lemma \ref{complex embed in case non-Gal deg 4}:

\begin{lem}\label{complex embed in case deg 6}
With the above notation and assumptions, the restriction $\sigma |_k \in \Hom(k,F)$ of $\sigma$ to $k$ acts on $\Gal(\tilde{k}/k)$ and on $X(\kcyc)$ via inner automorphisms.
Moreover,
$$
\begin{array}{ccc}
&
\sigma|_k (\D{1})=\Do{2},
&
(\sigma|_k)^2 (\D{1})=\D{3},
\\[3pt]
(\sigma|_k)^3 (\D{1})=\Do{1},
&
(\sigma|_k)^4 (\D{1})=\D{2},
&
(\sigma|_k)^5 (\D{1})=\Do{3}.
\end{array}
$$
\end{lem}

\begin{proof}
Note that $\Gal(\tilde{F}/F)$ is endowed with the action of $\sigma$.
Let $D_{\P_i}$ and $D_{\overline{\P_i}}$ denote the decomposition groups of $\P_i$ and $\overline{\P_i}$ in $\Gal(\tilde{F}/F)$, respectively, for $i=1,2,3$.
By the definition of these prime ideals (see \ref{def of primes in F in case non-Gal deg 6}), we have 
$$
\begin{array}{ccc}
&
\sigma (D_{\P_1})=D_{\overline{\P_2}},
&
\sigma^2 (D_{\P_1})=D_{\P_3},
\\[3pt]
\sigma^3 (D_{\P_1})=D_{\overline{\P_1}},
&
\sigma^4 (D_{\P_1})=D_{\P_2},
&
\sigma^5 (D_{\P_1})=D_{\overline{\P_3}}.
\end{array}
$$
Since $p$ splits completely in $F/\Q$, all the canonical projections
$\pi_{\P_i} \colon D_{\P_i} \to \D{i}$ and
$\pi_{\overline{\P_i}} \colon D_{\overline{\P_i}} \to \Do{i}$
for $i=1,2,3$ are isomorphisms.
Hence,
$$
\sigma|_k (\D{1})=\Do{2},
$$ 
and similar relations follow.
For example, the map $\sigma|_k \colon \D{1} \to \Do{2}$
can be expressed as 
$\sigma|_k=\pi_{\overline{\P_2}} \circ \sigma \circ \pi_{\P_1}$.
Note that these maps agree on the intersection of their domains.
Since 
\begin{eqnarray}\label{Gal(tilde{k}/k) is generated by these six decomposition groups}
\Gal(\tilde{k}/k)=
\langle \D{1}, \D{2}, \D{3}, \Do{1}, \Do{2}, \Do{3} \rangle
\end{eqnarray}
by the condition $A(k)=D(k)$ in Lemma  \ref{A(k)=D(k)},
it follows that $\sigma$ acts on $\Gal(\tilde{k}/k)$.
Via the canonical isomorphism $\Gal(F_{\rm cyc}/F) \isom \Gal(\kcyc/k)$, 
we see that $\sigma|_k$ acts on $\Gal(\kcyc/k)$
and that this action is trivial.
Therefore, $\sigma|_k$ also acts on $X(\kcyc)=\Ker(\Gal(\tilde{k}/k) \to \Gal(\kcyc/k))$.
This completes the proof.
\end{proof}

We now obtain a result similar to those in Lemmas \ref{Z in case non-Gal deg 4} and \ref{Z in case (2,2)}:

\begin{lem}\label{Z in case deg 6}
Let $\gam \in \Gal(\tilde{k}/k)$ be a topological generator of $\I{1}$.
Then there exist elements $x, y, z \in X(\kcyc)$ and $a,b,c \in \Zp$ such that 
$$
\begin{cases}
\Gal(\tilde{k}/k)=\langle \gam, x,y,z \rangle,
&
X(\kcyc)=\langle x,y,z \rangle,
\\
\D{1}=\langle \gam \rangle \x \langle x \rangle,
&
\Do{1}=\langle \gam x^{a-b-c}y^{a+b-c}z^{a+b+c} \rangle \x \langle x^{-1} \rangle,
\\
\D{2}=\langle \gam x^{-b-c}y^{a-c}z^{a+b} \rangle \x \langle y^{-1} \rangle,
&
\Do{2}=\langle \gam x^{a}y^{b}z^{c} \rangle \x \langle y \rangle,
\\
\D{3}=\langle \gam x^{a-c}y^{a+b}z^{b+c} \rangle \x \langle z \rangle,
&
\Do{3}=\langle \gam x^{-b}y^{-c}z^{a} \rangle \x \langle z^{-1} \rangle,
\end{cases}
$$
Here, the left factor of each direct product is the inertia group corresponding to the respective prime ideal.
\end{lem}

\begin{proof}
The lemma follows from the same argument as in the proof of Lemma \ref{Z in case non-Gal deg 4}.
By Lemma \ref{A(k)=D(k)}, every prime ideal lying above $p$ is inert in the extension $\tilde{k}/\kcyc$ and
$\Gal(\tilde{k}/k)$ is written as (\ref{Gal(tilde{k}/k) is generated by these six decomposition groups}).
Choose $x \in X(\kcyc)$ such that $\D{1}=\langle \gam \rangle \x \langle x \rangle$, 
and define $y:=\sigma|_k(x)$ and $z:=(\sigma|_k)^2(x)$.
Note that $y,z \in X(\kcyc)$ by Lemma \ref{complex embed in case deg 6}, 
and that 
$(\sigma|_k)^3$ is the complex conjugation, which acts on $X(\kcyc)=X(\kcyc)^-$ as inverse.
Applying $\sigma|_k$ repeatedly to $\D{1}$ and $\I{1}$, we obtain
$$
\begin{array}{lll}\label{decomp group in proof in the case deg 6}
&
\Do{2}=\langle \sigma|_k(\gam) \rangle \x \langle y \rangle,
&
\D{3}=\langle (\sigma|_k)^2(\gam) \rangle \x \langle z \rangle,
\\
\Do{1}=\langle (\sigma|_k)^3(\gam) \rangle \x \langle x^{-1} \rangle,
&
\D{2}=\langle (\sigma|_k)^4(\gam) \rangle \x \langle y^{-1} \rangle,
&
\Do{3}=\langle (\sigma|_k)^5(\gam) \rangle \x \langle z^{-1} \rangle,
\end{array}
$$
again by Lemma \ref{complex embed in case deg 6}.
Here, the left factor of each direct product is the inertia group corresponding to the respective prime ideal.
Since $X'(\kcyc)=0$, we have
$$
X(\kcyc)=
\left\langle
\D{1} \cap X(\kcyc),\ 
\ldots,\ 
\Do{3} \cap X(\kcyc)
\right\rangle
=
\langle
x,y,z
\rangle.
$$
Since $\sigma|_k$ acts trivially on $\Gal(\kcyc/k)$,
we get 
$\sigma|_k(\gam) \equiv \gam \pmod {X(\kcyc)}$.
Thus, there exist $a,b,c \in \Zp$ such that 
$$
\sigma|_k(\gam)=\gam x^a y^b z^c.
$$
Combining this with (\ref{decomp group in proof in the case deg 6}), we obtain the desired representation of the decomposition groups. 
Finally, 
by (\ref{Gal(tilde{k}/k) is generated by these six decomposition groups}),
we have
$\Gal(\tilde{k}/k)=\langle \gam, x,y,z \rangle$.
\end{proof}

We write $A \sim B$ for matrices $A$ and $B$ if they have the same rank.
We state some properties of the parameters $a$, $b$, and $c$:

\begin{cor}\label{A(k) in case deg 6}
Define
$$
|A|:=
(a-b+c)\left( (a^2+bc)+(b^2-ac)+(c^2+ab) \right).
$$
{\rm (} $|A|$ is actually the determinant of a certain matrix.{\rm )}
Depending on the number of generators of $A(k)$, the parameters $a$, $b$, and $c$ satisfy the following congruences modulo $p$:
\begin{itemize}
\setlength{\parskip}{0pt} 
\setlength{\itemsep}{0pt} 
\item[{\rm (i)}]
$A(k)=0$ $\iff$ $|A| \not\equiv 0$.
\item[{\rm (ii)}]
$A(k)$ is cyclic $\iff$ 
$|A| \equiv 0$ and
$
\begin{cases}
a^2+bc \not\equiv 0\ \text{or}
\\
b^2-ac \not\equiv 0\ \text{or}
\\
c^2+ab \not\equiv 0.
\end{cases}
$
\item[{\rm (iii)}]
$\dim_{\Fp}A(k)/pA(k)=2$ $\iff$ 
$
\begin{cases}
a^2+bc \equiv b^2-ac \equiv c^2+ab \equiv 0\ \text{and}
\\
\text{at least one of $a$, $b$, and $c$ is not congruent to $0$}.
\end{cases}
$
\end{itemize}
\end{cor}

\begin{proof}
By Lemma \ref{Z in case deg 6}, we have
$$
A(k) \isom 
\langle \gam, x, y,z  \rangle/\langle \I{1},\ldots, \Io{3} \rangle
=
\langle \gam, x, y,z \rangle/\langle \gam, x^ay^bz^c, x^{-c}y^az^b, x^{-b}y^{-c}z^{a} \rangle.
$$
Hence, $A(k)$ admits a presentation with representation matrix
$$
A:=
\begin{bmatrix}
1 & 0 & 0 &0
\\
0 & a & b & c
\\
0 & -c & a & b
\\
0 & -b & -c & a
\end{bmatrix}.
$$
In other words, there exists a linear map $\langle \gam, x, y,z \rangle \to \langle \gam, x, y, z \rangle$ whose cokernel is isomorphic to $A(k)$ and whose representation matrix is $A$.
Then the determinant of $A$ is equal to $|A|$, as defined previously.
Therefore, it suffices to analyze the rank of the reduction modulo $p$, denoted $\overline{A}:=A \pmod p$:
$$
\left\{
\begin{array}{ccc}
A(k)=0 & \iff & \rank\; \overline{A} =4 
\\
\text{$A(k)$ is cyclic} & \iff & \rank\; \overline{A} =3 
\\
\dim_{\Fp}A(k)/pA(k)=2 & \iff & \rank\; \overline{A} =2.
\end{array}
\right.
$$
The corollary follows from this observation.
\end{proof}

%

\begin{lem}\label{(a^2+bc)+(c^2+ab) (b^2-ac)+(c^2+ab) not 0}
We have 
$(a^2+bc)+(c^2+ab) \neq 0$ and $(b^2-ac)+(c^2+ab) \neq 0$.
\end{lem}

\begin{proof}
Define 
$K \leftrightarrow \langle \I{1}, \Io{1}, \I{2},\Io{2} \rangle$,
where the notation $\leftrightarrow$ is as in Definition {\rm \ref{leftrightarrow}}.
Then $K/k$ is an abelian $p$-extension unramified outside $\{ \p_3, \overline{\p_3} \}$.
By \cite[Lemmas 2(2) and 3]{Fujii2015}, we have 
$
\rank_{\Zp} \Gal(K/k)=1
$, 
and the prime ideal $\overline{\p_2}$ splits finitely in $K/k$.
Hence, the group
$
\langle \I{1}, \Io{1}, \I{2},\Do{2} \rangle
=
\langle \gam, y, x^az^c, x^{-b-c}z^{a+b} \rangle
$
is a subgroup of $\langle \gam, x,y,z \rangle$ of finite index.
This implies that 
$$
\det 
\begin{bmatrix}
a & c
\\
-b-c & a+b
\end{bmatrix}
=(a^2+bc)+(c^2+ab) \neq 0.
$$
By a similar argument, considering
$
\langle \I{1}, \Do{1}, \I{2},\Io{2} \rangle
=
\langle \gam, x, y^bz^c, y^{a-c}z^{a+b} \rangle
$, 
we obtain
$(b^2-ac)+(c^2+ab) \neq 0$.
\end{proof}

Let $G:=\Gal(L(\kcyc)/k)$.
We compute $\mathcal{K}(G)$.

\begin{lem}\label{mathcal{K} in case deg 6}
Define
$$
|K|:=
(a+b+c)^3+(a+b-c)^3
=
2(a+b)\left( (a+b)^2+3c^2 \right).
$$
{\rm (} $|K|$ is actually the determinant of a certain matrix.{\rm )}
Let $G:=\Gal(L(\kcyc)/k)$ and let $\mathcal{K}(G)$ be the module defined in Proposition {\rm \ref{central}}.
Then $\mathcal{K}(G) =0$ if and only if 
$|K| \not\equiv 0 \pmod p$.
\end{lem}

\begin{proof}
Let $F$ be a free pro-$p$ group generated by $\gam$, $x$, $y$, and $z$, 
and let $R$ be the closed normal subgroup of $F$ generated by the commutators $[\gam,x]$, $[\gam,y]$, $[\gam,z]$, $[x,y]$, $[y,z]$, $[z,x]$, and their conjugates.
By Lemma \ref{Z in case deg 6}, we have a minimal presentation
$
1 \to R \to F \to \Gal(L(\kcyc)/k) \to 1
$
of $\Gal(L(\kcyc)/k)$, and by (\ref{homology isom commutator}),
\begin{eqnarray*}
H_2(\Gal(L(\kcyc)/k),\Zp) 
&\isom& 
R \cap [F,F]/[R,F]
\\
&=&
\langle [\gam,x], [\gam,y], [\gam,z], [x,y], [y,z], [z,x] \rangle [R,F]/[R,F]
\end{eqnarray*}
Similarly, the images of $H_2(\D{i},\Zp)$ and $H_2(\Do{i},\Zp)$ ($i=1,2,3$) via the map (\ref{judge}) are generated by
$$
\begin{array}{lll}
{\rm [} \gam,x ],
&
[\gam x^{-b-c}z^{a+b}, y^{-1}],
&
[\gam x^{a-c}y^{a+b},z],
\\ 
{\rm [} \gam y^{a+b-c}z^{a+b+c}, x^{-1} ],
&
[\gam x^az^c, y],
&
[\gam x^{-b}y^{-c}, z^{-1}],
\end{array}
$$
respectively.
Modulo $[R,F]$, these satisfy
$$
\begin{array}{l}
{\rm [}\gam x^{-b-c}z^{a+b}, y^{-1}] \equiv [\gam,y]^{-1}[x,y]^{b+c}[y,z]^{a+b},
\\
{\rm [}\gam x^{a-c}y^{a+b},z] \equiv [\gam,z][y,z]^{a+b}[z,x]^{c-a},
\\ 
{\rm [} \gam y^{a+b-c}z^{a+b+c}, x^{-1} ] \equiv [\gam,x]^{-1}[x,y]^{a+b-c}[z,x]^{-a-b-c},
\\
{\rm [}\gam x^az^c, y] \equiv [\gam,y][x,y]^a[y,z]^{-c},
\\
{\rm [}\gam x^{-b}y^{-c}, z^{-1}] \equiv [\gam,z]^{-1}[y,z]^c[z,x]^{-b}.
\end{array}
$$
Fixing the basis $\{ [\gam,x], [\gam,y], [\gam,z], [x,y], [y,z], [z,x] \}$ of the $\Zp$-module $R \cap [F,F]/[R,F]$, we obtain a presentation of $\mathcal{K}(G)$ with the following representation matrix:
\begin{eqnarray*}
K
&:=&
\left[
\begin{array}{ccc|ccc}
1 & & & 0 & 0 & 0
\\
 & 1 & & a & -c & 0
\\
 & & 1 & 0 & a+b & c-a
\\ \cline{1-6}
-1 & & & a+b-c & 0 & -a-b-c
\\
 & -1 & & b+c & a+b & 0
\\
 & & -1 & 0 & c & -b
\end{array}
\right],
\end{eqnarray*}
where blank entries are all zeros.
By elementary row operations, $K$ can be reduced to
\begin{eqnarray*}
\left[
\begin{array}{ccc|ccc}
1 & & & 0 & 0 & 0
\\
 & 1 & & a & -c & 0
\\
 & & 1 & 0 & a+b & c-a
\\ \cline{1-6}
 & & & a+b-c & 0 & -a-b-c
\\
& & & a+b+c & a+b-c & 0
\\
 & & & 0 & a+b+c & -a-b+c
\end{array}
\right].
\end{eqnarray*}
Thus, $\mathcal{K}(G)=0$ if and only if 
the determinant of the bottom-right $3 \x 3$ block is nonzero modulo $p$, i.e.,
$$
-\det 
\begin{bmatrix}
a+b-c & 0 & -a-b-c
\\
a+b+c & a+b-c & 0
\\
0 & a+b+c & -a-b+c
\end{bmatrix}
=(a+b+c)^3+(a+b-c)^3 \not\equiv 0
\mod p.
$$
This completes the proof.
\end{proof}

\begin{rem}
{\rm 
The correspondence with \cite{Okano2012-1}:
The special case considered in this article --- 
where $p=3$ and $k$ is a totally imaginary abelian extension of degree $6$ ---
is treated in \cite{Okano2012-1}.
Here, we present the correspondence of notation in Table \ref{table correspondence with okano2012-1}
(see also Remark \ref{in case deg 6 dim=1 p=3}).
}
\end{rem}
\begin{longtable}{|c|c|c|}
\hline
   \cite{Okano2012-1} & this article 
\\
\hline
\hline
$\del$, $\ep$, $\ep^\del$, $\ep^{\del^2}$ & $\sigma^2$, $x$, $z$, $y^{-1}$
\\
\hline
$j_0$, $j_1$, and $j_2$ & $a-b-c$, $a+b+c$, and $-a-b+c$
\\
$J(\del)$ & $(a-b-c)+(a+b+c)\sigma^2 +(-a-b+c)(\sigma^2)^2$
\\
$A(\del)$ & $(a-c)+(b+c)\sigma^2+(-a-b)(\sigma^2)^2$
\\
\hline
$\alp=j_1-j_2$ and $\alp \equiv 0$ mod $3$ & $2(a+b)$ and $a+b \equiv 0$ mod $3$ 
\\
\hline
\caption{}
\label{table correspondence with okano2012-1}
\end{longtable}

Hereafter, the notation $\equiv$ denotes congruence modulo $p$.

In the remainder of this subsection, we restate condition (iii) of Theorem \ref{main thm}(III) using slightly different language.
Following Theorem \ref{main thm}, let $T_1=\{\p_1, \overline{\p_1} \}$ and $T_2=\{ \p_2, \overline{\p_2} \}$, and denote by $M_p^{T_i}(k)$ the maximal abelian $p$-extension of $k$ that is unramified outside $p$ and completely decomposed at every prime ideal in $T_i$ ($i=1,2$).
Then $M_p^{T_i}(k)$ is the fixed field of 
$\langle \D{i}, \Do{i}\rangle$,
and hence the compositum $M_p^{T_1}(k)M_p^{T_2}(k)$ is the fixed field of 
$\langle \D{1}, \Do{1}\rangle \cap \langle \D{2}, \Do{2}\rangle$:
$$
M_p^{T_1}(k)M_p^{T_2}(k) \leftrightarrow \langle \D{1}, \Do{1}\rangle \cap \langle \D{2}, \Do{2}\rangle.
$$
Therefore, condition (iii) of Theorem \ref{main thm}(III) is equivalent to the following condition:
\begin{eqnarray}\label{equiv (iii) by decomp gp} 
\left\langle \langle \D{1}, \Do{1}\rangle \cap \langle \D{2}, \Do{2}\rangle, \D3 \right\rangle 
=
\langle \gam, x,y,z \rangle.
\end{eqnarray}

\begin{lem}\label{|K| not cong 0 iff}
\ 
\begin{itemize}
\setlength{\parskip}{0pt} 
\setlength{\itemsep}{0pt} 
\item[{\rm (i)}]
Condition {\rm (iii)} of {\rm Theorem \ref{main thm}(III)} holds if and only if 
$|K| \not\equiv 0 \pmod p$.
\item[{\rm (ii)}]
The complex conjugation acts on the quotient 
$\langle \gam, x,y,z \rangle/\left( \langle \D{1}, \Do{1}\rangle \cap \langle \D{2}, \Do{2}\rangle \right)$ 
as inverse.
\end{itemize}
\end{lem}

\begin{proof}
(i) 
By the above argument, it suffices to show that condition (\ref{equiv (iii) by decomp gp}) holds if and only if 
$|K| \not\equiv 0$.
To this end, we compute the intersection
$\langle \D{1}, \Do{1}\rangle \cap \langle \D{2}, \Do{2}\rangle$.
For simplicity, set
$$
\op:=a+b+c, 
\ \  
\ominus:=a+b-c.
$$
Every element in this group can be expressed as
$$
\gam^{p_1}x^{q_1}(y^{\ominus}z^{\op})^{r_1}
=
(\gam x^a z^c)^{p_2}y^{q_2}(x^{-\op} z^{\ominus})^{r_2}
$$
for some $p_i, q_i, r_i \in \Zp$ ($i=1,2$).
Treating $p_i, q_i, r_i$ as six variables and fixing the basis $\{ \gam,x,y,z\}$ of the $\Zp$-module $\langle \gam, x,y,z \rangle$, we consider the 
coefficient matrix
\begin{eqnarray}\label{augmented coefficient matrix}
\left[
\begin{array}{ccc|ccc}
1 & 0 & 0 & 1 & 0 & 0  
\\
0 & 1 & 0 & a & 0 & -\op  
\\
0 & 0 & \ominus & 0 & 1 & 0  
\\
0 & 0 & \op & c & 0 & \ominus
\end{array}
\right].
\end{eqnarray}

In the case where $\op \equiv \ominus \equiv 0$, we have $a+b \equiv c \equiv 0$ and $|K| \equiv 0$ by Lemma \ref{mathcal{K} in case deg 6}.
On the other hand, modulo $p$, the matrix in (\ref{augmented coefficient matrix}) reduces to
$$
\left[
\begin{array}{ccc|ccc}
1 & 0 & 0 & 1 & 0 & 0  
\\
0 & 1 & 0 & a & 0 & 0  
\\
0 & 0 & 0 & 0 & 1 & 0  
\\
0 & 0 & 0 & 0 & 0 & 0
\end{array}
\right]
\mod p.
$$
This implies that $\langle \D{1}, \Do{1}\rangle \cap \langle \D{2}, \Do{2}\rangle$ is 
generated by an element of the form $\gam x^a$ multiplied by some $p$-power elements, together with $p$-power elements.
Hence, the dimension over $\Fp$ of 
\begin{eqnarray*}
\langle \gam, x,y,z \rangle/\left\langle \langle \D{1}, \Do{1}\rangle \cap \langle \D{2}, \Do{2}\rangle, \D{3}, \gam^p,x^p,y^p,z^p  \right\rangle
=
\dim_{\Fp}
\langle \gam, x,y,z \rangle/\left\langle \gam x^a,x^p,y^p,z  \right\rangle
\end{eqnarray*}
is greater than or equal to $2$, 
which implies that 
$\langle \gam, x,y,z \rangle 
\neq 
\left\langle \langle \D{1}, \Do{1}\rangle \cap \langle \D{2}, \Do{2} \rangle, \D{3} \right\rangle$.
Therefore, the equivalence in the claim holds in this case.

In the case where $\op \not\equiv 0$, applying Gaussian elimination to the matrix in (\ref{augmented coefficient matrix}), we obtain the solutions
\begin{eqnarray*}
&&
\begin{bmatrix}
p_1
\\
q_1
\\
r_1
\end{bmatrix}
=
\begin{bmatrix}
1
\\
a
\\
c \op^{-1} 
\end{bmatrix}
p_2+
\begin{bmatrix}
0
\\
-\op
\\
\op^{-1} \ominus
\end{bmatrix}
r_2
\ \ 
(\text{$p_2, r_2 \in \Zp$ are arbitrary}).
\end{eqnarray*}
Hence, we have
$$
\langle \D{1}, \Do{1}\rangle \cap \langle \D{2}, \Do{2}\rangle
=
\langle \gam^{\op}x^{a\op}y^{c\ominus}z^{c\op},\;  x^{\op^2}y^{-\ominus^2}z^{-\op\ominus} \rangle,
$$
and consequently
$$
\left\langle \langle \D{1}, \Do{1}\rangle \cap \langle \D{2}, \Do{2}\rangle, \D3 \right\rangle
=
\langle \gam^{\op}x^{a\op}y^{c\ominus},  x^{\op^2}y^{-\ominus^2}, \gam x^{a-c}y^{a+b}, z\rangle.
$$
Since
$$
\det
\left[
\begin{array}{cccc}
\op & a\op & c\ominus & 0   
\\
0 & \op^2 & -\ominus^2 & 0  
\\
1 & a-c & a+b & 0 
\\
0 & 0 & 0 & 1
\end{array}
\right]
=\op |K|,
$$
the equivalence in the claim follows.
In the case where $\ominus \not\equiv 0$, we similarly obtain the solutions
\begin{eqnarray*}
&&
\begin{bmatrix}
p_1
\\
q_1
\\
r_1
\end{bmatrix}
=
\begin{bmatrix}
1
\\
a+c \op \ominus^{-1}
\\
0
\end{bmatrix}
p_2
+
\begin{bmatrix}
0
\\
-\op^2 \ominus^{-2}
\\
\ominus^{-1}
\end{bmatrix}
q_2
\ \ 
(\text{$p_2, q_2 \in \Zp$ are arbitrary}).
\end{eqnarray*}
Hence, we have
$$
\langle \D{1}, \Do{1}\rangle \cap \langle \D{2}, \Do{2}\rangle
=
\langle \gam^{\ominus}x^{c\op+a\ominus},\;  x^{\op^2}y^{-\ominus^2}z^{-\op\ominus} \rangle,
$$
and consequently
$$
\left\langle
\langle \D{1}, \Do{1}\rangle \cap \langle \D{2}, \Do{2}\rangle,
\D{3}
\right\rangle
=
\langle \gam^{\ominus}x^{c\op+a\ominus},  x^{\op^2}y^{-\ominus^2},
\gam x^{a-c}y^{a+b}, z \rangle.
$$
Thus, the equivalence in the claim follows from the computation
 $$
\det
\left[
\begin{array}{cccc}
\ominus & c\op+a\ominus & 0 & 0   
\\
0 & \op^2 & -\ominus^2 & 0  
\\
1 & a-c & a+b & 0 
\\
0 & 0 & 0 & 1
\end{array}
\right]
=\ominus |K|.
$$
\\
(ii)
It suffices to show that, for $i=1,2$, the element $(\sigma|_k)^3(\gam)=\gam x^{a-b-c}y^{a+b-c}z^{a+b+c}$ is congruent to $\gam^{-1}$ modulo $\langle \D{i}, \Do{i} \rangle$.
This is equivalent to 
$\langle \D{i}, \Do{i}, \gam^2 x^{a-b-c}y^{a+b-c}z^{a+b+c} \rangle=\langle \D{i}, \Do{i} \rangle$.
For $i=1$, considering the representation matrix of the left-hand group and performing the computation
$$
\left[
\begin{array}{cccc}
1 & 0 & 0 & 0   
\\
1 & a-b-c & a+b-c & a+b+c 
\\
0 & 1 & 0 & 0 
\\
2 & a-b-c & a+b-c & a+b+c
\end{array}
\right]
\sim
\left[
\begin{array}{cccc}
1 & 0 & 0 & 0   
\\
1 & a-b-c & a+b-c & a+b+c 
\\
0 & 1 & 0 & 0 
\\
0 & 0 & 0 & 0
\end{array}
\right],
$$
we conclude that 
$\langle \D{1}, \Do{1}, \gam^2 x^{a-b-c}y^{a+b-c}z^{a+b+c} \rangle=\langle \D{1}, \Do{1} \rangle$.
Similarly, the same result holds for $i=2$.
\end{proof}

If $\dim_{\Fp} A(k)/pA(k)=2$, then $|K| \equiv 0 \pmod p$ always holds.
Indeed, by Lemma \ref{A(k) in case deg 6}, 
all of $a^2+bc$, $b^2-ac$, and $c^2+ab$ are congruent to $0$ modulo $p$, 
and hence the result follows immediately from Lemma \ref{mathcal{K} in case deg 6}.
Combining this with Lemma \ref{|K| not cong 0 iff}(i) and the definition of $|K|$ in Lemma \ref{mathcal{K} in case deg 6},
we see that condition {\rm (iii)} of {\rm Theorem \ref{main thm}(III)} does not hold if one of the following two cases occurs: 
\begin{itemize}
\setlength{\parskip}{0pt} 
\setlength{\itemsep}{0pt} 
\item[{\rm (A)}]
``$\dim_{\Fp}A(k)/pA(k) =2$''
or
``$\dim_{\Fp}A(k)/p A(k) \le 1$, $a+b \not\equiv 0$, and $(a+b)^2+3c^2 \equiv 0 \pmod p$''.
\item[{\rm (B)}]
$\dim_{\Fp}A(k)/pA(k) \le 1$ and $a+b \equiv 0 \pmod p$.
\end{itemize}
On the other hand, if 
condition {\rm (iii)} of {\rm Theorem \ref{main thm}(III)} holds, 
then $\mathcal{K}(\Gal(L(\kcyc)/k))=0$.
Hence, $X(\tilde{k})=0$ by Proposition \ref{central} and Nakayama's lemma.
Summarizing the above arguments, the problem reduces to proving the following proposition:

\begin{prop}\label{(A)(B) then X(tilde{k}) neq 0}
Suppose that conditions {\rm (I)}, {\rm (II)}, and {\rm (III)} given in {\rm \S\ref{Narrowing-down target}} hold.
If either case {\rm (A)} or {\rm (B)} above holds, then $X(\tilde{k}) \neq 0$. 
\end{prop}

\begin{rem}\label{in case deg 6 dim=1 p=3}
{\rm
(In the case where $p=3$, relating to the correspondence with \cite{Okano2012-1})
Suppose that $p=3$.
Then we have $|K| \equiv 2(a+b)^3 \equiv 2(a+b)$.
Hence, 
$$
\text{
$|K| \not\equiv 0 \pmod 3$ 
$\iff$
$a+b \not\equiv 0 \pmod 3$.
}
$$
If $\dim_{\Fp}A(k)/3A(k)=2$, then $|K| \equiv 0 \pmod 3$ by the above argument.

On the other hand, if $A(k)$ is cyclic, then the condition $a+b \not\equiv 0$ always holds.
Indeed, assume $p=3$, $A(k)$ is cyclic, and $a+b \equiv 0$.
Then by Corollary \ref{A(k) in case deg 6}(ii), we have
$$
0 \equiv |A| 
\equiv 
(2a+c)\left( (a^2-ac)+(a^2-ac)+(c^2-a^2) \right)
\equiv 
(c-a)^3 \equiv c-a.
$$
From this, it follows that $a^2+bc \equiv b^2-ac \equiv c^2+ab \equiv 0$.
This contradicts Corollary \ref{A(k) in case deg 6}(ii).
Therefore, when $p=3$ and $A(k)$ is cyclic, 
we can omit condition (iii) of Theorem \ref{main thm}(III).

Moreover, since 
$\langle \D{1}, \D{2} \rangle= \langle \gam,x,y,z^{a+b} \rangle$,
the condition $a+b \not\equiv 0$ means that there is no nontrivial $p$-extension of $k$ that is completely decomposed at any prime ideal in 
$\{\p_1, \p_2 \}$ and unramified outside $\{ \p_3, \overline{\p_1}, \overline{\p_2}, \overline{\p_3} \}$.
Therefore, $|K| \not\equiv 0 \pmod 3$ is equivalent to condition (e) in Theorem 1.2 of \cite{Okano2012-1}.
}
\end{rem}



\subsection{Proof of Theorem \ref{main thm}(III)}\label{Pf of in case deg 6}
\subsubsection{Proof in the case (A)}
\label{Proof in the case deg 6 dim=2}

Suppose that case (A) in Proposition \ref{(A)(B) then X(tilde{k}) neq 0} holds.
In this subsection, we prove the following proposition.
The method of proof is similar to that in \S \ref{proof in case non-Gal deg 4}.

\begin{prop}\label{main thm in case deg 6 dim=2}
Let $k$ be a CM-field of degree $6$.
Suppose that $k$ satisfies conditions {\rm (I)}, {\rm (II)}, and {\rm (III)} given in {\rm \S\ref{Narrowing-down target}}.
If either ``$\dim_{\Fp}A(k)/pA(k) =2$''
or
``$\dim_{\Fp}A(k)/pA(k) \le 1$, $a+b \not\equiv 0$, and $(a+b)^2+3c^2 \equiv 0 \pmod p$''
holds, then $X(\tilde{k})$ is nontrivial.
\end{prop}

Recall that the notation $\equiv$ denotes congruence modulo $p$.
For simplicity, define
$$
\op:=a+b+c,
\ \ 
\ominus:=a+b-c.
$$
We then have the following lemma.

\begin{lem}\label{abc, a+b+c, a+b-c notequi 0}
We have $c \not\equiv 0$, $\op=a+b+c \not\equiv 0$, and $\ominus=a+b-c \not\equiv 0 \pmod p$.
\end{lem}

\begin{proof}
First, we consider the case $\dim_{\Fp}A(k)/pA(k) =2$.
Suppose, for the sake of contradiction, that $a \equiv 0$.
Then $b \equiv c \equiv 0$ follows, since $b^2-ac \equiv c^2+ab \equiv 0$.
This is a contradiction since at least one of $a$, $b$, and $c$ is not congruent to $0$. 
Thus, $a \not\equiv 0$.
Similarly, we have $b \not\equiv 0$ and $c \not\equiv 0$.
Assume that $\op \equiv 0$, i.e., $c \equiv -a-b$.
Then
$$
0 \equiv c^2+ab \equiv a^2+3ab+b^2
\ \ \text{and}\ \ 
0 \equiv a^2+bc \equiv a^2-ab-b^2.
$$
From this, we obtain $2a(a+b) \equiv 0$.
Since $a \not\equiv 0$, this implies $a+b \equiv 0$.
Thus, $0 \equiv \op \equiv c \not\equiv 0$, which yields a contradiction.
Therefore, $\op \not\equiv 0$.
Now, suppose $\ominus \equiv 0$, i.e., $c \equiv a+b$.
Then 
$$
0 \equiv c^2+ab \equiv a^2+3ab+b^2
\ \ \text{and}\ \ 
0 \equiv a^2+bc \equiv a^2+ab+b^2.
$$
These imply $2ab \equiv 0$, contradicting $a \not\equiv 0$ and $b \not\equiv 0$.
Hence, $\ominus \not\equiv 0$.

Next, consider the case where $\dim_{\Fp}A(k)/pA(k) \le 1$, $a+b \not\equiv 0$, and $(a+b)^2+3c^2 \equiv 0 \pmod p$.
Suppose $c \equiv 0$.
Then, since $|K| \equiv 0$ and $a+b \not\equiv 0$ by assumption, we have
$$
0 \equiv (a+b)^2+3c^2 \equiv (a+b)^2 \not\equiv 0,
$$
which is a contradiction.
Therefore, $c \not\equiv 0$.
Now, suppose $a+b \equiv \pm c$.
Then $0 \equiv (a+b)^2+3c^2 \equiv 4c^2$,
which contradicts $c \not\equiv 0$.
\end{proof}

Define $P_\infty$ and $Q_\infty$ as follows:
\begin{eqnarray}\label{D1Do1,D2Do2 in case deg 6 dim=2}
\begin{cases}
P_\infty \leftrightarrow \langle \D{1}, \Do{1} \rangle 
=
\langle \gam,x, y^{a+b-c}z^{a+b+c} \rangle
=
\langle \gam,x, y^{\ominus}z^{\op} \rangle,
\\
Q_\infty \leftrightarrow \langle \D{2}, \Do{2} \rangle 
=
\langle \gam x^az^c,y, x^{-a-b-c}z^{a+b-c} \rangle
=
\langle \gam x^az^c,y, x^{-\op}z^{\ominus} \rangle.
\end{cases}
\end{eqnarray}
By Lemma \ref{abc, a+b+c, a+b-c notequi 0}, both $P_\infty$ and $Q_\infty$ are $\Zp$-extensions of $k$.
For convenience, set 
$\mathbb{D}:= \langle \D{1}, \Do{1} \rangle \cap \langle \D{2}, \Do{2} \rangle$.
Then, by Lemma \ref{|K| not cong 0 iff} and its proof, 
$$
\mathbb{D}
=
\langle \gam^{\op}x^{a\op}y^{c\ominus}z^{c\op},\;  x^{\op^2}y^{-\ominus^2}z^{-\op\ominus} \rangle
$$
and the complex conjugation 
acts on $\langle \gam,x,y,z \rangle/\mathbb{D}$ as inverse.
Note that $\langle \mathbb{D}, \I{3} \rangle=\langle \mathbb{D}, \Io{3} \rangle$, 
and that 
$\langle \gam,x,y,z \rangle/\mathbb{D}$ is a free $\Zp$-module of rank $2$.
\begin{figure}[H]
$$ 
\begin{xy}
(0,0) *{\text{$\langle \gam,x,y,z \rangle$}}="k",
(-18,12) *{\text{$\langle \D{1},\Do{1},\I{3} \rangle$}}="P",
(-36,24) *{\text{$\langle \D{1},\Do{1} \rangle$}}="Pinf",
(18,12) *{\text{$\langle \D{2},\Do{2},\I{3} \rangle$}}="Q",
(36,24) *{\text{$\langle \D{2},\Do{2} \rangle$}}="Qinf",
(-9,30) *{\text{$\langle \mathbb{D}, \I{3} \rangle$}}="inertia",
(-12,24) *{\text{$ \bullet $}}="someZp",
(12,24) *{\text{$\bullet$}}="kinf",
(19,24) *{\text{$\leftrightarrow k_\infty$}},
(0,48) *{\text{$\mathbb{D}$}}="PinfQinf",
(-51,24) *{\text{$P_\infty \leftrightarrow$}},
(51,24) *{\text{$\leftrightarrow Q_\infty$}},
(12,48) *{\text{$\leftrightarrow P_\infty Q_\infty$}},
\ar@{-} "k";"P"
\ar@{-} "P";"Pinf"
\ar@{-} "k";"Q"
\ar@{-} "Q";"Qinf"
\ar@{-} "k";"kinf"
\ar@{-} "kinf";"PinfQinf"
\ar@{-} "k";"someZp"
\ar@{-} "someZp";"inertia"
\ar@{-} "inertia";"PinfQinf"
_(.4){\text{$\substack{\p_3, \overline{\p_3}: \\ \text{ram}}$}}
\ar@{-} "Pinf";"PinfQinf"
^{\text{$\substack{\p_3, \overline{\p_3}: \\ \text{unram}}$}}
\ar@{-} "Qinf";"PinfQinf"
_{\text{$\substack{\p_3, \overline{\p_3}: \\ \text{unram}}$}}
\end{xy}
$$
\caption{}
\label{X neq 0 galois diagram in case deg 6 dim=2}
\end{figure}
To construct the desired anticyclotomic $\Zp$-extension $\ki/k$ satisfying 
the condition in Theorem \ref{main prop},
we examine the inertia groups of $\p_3$ and $\overline{\p_3}$ in $\Gal(P_\infty Q_\infty/k)$
(see Figure \ref{X neq 0 galois diagram in case deg 6 dim=2}).
Note that 
$$
k \leftrightarrow 
\langle \D{1}, \Do{1}, \D{2}, \Do{2} \rangle
=
\langle \gam,x,y,z^{\op},z^{\ominus} \rangle
$$ 
by (\ref{D1Do1,D2Do2 in case deg 6 dim=2}) and Lemma \ref{abc, a+b+c, a+b-c notequi 0}.
Since $(a^2+bc)+(c^2+ab) \neq 0$ by Lemma \ref{(a^2+bc)+(c^2+ab) (b^2-ac)+(c^2+ab) not 0}, we have
$$
\langle \D{1}, \Do{1},\I{3} \rangle
=\langle \D{1}, \Do{1},\Io{3} \rangle
=\langle \gam, x, y^{-c}z^a, y^{a+b}z^{b+c} \rangle,
$$
which is a subgroup of $\langle \gam, x,y,z \rangle$ of finite index.
This implies that the extension $P_\infty/k$ is ramified at both $\p_3$ and $\overline{\p_3}$.
Hence, the extension $P_\infty Q_\infty/P_\infty$ is unramified at all primes lying above $\p_3$ and $\overline{\p_3}$.
Similarly, since $(b^2-ac)+(c^2+ab) \neq 0$, the extension $P_\infty Q_\infty/Q_\infty$ is unramified at all primes lying above $\p_3$ and $\overline{\p_3}$.
On the other hand, it is straightforward to verify that
$\langle \mathbb{D}, x \rangle =\langle \D{1}, \Do{1} \rangle$ and
$\langle \mathbb{D}, y \rangle =\langle \D{2}, \Do{2} \rangle$,
by Lemma \ref{abc, a+b+c, a+b-c notequi 0} and (\ref{D1Do1,D2Do2 in case deg 6 dim=2}).
Hence, there exist $\alp,\beta \in \Zp$ ($\alp, \beta \neq 0$) such that 
$$
\langle \mathbb{D}, \I{3} \rangle=\langle \mathbb{D}, \Io{3} \rangle=\langle \mathbb{D}, x^\alp y^\beta\rangle.
$$
Now, define an integer $d$ as follows:
If $\ord_p(\alp) \neq \ord_p(\beta)$, then set $d:=0$; otherwise, i.e., if $\beta/\alp \in \Zp^\x$, then choose $d$ such that $\pm 1 + pd \neq \beta/\alp$.
Then define
$$
\ki \leftrightarrow 
\begin{cases}
\langle \mathbb{D},xy^{1+pd} \rangle & \text{if $a+b \equiv 0$},
\\
\langle \mathbb{D},xy^{-1+pd} \rangle & \text{if $a+b \not\equiv 0$}.
\end{cases}
$$
Then $\ki$ is distinct from 
$P_\infty(\leftrightarrow \langle \mathbb{D}, x \rangle)$ and 
$Q_\infty(\leftrightarrow \langle \mathbb{D}, y \rangle)$.
Moreover, $\ki$ satisfies the following properties.

\begin{lem}\label{ki in case deg 6 dim=2}
\ 
\begin{itemize}
\setlength{\parskip}{0pt} 
\setlength{\itemsep}{0pt} 
\item[{\rm (i)}]
The extension $\ki/k$ is a $\Zp$-extension, and $P_\infty Q_\infty/\ki$ is a $\Zp$-extension unramified at $\p_3$ and $\overline{\p_3}$.
\item[{\rm (ii)}]
The complex conjugation acts on $\Gal(\ki/k)$ as inverse.
Hence, $\ki/k$ is an anticyclotomic $\Zp$-extension.
\item[{\rm (iii)}]
In the extension $\ki/k$, the prime $p$ is non-split and ramified.
\end{itemize}
\end{lem}

\begin{proof}
Fix the basis $\{ \gam,x,y,z \}$ of the $\Zp$-module $\langle \gam,x,y,z \rangle$.
\\
(i)
Considering the representation matrix of $\Gal(\ki/k) \isom \langle \gam,x,y,z \rangle/\langle \mathbb{D},xy^{\pm 1+pd} \rangle$, we have
$$
\begin{bmatrix}
\op & a\op & c\ominus & c\op
\\
0 & \op^2 & -\ominus^2 & -\op\ominus
\\
0 & 1 & \pm 1+pd & 0
\end{bmatrix}
\sim 
\begin{bmatrix}
1 & 0 & 0 & *
\\
0 & 1 & 0 & *
\\
0 & 0 & 1 & *
\end{bmatrix}
$$
by Lemma \ref{abc, a+b+c, a+b-c notequi 0}.
This implies that both $\ki$ and $P_\infty Q_\infty$ are $\Zp$-extensions over $k$ and over $\ki$, respectively.
Since 
$$
\langle \mathbb{D},xy^{\pm 1+pd}, \I{3} \rangle
=
\langle \mathbb{D},xy^{\pm 1+pd}, x^\alp y^\beta \rangle
$$
and $\pm 1 + pd \neq \beta/\alp$,
we have 
$\langle \mathbb{D},xy^{\pm 1+pd}, \I{3} \rangle \neq \langle \mathbb{D},xy^{\pm 1+pd}\rangle$.
This implies that $\p_3$ is ramified in $\ki/k$.
Therefore, $\p_3$ is unramified in $P_\infty Q_\infty/\ki$, and so is $\overline{\p_3}$.
\\
(ii)
This follows directly from Lemma \ref{|K| not cong 0 iff}.
\\
(iii)
First, we show that $p$ is non-split in $\ki/k$.
By Lemma \ref{abc, a+b+c, a+b-c notequi 0}, 
$$
\langle \mathbb{D}, xy^{\pm 1+pd}, \D{1} \rangle
=\langle \gam, x,y,z^{-\op\ominus} \rangle
=\langle \gam, x,y,z \rangle.
$$
This implies that $\p_1$ is non-split in the extension $\ki/k$.
By similar arguments, the same holds for $\p_2$, $\overline{\p_1}$, and $\overline{\p_2}$.
Now, consider $\p_3$ and $\overline{\p_3}$.
When $a+b \equiv 0$, by considering the representation matrix of $\langle \mathbb{D}, xy^{1+pd}, \D{3} \rangle$ modulo $p$, we have
\begin{eqnarray*}
&&
\begin{bmatrix}
\op & a\op & c\ominus & c\op
\\
0 & \op^2 & -\ominus^2 & -\op\ominus
\\
0 & 1 &  +1+pd & 0
\\
1 & a-c & a+b & 0 
\\
0 & 0  & 0 & 1
\end{bmatrix}
\\
&\equiv&
\begin{bmatrix}
c & ac & -c^2 & c^2
\\
0 & c^2 & -c^2 & c^2
\\
0 & 1 &  1 & 0
\\
1 & a-c & 0 & 0 
\\
0 & 0  & 0 & 1
\end{bmatrix}
\sim
\begin{bmatrix}
1 & a & -c & 0
\\
0 & 1 & -1 & 0
\\
0 & 1 &  1 & 0
\\
1 & a-c & 0 & 0 
\\
0 & 0  & 0 & 1
\end{bmatrix}
\sim
\begin{bmatrix}
1 & a & -c & 0
\\
0 & 1 & -1 & 0
\\
0 & 1 & 1 & 0
\\
0 & -c & c & 0 
\\
0 & 0  & 0 & 1
\end{bmatrix}.
\end{eqnarray*}
The last matrix has rank $4$, which implies that $\p_3$ (and hence $\overline{\p_3}$) is non-split in $\ki/k$.
Similarly, when $a+b \not\equiv 0$, the representation matrix of $\langle \mathbb{D}, xy^{- 1+pd}, \D{3} \rangle$ modulo $p$ takes the form
\begin{eqnarray*}
&&
\begin{bmatrix}
\op & a\op & c\ominus & c\op
\\
0 & \op^2 & -\ominus^2 & -\op\ominus
\\
0 & 1 &  -1 & 0
\\
1 & a-c & a+b & 0
\\
0 & 0  & 0 & 1
\end{bmatrix}.
\end{eqnarray*}
Adding $-\op$ times the fourth row to the first row, 
and $-\op^2$ times the third row to the second row,
the matrix reduces to 
\begin{eqnarray*}
&&
\begin{bmatrix}
0  & c(a+b+c) & -(a+b)^2-c^2 & c(a+b+c)
\\
0 & 0 & 4(a+b)c & -(a+b-c)(a+b+c)
\\
0 & 1 &  -1 & 0
\\
1 & a-c & a+b & 0
\\
0 & 0  & 0 & 1
\end{bmatrix}.
\end{eqnarray*}
This matrix has rank $4$, since $4(a+b)c$ is a unit in $\Zp$.
Hence, $\langle \mathbb{D}, xy^{- 1+pd}, \D{3} \rangle=\langle \gam,x,y,z \rangle$,
and therefore $\p_3$ (and hence $\overline{\p_3}$) is non-split in $\ki/k$.

Next, we show that $p$ is ramified in the extension $\ki/k$.
Recall that $P_\infty Q_\infty/k$ is a $\Zp^2$-extension, in particular a bicyclic extension, and that $\ki/k$ is a $\Zp$-extension.
Suppose that $\p_2$ is unramified in $P_\infty/k$.
Then $\overline{\p_2}$ is also unramified in $P_\infty/k$.
This implies that $P_\infty/k$ is a $\Zp$-extension unramified outside $\{ \p_3, \overline{\p_3} \}$ in which $\p_1$ splits completely.
This contradicts \cite[Lemma 3]{Fujii2015}.
Therefore, $\p_2$ and $\overline{\p_2}$ are ramified in $P_\infty/k$.
We have shown that 
$\langle \D{1}, \Do{1}, \I{3} \rangle 
=\langle \D{1}, \Do{1}, \Io{3} \rangle$
is a subgroup of $\langle \gam, x,y,z \rangle$ of finite index.
Hence, $\p_3$ and $\overline{\p_3}$ are also ramified in $P_\infty/k$.
Similarly, $\p_1$, $\overline{\p_1}$, $\p_3$, and $\overline{\p_3}$ are ramified in $Q_\infty/k$.
Consequently, $\p_1$ and $\overline{\p_1}$ are ramified in $P_\infty Q_\infty/P_\infty$ 
and $\p_2$ and $\overline{\p_2}$ are ramified in $P_\infty Q_\infty/Q_\infty$.
Therefore, these primes are unramified in $P_\infty Q_\infty/\ki$.
By (i), $\p_3$ and $\overline{\p_3}$ are unramified in $P_\infty Q_\infty/\ki$.
Thus, $P_\infty Q_\infty/\ki$ is an unramified extension, which implies that $p$ is ramified in $\ki/k$.
This completes the proof.
\end{proof}

We now prove Proposition \ref{main thm in case deg 6 dim=2}.
Recall that, by Lemma \ref{ki in case deg 6 dim=2}(ii), 
$\ki/k$ is an anticyclotomic $\Zp$-extension.
Since $\p_1$, $\overline{\p_1}$, $\p_2$, and $\overline{\p_2}$ split infinitely in $P_\infty Q_\infty/k$ and are non-split in $\ki/k$ by Lemma \ref{ki in case deg 6 dim=2}(iii), these four primes split completely in the extension $P_\infty Q_\infty/\ki$.
On the other hand, by Lemma \ref{|K| not cong 0 iff} and our assumption that $|K| \equiv 0$,
the prime ideal $\p_3$ splits in $P_\infty Q_\infty/k$.
However, since $\p_3$ does not split in $\ki/k$, it follows that $\p_3$ splits in $P_\infty Q_\infty/\ki$.
The same holds for $\overline{\p_3}$.
Therefore, there exists an unramified subextension $L'/\ki$ with $L' \neq \ki$ contained in $P_\infty Q_\infty$, over which $p$ splits completely.
This implies that $X'(\ki)$ is nontrivial.
Hence,  $X(\tilde{k})$ is also nontrivial by Theorem \ref{main prop}.



\subsubsection{Proof in the case (B)}
\label{Proof in the case deg 6 a+b equiv 0}

Suppose that 
$\dim_{\Fp}A(k)/pA(k) \le 1$ and $a+b \equiv 0 \pmod p$.
In this subsection, we prove the following proposition.
The outline of the proof is similar to that of \S\ref{proof in case non-Gal deg 4}.

\begin{prop}\label{main thm in case deg 6 dim=0,1 and a+b equiv 0}
Let $k$ be a CM-field of degree $6$.
Suppose that $k$ satisfies conditions {\rm (I)}, {\rm (II)}, and {\rm (III)} given in {\rm \S\ref{Narrowing-down target}}.
If $\dim_{\Fp}A(k)/pA(k) \le 1$ and $a+b \equiv 0 \pmod p$, then $X(\tilde{k})$ is nontrivial.
\end{prop}


First, set
$m:=\ord_p(a+b)>0$,
and let $J \in \Gal(k/k^+)$ denote the complex conjugation.
We note that 
$$
a-c \not\equiv 0,
\ \ 
b+c \not\equiv 0.
$$
Indeed, since $a+b \equiv 0$, we have the congruences 
$a^2+bc \equiv a(a-c)$, 
$b^2-ac \equiv a(a-c)$, and
$c^2+ab \equiv (c+a)(c-a)$.
By Corollary \ref{A(k) in case deg 6}(i)(ii), it follows that $a-c \not\equiv 0$, 
and hence 
$b+c \not\equiv 0$.
Next, to construct the desired anticyclotomic $\Zp$-extension $\ki/k$ satisfying 
the condition in Theorem \ref{main prop},
we define three cyclic $p$-extensions $R_m^{(i)}$ ($i=1,2,3$), each of degree $p^m$.
Define
$$
P_m^{(1)} 
\leftrightarrow 
\langle \D{1}, \D{2} \rangle = \langle \gam, x,y, z^{a+b} \rangle,
\ \ 
Q_m^{(1)} 
\leftrightarrow 
\langle \D{1}, \D{3} \rangle = \langle \gam, x,y^{a+b}, z \rangle.
$$
Then 
$P_m^{(1)}Q_m^{(1)} \leftrightarrow 
\langle \D{1}, \D{2} \rangle \cap \langle \D{1}, \D{3} \rangle= \langle \gam, x,y^{a+b}, z^{a+b} \rangle$,
and
$P_m^{(1)} \cap Q_m^{(1)}=k$.
Therefore, $P_m^{(1)}Q_m^{(1)}/k$ is a $(p^m,p^m)$-extension unramified outside $p$.

\begin{lem}\label{R_m^{(1)}}
Define $R_m^{(1)} \subset P_m^{(1)}Q_m^{(1)}$ by $R_m^{(1)} \leftrightarrow \langle \gam, x, yz^{-1}, z^{a+b} \rangle$.
Then:
\begin{itemize}
\setlength{\parskip}{0pt} 
\setlength{\itemsep}{0pt} 
\item[{\rm (i)}]
$R_m^{(1)}/k$ is a cyclic $p$-extension of degree $p^m$, and $J$ acts on $\Gal(R_m^{(1)}/k)$ as inverse.
\item[{\rm (ii)}]
In $R_m^{(1)}/k$, $\p_2$, $\overline{\p_2}$, $\p_3$, and $\overline{\p_3}$ are totally ramified, 
and $\p_1$ and $\overline{\p_1}$ split completely.
\end{itemize}
\end{lem}

\begin{proof}
(i)
It is straightforward to verify that $\Gal(R_m^{(1)}/k)=\langle \gam, x,y,z \rangle/\langle \gam, x, yz^{-1}, z^{a+b} \rangle \isom \Z/p^m \Z$.
Note that
$$
J(\gam)=
\gam x^{a-b-c} y^{a+b-c} z^{a+b+c}=\gam \cdot x^{a-b-c} \cdot (yz^{-1})^{a+b-c} \cdot z^{2(a+b)}
\in 
\langle \gam, x, yz^{-1}, z^{a+b} \rangle,
$$
so $J$ acts on $\langle \gam, x, yz^{-1}, z^{a+b} \rangle$.
Since the class of $z$ generates $\Gal(R_m^{(1)}/k)$, 
$J$ acts on $\Gal(R_m^{(1)}/k)$ as inverse.
\\
(ii)
Since $\D{1} \subset \langle \D{1}, \D{2} \rangle \cap \langle \D{1}, \D{3} \rangle$, 
$\p_1$ splits completely in $P_m^{(1)}Q_m^{(1)}/k$, 
and hence also in $R_m^{(1)}/k$.
On the other hand, because
$\langle \gam, x, yz^{-1}, z^{a+b}, \I{2} \rangle
=\langle \gam, x, y, z \rangle$
due to $a-c \in \Zp^\x$,
the prime ideal $\p_2$ is totally ramified in $R_m^{(1)}/k$.
Similarly, $\p_3$ is also totally ramified in $R_m^{(1)}/k$.
By (i), it follows that $\overline{\p_2}$ and $\overline{\p_3}$ are totally ramified, while $\overline{\p_1}$ splits completely in $R_m^{(1)}/k$. 
\end{proof}

Similarly, we construct cyclic $p$-extensions
$R_m^{(2)}$ and $R_m^{(3)}$, each of degree $p^m$, and obtain their properties from
$P_m^{(2)} \leftrightarrow \langle \D{2}, \D{1} \rangle$, $Q_m^{(2)} \leftrightarrow \langle \D{2}, \D{3} \rangle$, and $P_m^{(3)} \leftrightarrow \langle \D{3}, \D{1} \rangle$, $Q_m^{(3)} \leftrightarrow \langle \D{3}, \D{2} \rangle$, respectively:
$$
\begin{array}{lcl}
R_m^{(2)} \leftrightarrow \langle \gam x^{a-c},y,xz, x^{a+b} \rangle
&:&
\text{In $R_m^{(2)}/k$, }
\begin{cases}
\text{$\p_1$, $\overline{\p_1}$, $\p_3$, $\overline{\p_3}$ are totally ramified,}
\\
\text{$\p_2$, $\overline{\p_2}$ split completely,}
\end{cases}
\\
R_m^{(3)} \leftrightarrow \langle \gam x^{a-c},xy^{-1},z, y^{a+b} \rangle
&:&
\text{In $R_m^{(3)}/k$, }
\begin{cases}
\text{$\p_1$, $\overline{\p_1}$, $\p_2$, $\overline{\p_2}$ are totally ramified,}
\\
\text{$\p_3$, $\overline{\p_3}$ split completely.}
\end{cases}
\end{array}
$$
We now prove Proposition \ref{main thm in case deg 6 dim=0,1 and a+b equiv 0}.
Define
$$
L'_m:=R_m^{(2)}R_m^{(3)} 
\leftrightarrow 
\langle \gam x^{a-c},xy^{-1}z, x^{a+b}, y^{a+b}, z^{a+b} \rangle.
$$
Then we have the inclusion 
$R_m^{(1)} \subset L'_m$
(see Figure \ref{X neq 0 galois diagram in case deg 6 dim=0,1 a+b equiv 0}).
\begin{figure}[H]
$$ 
\begin{xy}
(0,0) *{\text{$\langle \gam,x,y,z \rangle$}}="k",
(-48,24) *{\text{$\langle \gam x^{a-c},y,xz, x^{a+b} \rangle$}}="R_m^{(2)}",
(48,24) *{\text{$\langle  \gam x^{a-c},xy^{-1},z, y^{a+b} \rangle  $}}="R_m^{(3)}",
(-12,24) *{\text{$\langle \gam, x, yz^{-1}, z^{a+b} \rangle$}}="R_m^{(1)}",
(12,24) *{\text{$\bullet$}}="k_m",
(19,24) *{\text{$\leftrightarrow k_m$}},
(0,48) *{\text{$\langle \gam x^{a-c},xy^{-1}z, x^{a+b}, y^{a+b}, z^{a+b} \rangle$}}="R_m^{(1)}R_m^{(2)}",
(33,48) *{\text{$\leftrightarrow L'_m$}},
(-60,30) *{\text{$R_m^{(2)} \leftrightarrow$}},
(60,30) *{\text{$\leftrightarrow R_m^{(3)}$}},
(-20,30) *{\text{$R_m^{(1)} \leftrightarrow $}},
\ar@{-} "k";"R_m^{(1)}"
_{\text{$\substack{\p_1, \overline{\p_1}: \\ \text{split}}$}}
\ar@{-} "k";"R_m^{(2)}"
^{\text{$\substack{\p_2, \overline{\p_2}: \\ \text{split}}$}}
\ar@{-} "k";"R_m^{(3)}"
_{\text{$\substack{\p_3, \overline{\p_3}: \\ \text{split}}$}}
\ar@{-} "k";"k_m"
\ar@{-} "k_m";"R_m^{(1)}R_m^{(2)}"
\ar@{-} "R_m^{(1)}";"R_m^{(1)}R_m^{(2)}"
\ar@{-} "R_m^{(2)}";"R_m^{(1)}R_m^{(2)}"
\ar@{-} "R_m^{(3)}";"R_m^{(1)}R_m^{(2)}"
\end{xy}
$$
\caption{}
\label{X neq 0 galois diagram in case deg 6 dim=0,1 a+b equiv 0}
\end{figure}
Therefore, in the extension $L'_m/k$, the decomposition field of each prime ideal lying above $p$ has degree $p^m$ over $k$.
In the Galois group $\Gal(L'_m/k)$, the inertia groups of $\p_i$ and $\overline{\p_i}$ correspond to $\Gal(L'_m/R_m^{(i)})$ ($i=1,2,3$).
Note that the complex conjugation $J$ acts on $\Gal(L'_m/k)$ as inverse.
Define the intermediate field $k_m$ of $L'_m/k$ by 
$$
k_m \leftrightarrow \langle \gam x^{a-c},xy^{-1}z, xy, x^{a+b}, y^{a+b}, z^{a+b} \rangle.
$$
Then $k_m/k$ is a cyclic extension of degree $p^m$ in which the prime $p$ is totally ramified.
It follows that $L'_m/k_m$ is an unramified cyclic extension in which $p$ splits completely.
Finally, define
\begin{eqnarray}\label{ki in case deg 6 dim=0,1 a+b equiv 0}
\ki \leftrightarrow 
\langle J(\gam)\gam,xy^{-1}z, xy \rangle
=\langle \gam^2x^{a-b-c}y^{a+b-c}z^{a+b+c},xy^{-1}z, xy \rangle.
\end{eqnarray}
Then $k_m \subset \ki$ and 
$J$ acts on $\Gal(\ki/k)$ as inverse, 
i.e., 
$\ki/k$ is an anticyclotomic $\Z_p$-extension. 
Moreover, since $p$ is totally ramified in $k_m/k$,
it is also totally ramified in the $\Z_p$-extension $\ki/k$.
Set $L':=L'_m \ki$.
Then there is a natural projection
$X'(\ki) \surj \Gal(L'/\ki)$.
This implies that $X'(\ki)$ is nontrivial.
Therefore, $X(\tilde{k})$ is also nontrivial by Theorem \ref{main prop}.
This completes the proof of Proposition \ref{main thm in case deg 6 dim=0,1 and a+b equiv 0}.



\subsection{An alternative proof of the sufficient condition}
\label{Proof in the case deg 6 |K| notequiv 0}

We present an alternative proof of the sufficient condition in Theorem \ref{main thm}(III),
based on ideas from Minardi \cite{Minardi}, Itoh \cite{Itoh2011}, and Fujii \cite{Fujii2015}.
In other words, we prove the following statement without relying on Proposition \ref{central}.

\begin{prop}\label{main thm in case deg 6 dim=0,1 and a+b notequiv 0}
Let $k$ be a CM-field of degree $6$.
Suppose that $k$ satisfies conditions {\rm (I)}, {\rm (II)}, and {\rm (III)} given in {\rm \S\ref{Narrowing-down target}}.
If $(a+b) \not\equiv 0$ and $(a+b)^2+3c^2 \not\equiv 0 \pmod p$, 
then $X(\tilde{k})$ is trivial.
\end{prop}

Note that, under these assumptions, we have $\dim_{\Fp}A(k)/pA(k) \le 1$
(see the argument preceding Proposition \ref{(A)(B) then X(tilde{k}) neq 0}).
Therefore, at least one of $a^2+bc$, $b^2-ac$, and $c^2+ab$ is nonzero modulo $p$.
For $1 \le i, j, k \le 3$, we define
$$
\begin{array}{lll}
T^{\langle i \rangle} \leftrightarrow \langle \I{i} \rangle,
&
T^{\langle ij \rangle} \leftrightarrow \langle \I{i}, \I{j} \rangle,
& 
T^{\langle ijk \rangle} \leftrightarrow \langle \I{i}, \I{j},\I{k} \rangle.
\end{array}
$$
Similarly, we define $T^{\langle \overline{i} \rangle}$, $T^{\langle i,\overline{j} \rangle}$, 
and so on;
for example, $T^{\langle 1\overline{2}3 \rangle} \leftrightarrow \langle \I{1}, \Io{2},\I{3} \rangle$.

\begin{lem}\label{(a^2+bc)+(c^2+ab) notequiv 0 or (b^2-ac)+(c^2+ab) notequiv 0}
Either $(a^2+bc)+(c^2+ab) \not\equiv 0 \pmod p$ or $(b^2-ac)+(c^2+ab) \not\equiv 0 \pmod p$.
\end{lem}

\begin{proof}
Suppose, for the sake of contradiction, that 
$(a^2+bc)+(c^2+ab) \equiv 0$ and $(b^2-ac)+(c^2+ab) \equiv 0$.
It then follows that $a^2+bc \equiv b^2-ac$.
Since $a+b \not\equiv 0$, 
this implies that $c \equiv b-a$.
It follows that 
$0 \equiv (a^2+bc)+(c^2+ab) \equiv 2(a^2-ab+b^2)$.
However, evaluating $(a+b)^2 + 3 c^2$ yields
$0 \not\equiv (a+b)^2+3c^2 \equiv 4(a^2-ab+b^2) \equiv 0$,
which is a contradiction.
Hence, the claim follows.
\end{proof}



\subsubsection{Climbing up $\Zp$-extensions}
\label{ascent}

The first step of the proof is to show the following, in a manner similar to the argument in \S \ref{another proof in case non-Gal deg 4}.

\begin{prop}\label{X(T^{1})=0}
For each $i$ with $1 \le i \le 3$, 
the modules $X(T^{\langle i \rangle})$ and $X(T^{\langle \overline{i} \rangle})$ are trivial.
\end{prop}

Consider the field $T^{\langle 1\overline{2}3 \rangle} \leftrightarrow \langle \I{1}, \Io{2},\I{3} \rangle$, 
which contains $L(k)$.
Since $\langle \I{1}, \Io{2},\I{3} \rangle=\langle \gam, x^a y^b z^c, x^{-c} y^a z^b \rangle$ and $\dim_{\Fp}A(k)/pA(k) \le 1$, 
it follows that $T^{\langle 1\overline{2}3 \rangle}/k$ is a $\Zp$-extension.

\begin{lem}\label{X(T^{1,v2,3})=0}
$X(T^{\langle 1\overline{2}3 \rangle})=0$.
\end{lem}

\begin{proof}
The proof is essentially the same as that of Lemma \ref{X^1=0 in case non-Gal deg 4}.
Note that if $A(k)=0$, the claim follows directly from \cite[Corollary 3.3]{Itoh2011}, although this fact is not required here.
Let $L_0$ denote the maximal subfield of $L(T^{\langle 1\overline{2}3 \rangle})$ that is abelian over $k$.
Then, since $\Gal(T^{\langle 1\overline{2}3 \rangle}/k) \isom \Zp$,
we have an isomorphism 
$\Gal(L_0/T^{\langle 1\overline{2}3 \rangle}) \isom X(T^{\langle 1\overline{2}3 \rangle})_{\Gal(T^{\langle 1\overline{2}3 \rangle}/k)}$.
Recall that $\tilde{k}$ is the maximal abelian $p$-extension of $k$ unramified outside $p$.
Thus, $T^{\langle 1\overline{2} \rangle}$ is the maximal abelian $p$-extension of $k$ unramified outside the set
$\{ \overline{\p_1},\p_2, \p_3, \overline{\p_3} \}$.
Consequently, since $L_0/k$ is an abelian $p$-extension unramified outside $\{ \overline{\p_1}, \p_2, \p_3, \overline{\p_3} \}$, 
we conclude that $L_0 \subset T^{\langle 1\overline{2} \rangle} \cap L(T^{\langle 1\overline{2}3 \rangle})$.
On the other hand, by the definition of $T^{\langle 1\overline{2} \rangle}$, 
all primes of $T^{\langle 1\overline{2}3 \rangle}$ lying above $\p_3$ are totally ramified in the extension $T^{\langle 1\overline{2} \rangle} \cap L(T^{\langle 1\overline{2}3 \rangle})/T^{\langle 1\overline{2}3 \rangle}$.
It follows that $T^{\langle 1\overline{2} \rangle} \cap L(T^{\langle 1\overline{2}3 \rangle})=T^{\langle 1\overline{2}3 \rangle}$, and hence $L_0=T^{\langle 1\overline{2}3 \rangle}$.
Therefore, the module $X(T^{\langle 1\overline{2}3 \rangle})_{\Gal(T^{\langle 1\overline{2}3 \rangle}/k)}$ is trivial, 
and consequently, by Nakayama's lemma, 
$X(T^{\langle 1\overline{2}3 \rangle})$ itself is trivial.
\end{proof}

From this point until the conclusion of Lemma \ref{X(k^{(3)})=0 in case deg 6 dim=0,1 and a+b notequiv 0}, we assume that $a^2+bc \not\equiv 0$.
(Similar arguments can be applied to the cases where $b^2-ac \not\equiv 0$ or $c^2+ab \not\equiv 0$.)
Consider the field $T^{\langle 1\overline{2} \rangle} \leftrightarrow \langle \I{1}, \Io{2} \rangle=\langle \gam,x^a y^b z^c \rangle$.

\begin{lem}\label{X(T^{1,v2})=0}
Under the assumption that $a^2+bc \not\equiv 0 \pmod p$, we have
$X(T^{\langle 1\overline{2} \rangle})=0$.
\end{lem}

\begin{proof}
Since $a^2+bc \not\equiv 0$, it follows that
$\langle \I{1}, \Io{2}, \D{3} \rangle=\langle \gam,x^a y^b, x^{-c} y^a, z \rangle=\langle \gam,x,y,z \rangle$.
This implies that the prime ideal $\p_3$ is totally inert in the extension $T^{\langle 1\overline{2}3 \rangle}/k$.
Hence, there is exactly one prime of $T^{\langle 1\overline{2}3 \rangle}$ lying above $\p_3$.
Hence, in the extension $T^{\langle 1\overline{2} \rangle}/T^{\langle 1\overline{2}3 \rangle}$, exactly one prime is ramified, and that prime is totally ramified.
Combining this with Lemma \ref{X(T^{1,v2,3})=0}, we conclude that $X(T^{\langle 1\overline{2} \rangle})=0$ 
by Iwasawa's criterion \cite{Iwasawa1956}.
\end{proof}

By Lemma \ref{(a^2+bc)+(c^2+ab) notequiv 0 or (b^2-ac)+(c^2+ab) notequiv 0},
we define a $\Zp$-extension $N$ of $k$ and its complex conjugate $\overline{N}$ as follows.
If $(a^2+bc)+(c^2+ab) \not\equiv 0$, we set
\begin{eqnarray*}
\begin{cases}
N \leftrightarrow 
\langle \I{1}, \Do{2} \rangle 
=\langle \gam, y, x^a z^c \rangle,
\\ 
\overline{N} \leftrightarrow 
\langle \Io{1}, \D{2} \rangle 
=\langle \gam x^{-b-c} z^{a+b}, y, x^a z^c \rangle.
\end{cases}
\end{eqnarray*}
Since $(a^2+bc)+(c^2+ab) \not\equiv 0$, 
we can easily verify that at least one of $a$ and $c$ is nonzero modulo $p$.
Hence, both $N$ and $\overline{N}$ are $\Zp$-extensions.
Moreover, since $(a^2+bc)+(c^2+ab) \not\equiv 0$ again, it follows that
\begin{eqnarray}\label{N^{(1)} cap overline{N^{(1)}}}
N \cap \overline{N} =k 
\leftrightarrow
\langle \gam, y, x^a z^c, x^{-b-c}z^{a+b} \rangle
=\langle \gam,x,y,z \rangle.
\end{eqnarray}
On the other hand, if $(b^2-ac)+(c^2+ab) \not\equiv 0$, then we define
\begin{eqnarray*}
\begin{cases}
N \leftrightarrow 
\langle \D{1}, \Io{2} \rangle 
=\langle \gam, x, y^b z^c \rangle,
\\
\overline{N} \leftrightarrow 
\langle \Do{1}, \I{2} \rangle 
=\langle \gam y^{a-c} z^{a+b}, x,y^b z^c \rangle.
\end{cases}
\end{eqnarray*}
By a similar argument, both $N$ and $\overline{N}$ are $\Zp$-extensions, and
$$
N \cap \overline{N} \leftrightarrow
\langle \gam, x, y^b z^c, y^{a-c} z^{a+b} \rangle
=\langle \gam,x,y,z \rangle.
$$
We assume that the former condition $(a^2+bc)+(c^2+ab) \not\equiv 0$ holds.
Then all primes lying above $\overline{\p_2}$ are totally inert in the extension $T^{\langle 1\overline{2} \rangle}/N$.
We now focus on the field $T^{\langle 1 \rangle} \leftrightarrow \I{1}$.
(In the complementary case where $(b^2-ac)+(c^2+ab) \not\equiv 0$, by interchanging the roles of $\p_1$ and $\overline{\p_2}$, we instead consider $T^{\langle \overline{2} \rangle} \leftrightarrow \Io{2}$.)

Let $L_2$ denote the maximal subfield of $L(T^{\langle 1 \rangle})$ that is abelian over $T^{\langle 1\overline{2} \rangle}$.
Then, since $\Gal(T^{\langle 1 \rangle}/T^{\langle 1\overline{2} \rangle}) \isom \Zp$, 
there exists an isomorphism $\Gal(L_2/T^{\langle 1\overline{2} \rangle}) \isom X(T^{\langle 1 \rangle})_{\Gal(T^{\langle 1 \rangle}/T^{\langle 1\overline{2} \rangle})}$.
By Lemma \ref{X(T^{1,v2})=0}, we know that $X(T^{\langle 1\overline{2} \rangle})$ is trivial, 
and $\overline{\p_2}$ is totally inert in the extension $T^{\langle 1\overline{2} \rangle}/N$.
Hence, by the same argument as in the proof of Lemma \ref{L_2/N_p is abelian in case non-Gal deg 4}, 
we obtain the following lemma, whose proof we omit.

\begin{lem}\label{L_2/N^{(1)} is abelian in case deg 6 dim=0,1 and a+b notequiv 0}
Suppose that $a^2+bc \not\equiv 0$ and $(a^2+bc)+(c^2+ab) \not\equiv 0 \pmod p$.
With the notation above, the extension $L_2/N$ is abelian.
\end{lem}

Using a technique similar to that in \S \ref{another proof in case non-Gal deg 4},
we prove the following:

\begin{lem}\label{X(k^{(3)})=0 in case deg 6 dim=0,1 and a+b notequiv 0}
Suppose that $a^2+bc \not\equiv 0 \pmod p$.
Then either $X(T^{\langle 1 \rangle})$ or $X(T^{\langle \overline{2} \rangle})$ is trivial,
according as $(a^2+bc)+(c^2+ab) \not\equiv 0$ or $(b^2-ac)+(c^2+ab) \not\equiv 0 \pmod p$,
respectively. 
\end{lem}

\begin{proof}
In the case where $(b^2-ac)+(c^2+ab) \not\equiv 0$, by interchanging the roles of $\p_1$ and $\overline{\p_2}$, 
we can prove the claim similarly.
Thus, we may assume that 
$(a^2+bc)+(c^2+ab) \not\equiv 0$.

We now show that $\Gal(L_2/T^{\langle 1 \rangle})$ is trivial.
To do so, it suffices to show that $L_2\tilde{k}=\tilde{k}$,
since $\p_1$ is totally ramified in $\tilde{k}/T^{\langle 1 \rangle}$, which implies that 
$\Gal(L_2/T^{\langle 1 \rangle}) \isom \Gal(L_2\tilde{k}/\tilde{k})$.
By Lemma \ref{L_2/N^{(1)} is abelian in case deg 6 dim=0,1 and a+b notequiv 0}, the extension $L_2\tilde{k}/N$ is abelian.
Let $\overline{L_2}$ denote the maximal subfield of $L(T^{\langle \overline{1} \rangle})$ that is abelian over $T^{\langle \overline{1}2 \rangle}$.
Similarly to Lemma \ref{L_2/N^{(1)} is abelian in case deg 6 dim=0,1 and a+b notequiv 0}, the extension
$\overline{L_2}\tilde{k}/\overline{N}$ 
is also abelian
(see Figure \ref{X(k^{(3)}) = 0 galois diagram in case deg 6 dim=0,1 and a+b notequiv 0}).
\begin{figure}[H]
$$ 
\begin{xy}
(0,12) *{\text{$k_{{\rm cyc}}$}}="kcyc",
(0,0) *{\text{$k$}}="k",
%
(-20,12) *{\text{$N$}}="N1",
(-20,24) *{\text{$T^{\langle 1\overline{2} \rangle}$}}="k2",
(-20,36) *{\text{$T^{\langle 1 \rangle}$}}="k3",
(20,12) *{\text{$\overline{N}$}}="N1-conj",
(20,24) *{\text{$T^{\langle \overline{1}2 \rangle}$}}="k2-conj",
(20,36) *{\text{$T^{\langle \overline{1} \rangle}$}}="k3-conj",
(0,48) *{\text{$\tilde{k}$}}="tildek",
(0,24) *{\text{$N\overline{N}$}}="N1N1-conj",
(-40,48) *{\text{$L_2$}}="L2",
(40,48) *{\text{$\overline{L_2}$}}="L2-conj",
(-10,54) *{\text{$M$}}="M",
(10,54) *{\text{$\overline{M}$}}="M-conj",
(-20,60) *{\text{$L_2\tilde{k}$}}="tildekL2",
(20,60) *{\text{$\overline{L_2}\tilde{k}$}}="tildekL2-conj",
\ar@{-} "k";"N1",
\ar@{-} "N1";"k2",
\ar@{-} "k2";"k3",
\ar@{-} "k3";"tildek",
\ar@{-} "N1";"N1N1-conj",
\ar@{-} "k";"N1-conj",
\ar@{-} "N1-conj";"k2-conj",
\ar@{-} "k2-conj";"k3-conj",
\ar@{-} "k3-conj";"tildek",
\ar@{-} "N1-conj";"N1N1-conj",
\ar@{-} "k";"kcyc",
\ar@{-} "kcyc";"N1N1-conj",
\ar@{-} "N1N1-conj";"tildek",
\ar@{-} "tildek";"M",
\ar@{-} "M";"tildekL2",
\ar@{-} "tildek";"M-conj",
\ar@{-} "M-conj";"tildekL2-conj",
\ar@{-} "L2";"tildekL2",
\ar@{-} "L2-conj";"tildekL2-conj",
\ar@{-} "k3";"L2",
\ar@{-} "k3-conj";"L2-conj",
\ar@/^10mm/ @{.} "N1";"tildekL2"^{\rm abel}
\ar@/_10mm/ @{.} "N1-conj";"tildekL2-conj"_{\rm abel}
\end{xy}
$$
\caption{}
\label{X(k^{(3)}) = 0 galois diagram in case deg 6 dim=0,1 and a+b notequiv 0}
\end{figure}
We have the following chain of field inclusions:
\begin{eqnarray}\label{k subset kcyc subset NN in case non-Gal deg 6}
k= N \cap \overline{N} \subset \kcyc \subset N\overline{N}
\ 
(\leftrightarrow \langle y, x^a z^c \rangle),
\end{eqnarray}
where the first equality follows from (\ref{N^{(1)} cap overline{N^{(1)}}}).
Let $\mathcal{C}_{\tilde{k}/\kcyc}$ denote the subfield of $L(\tilde{k})/\tilde{k}$ characterized by 
\begin{eqnarray}\label{central tilde{k}/kcyc}
\Gal(\mathcal{C}_{\tilde{k}/\kcyc}/\tilde{k}) \isom X(\tilde{k})_{\Gal(\tilde{k}/\kcyc)}.
\end{eqnarray}
Applying Lemma \ref{Galois L/F} with $F=k^+$, $K=\tilde{k}$, and $H=\Gal(\tilde{k}/\kcyc)$ (so $L=\mathcal{C}_{\tilde{k}/\kcyc}$),
we conclude that $\mathcal{C}_{\tilde{k}/\kcyc}$ is a Galois extension of $k^+$.
In particular, the complex conjugation $J \in \Gal(k/k^+) \isom \Gal(\kcyc/\kcyc^+)$ acts naturally on the module $X(\tilde{k})_{\Gal(\tilde{k}/\kcyc)}$.
Next, applying Lemma \ref{Galois L/F} again with $F=k$, $K=T^{\langle 1 \rangle}$, and $H=\Gal(T^{\langle 1 \rangle}/T^{\langle 1\overline{2} \rangle})$ (hence $L=L_2$), we deduce that $L_2/k$ is Galois, and so is the compositum $L_2 \tilde{k}/k$.
By Lemma \ref{G/H act on H}, the group $\Gal(\tilde{k}/k)$ acts naturally on $\Gal(L_2 \tilde{k}/\tilde{k})$.
Let $M$ and $\overline{M}$ denote the subfields of $L_2 \tilde{k}/\tilde{k}$ and $\overline{L_2} \tilde{k}/\tilde{k}$ characterized by 
$$
\Gal(M/\tilde{k}) \isom \Gal(L_2 \tilde{k}/\tilde{k})_{\Gal(\tilde{k}/\kcyc)}
\ \ 
{\rm and}
\ \ 
\Gal(\overline{M}/\tilde{k}) \isom \Gal \left( \overline{L_2}\tilde{k}/\tilde{k} \right)_{\Gal(\tilde{k}/\kcyc)},
$$
respectively.
Note that $\overline{M}$ is the complex conjugate of $M$.
The natural projection $X(\tilde{k}) \surj \Gal(L_2 \tilde{k}/\tilde{k})$ induces a surjection
$X(\tilde{k})_{\Gal(\tilde{k}/\kcyc)} \surj \Gal(M/\tilde{k})$.

We now show that $M=\overline{M}$.
Fix an extension 
$\tilde{J} \in \Gal(L(\tilde{k})/k^+)$
of $J$, 
and define the inner automorphism 
$\varphi_J \colon \Gal(L(\tilde{k})/\kcyc) \to \Gal(L(\tilde{k})/\kcyc)$
by 
$\varphi_J(g) :=\tilde{J} g \tilde{J}^{-1}$ $(g \in \Gal(L(\tilde{k})/\kcyc))$.
Note that $\tilde{J} g \tilde{J}^{-1} \in \Gal(L(\tilde{k})/\kcyc)$ because $\Gal(L(\tilde{k})/\kcyc)$ is a normal subgroup of $\Gal(L(\tilde{k})/k^+)$, 
and that this map depends on the choice of $\tilde{J}$.
Consider the following commutative diagram of exact sequences:
\begin{eqnarray}\label{CD tilde{k})/kcyc}
\begin{CD}
1 
@>>>
X(\tilde{k})
@>>> 
\Gal(L(\tilde{k})/\kcyc)
@>>> 
\Gal(\tilde{k}/\kcyc)
@>>>
1
\\
& &   @VV \varphi_J V   @VV \varphi_J V  @VV J V
\\
1 
@>>>
X(\tilde{k})
@>>> 
\Gal(L(\tilde{k})/\kcyc)
@>>> 
\Gal(\tilde{k}/\kcyc)
@>>>
1,
\end{CD}
\end{eqnarray}
where the map $J$ denote the action of complex conjugation.
Taking the associated Hochschild--Serre spectral sequence yields an exact sequence of $\Gal(k/k^+)$-modules:
$$
H_2(\Gal(\tilde{k}/\kcyc),\Zp) \to X(\tilde{k})_{\Gal(\tilde{k}/\kcyc)} \to \Gal(L(\tilde{k})/\kcyc)^{{\rm ab}} \to \Gal(\tilde{k}/\kcyc) \to 0,
$$
where $\Gal(L(\tilde{k})/\kcyc)^{{\rm ab}}$ denotes the maximal abelian quotient.
Under our assumptions, the maximal abelian subextension of $L(\tilde{k})/\kcyc$ is $L(\kcyc)=\tilde{k}$, yielding an isomorphism 
$\Gal(L(\tilde{k})/\kcyc)^{{\rm ab}} \isom X(\kcyc)=\Gal(\tilde{k}/\kcyc)$.
Furthermore, it is known that there is an isomorphism of $\Gal(k/k^+)$-modules 
$$
H_2(\Gal(\tilde{k}/\kcyc),\Zp) \isom X(\kcyc) \wedge_{\Zp} X(\kcyc),
$$
where $\Gal(k/k^+)$ acts diagonally on the wedge product.
Combining these with the exact sequence above, we obtain a surjection
\begin{eqnarray}\label{X wedge X to X_Gal(tilde{k}/kcyc)}
X(\kcyc) \wedge_{\Zp} X(\kcyc) \surj X(\tilde{k})_{\Gal(\tilde{k}/\kcyc)}
\end{eqnarray}
of $\Gal(k/k^+)$-modules.
Since $X(\kcyc)=X(\kcyc)^-$, the group $\Gal(k/k^+)$ acts trivially on both sides, 
and hence acts trivially on the quotient $\Gal(M/\tilde{k})$ of $X(\tilde{k})_{\Gal(\tilde{k}/\kcyc)}$.
Consequently,
$$
M=\overline{M}.
$$
Since $M/N$ is abelian, 
$\Gal(N\overline{N}/N)$ acts trivially on $\Gal(M/N\overline{N})$ by Lemma \ref{G/H act on H}.
Similarly, $\Gal(N\overline{N}/\overline{N})$ acts trivially on $\Gal(\overline{M}/N\overline{N})=\Gal(M/N\overline{N})$.
Since every element of 
$\Gal(N \overline{N}/\kcyc)$ can be expressed as a product of elements of 
$\Gal(N \overline{N}/N)$ and $\Gal(N \overline{N}/\overline{N})$ by (\ref{k subset kcyc subset NN in case non-Gal deg 6}), 
$\Gal(N\overline{N}/\kcyc)$ acts trivially on $\Gal(M/N\overline{N})$.
Hence, by Lemma \ref{G/H act on H}, the extension $M/\kcyc$ is abelian.
Since $M/\kcyc$ is unramified, we have $M \subset L(\kcyc)=\tilde{k}$, and thus 
$M=\tilde{k}$.
By Nakayama's lemma, it follows that $\Gal(L_2 \tilde{k}/\tilde{k})$ is trivial, 
and therefore $\Gal(L_2/T^{\langle 1 \rangle})$ is also trivial.
Again, by Nakayama's lemma, we conclude that $X(T^{\langle 1 \rangle})$ is trivial.
\end{proof}


We now prove Proposition \ref{X(T^{1})=0}.
Recall that we chose $\sigma|_k \in \Hom(k,F)$ satisfying the conditions in (\ref{def of primes in F in case non-Gal deg 6}) and Lemma \ref{complex embed in case deg 6}.
Hence, for any subgroup $H \subset \Gal(\tilde{k}/k)$ and any prime ideal $\p$ of $k$ lying above $p$, the decomposition behavior of the prime $\sigma|_k(\p)$ in the fixed field $\tilde{k}^{\sigma|_k(H)}$ under $\sigma|_k(H)$ coincides with that of $\p$ in $\tilde{k}^H$.
Therefore, by the arguments from Lemma \ref{X(T^{1,v2,3})=0} through Lemma \ref{X(k^{(3)})=0 in case deg 6 dim=0,1 and a+b notequiv 0}, we conclude that 
$$
X(T^{\langle \overline{2} \rangle})=0
$$
under the assumptions that $a^2+bc \not\equiv 0$ and $(a^2+bc)+(c^2+ab) \not\equiv 0$.
Indeed, suppose that these conditions hold. 
Consider the field $T^{\langle \overline{2}3\overline{1} \rangle} \leftrightarrow \sigma|_k(\langle \I{1}, \Io{2}, \I{3} \rangle)=\langle \Io{2}, \I{3}, \Io{1} \rangle
$.
Then, by the same argument as in the proof of Lemma \ref{X(T^{1,v2,3})=0}, 
we have $X(T^{\langle \overline{2}3\overline{1} \rangle})=0$.
Next, we deduce $X(T^{\langle \overline{2}3 \rangle})=0$ analogously to Lemma \ref{X(T^{1,v2})=0}, 
since $\sigma|_k(\p_3)=\overline{\p_1}$ is totally inert in $T^{\langle \overline{2}3 \rangle}/k$, 
and it is the unique prime that is totally ramified in the extension $T^{\langle \overline{2}3 \rangle}/T^{\langle \overline{2}3\overline{1} \rangle}$.
Furthermore, note that all primes lying above $\sigma|_k(\overline{\p_2})=\p_3$ are totally inert in the extension $T^{\langle \overline{2}3 \rangle}/N'$,
where $N':=\tilde{k}^{\sigma|_k(\langle \I{1}, \Do{2} \rangle)}=\tilde{k}^{\langle \Io{2}, \D{3} \rangle}$.
Thus, we conclude that the extension $T^{\langle \overline{2}3 \rangle}/N'$ is abelian.
Hence, by the same reasoning as in Lemma \ref{X(k^{(3)})=0 in case deg 6 dim=0,1 and a+b notequiv 0}, it follows that $X(T^{\langle \overline{2} \rangle})=0$.
By applying $\sigma|_k$ repeatedly, we obtain the triviality of $X(T^{\langle 3 \rangle})$, $X(T^{\langle \overline{1} \rangle})$, $X(T^{\langle 2 \rangle})$, and $X(T^{\langle \overline{3} \rangle})$ in the same manner.
Moreover, when $(b^2-ac)+(c^2+ab) \not\equiv 0$, the same argument applies, yielding the same conclusion.

For the remaining cases where $b^2-ac \not\equiv 0$ and $c^2+ab \not\equiv 0$, 
we apply analogous arguments and thus conclude the triviality of the corresponding modules, completing the proof of Proposition \ref{X(T^{1})=0}.
For clarity, Table \ref{remain cases b^2-ac notequiv 0 and c^2-ab notequiv 0} 
summarizes the fields involved in the similar arguments from Lemma \ref{X(T^{1,v2,3})=0} to Lemma \ref{X(k^{(3)})=0 in case deg 6 dim=0,1 and a+b notequiv 0}.
The choice of $N$ and the relevant fields in Lemma \ref{X(k^{(3)})=0 in case deg 6 dim=0,1 and a+b notequiv 0} depends on whether $(a^2+bc)+(c^2+ab) \not\equiv 0$ or $(b^2-ac)+(c^2+ab) \not\equiv 0$.

{
\renewcommand{\arraystretch}{1.3} 
\begin{longtable}{|c||c|c|c|c|}
\hline
   cases & in Lemma \ref{X(T^{1,v2,3})=0} & in Lemma \ref{X(T^{1,v2})=0} & N & in Lemma \ref{X(k^{(3)})=0 in case deg 6 dim=0,1 and a+b notequiv 0}
\\
\hline
\hline
\multirow{2}{*}{$a^2+bc \not\equiv 0$}
 & 
\multirow{6}{*}{$T^{\langle 1\overline{2}3 \rangle}$}
 & 
\multirow{2}{*}{$T^{\langle 1\overline{2} \rangle}$}
  & $\tilde{k}^{\langle \I{1},\Do{2} \rangle}$ & $T^{\langle 1 \rangle}$
\\
\cline{4-5}
 &  &  & $\tilde{k}^{\langle \D{1}, \Io{2} \rangle}$ & $T^{\langle\overline{2} \rangle}$
\\
\cline{1-1}\cline{3-5}
\multirow{2}{*}{$b^2-ac \not\equiv 0$}
 & 
 & 
\multirow{2}{*}{$T^{\langle \overline{2}3 \rangle}$} 
 & 
 $\tilde{k}^{\langle \Io{2},\D{3} \rangle}$ & $T^{\langle \overline{2} \rangle}$
\\
\cline{4-5}
 &  &  & $\tilde{k}^{\langle \Do{1}, \I{3} \rangle}$ & $T^{\langle 3 \rangle}$
\\
\cline{1-1}\cline{3-5}
\multirow{2}{*}{$c^2+ab \not\equiv 0$}
 & 
 & 
\multirow{2}{*}{$T^{\langle 13 \rangle}$}
 & 
 $\tilde{k}^{\langle \D{1},\I{3} \rangle}$ & $T^{\langle 3 \rangle}$
\\
\cline{4-5}
 &  &  & $\tilde{k}^{\langle \I{1}, \D{3} \rangle}$ & $T^{\langle 1 \rangle}$
\\
\hline
\caption{}
\label{remain cases b^2-ac notequiv 0 and c^2-ab notequiv 0}
\end{longtable}
}

\subsubsection{Climbing down $\Zp$-extensions}
\label{descent}

Let $\mathcal{C}_{\tilde{k}/k}$ denote the subfield of $L(\tilde{k})/\tilde{k}$ corresponding to $\Gal(\mathcal{C}_{\tilde{k}/k}/\tilde{k}) \isom X(\tilde{k})_{\Gal(\tilde{k}/k)}$,
noting that this differs from the field defined in (\ref{central tilde{k}/kcyc}).
To prove the triviality of $X(\tilde{k})$, it suffices to show that $\mathcal{C}_{\tilde{k}/k}=\tilde{k}$.
Applying Lemma \ref{Galois L/F} with $F=k^+$, $K=\tilde{k}$, and $H=\Gal(\tilde{k}/k)$ (so $L=\mathcal{C}_{\tilde{k}/k}$),
we conclude that $\mathcal{C}_{\tilde{k}/k}$ is a Galois extension of $k^+$.
In particular, the complex conjugation $J \in \Gal(k/k^+) \isom \Gal(\kcyc/\kcyc^+)$ acts naturally on the module $X(\tilde{k})_{\Gal(\tilde{k}/k)}$.
For $i=1,2,3$, denote $Z^{\langle i\overline{i} \rangle} \leftrightarrow \langle \D{i},\Do{i}\rangle$.

\begin{lem}\label{mathcal{C}_{tilde{k}/k}/Z^{langle ioverline{i} rangle} is abelian}
For all $1 \le i \le 3$, the extension $\mathcal{C}_{\tilde{k}/k}/Z^{\langle i\overline{i} \rangle}$ is abelian.
\end{lem}

\begin{proof}
First, we show that $\mathcal{C}_{\tilde{k}/k}/T^{\langle i\overline{i} \rangle}$ is abelian.
Similar to (\ref{CD tilde{k})/kcyc}), we consider the following commutative diagram of exact sequences:
\begin{eqnarray}\label{ex seq Gal(mathcal{C}_{tilde{k}/k}/T^{langle ioverline{i} rangle})}
\begin{CD}
1 
@>>>
X(\tilde{k})_{\Gal(\tilde{k}/k)}
@>>> 
\Gal(\mathcal{C}_{\tilde{k}/k}/T^{\langle i\overline{i} \rangle})
@>>> 
\langle \I{i},\Io{i} \rangle
@>>>
1
\\
& &   @VV J V   @VV \varphi_J V  @VV J V
\\
1 
@>>>
X(\tilde{k})_{\Gal(\tilde{k}/k)}
@>>> 
\Gal(\mathcal{C}_{\tilde{k}/k}/T^{\langle i\overline{i} \rangle})
@>>> 
\langle \I{i},\Io{i} \rangle
@>>>
1.
\end{CD}
\end{eqnarray}
Here, $\varphi_J \colon \Gal(\mathcal{C}_{\tilde{k}/k}/T^{\langle i\overline{i} \rangle}) \to \Gal(\mathcal{C}_{\tilde{k}/k}/T^{\langle i\overline{i} \rangle})$
is the inner automorphism defined analogously to that in (\ref{CD tilde{k})/kcyc})
and the map $J$ denotes the action of complex conjugation.
%
Taking the associated Hochschild--Serre spectral sequence yields the following commutative diagram:
\begin{eqnarray}\label{Hochschild--Serre central tilde{k}/k}
\begin{CD}
\langle \I{i},\Io{i} \rangle \wedge \langle \I{i},\Io{i} \rangle 
@> \del >>
X(\tilde{k})_{\Gal(\tilde{k}/k)}
@>>> 
\Gal(\mathcal{C}_{\tilde{k}/k}/T^{\langle i\overline{i} \rangle})^{\rm ab}
@>>> 
\langle \I{i},\Io{i} \rangle 
@>>>
0_{\ \ }
\\
 @VV J V   @VV J V   @VV \varphi_J V  @VV J V
\\
 \langle \I{i},\Io{i} \rangle \wedge \langle \I{i},\Io{i} \rangle 
@> \del >>
X(\tilde{k})_{\Gal(\tilde{k}/k)}
@>>> 
\Gal(\mathcal{C}_{\tilde{k}/k}/T^{\langle i\overline{i} \rangle})^{\rm ab}
@>>> 
\langle \I{i},\Io{i} \rangle 
@>>>
0.
\end{CD}
\end{eqnarray}
Since $\langle \I{i},\Io{i} \rangle \isom \Zp[\Gal(k/k^+)]$ as $\Gal(k/k^+)$-module, 
the complex conjugation $J \in \Gal(k/k^+)$ acts as inverse on the wedge product $\langle \I{i},\Io{i} \rangle \wedge \langle \I{i},\Io{i} \rangle$.
On the other hand, from the natural projection 
$
X(\tilde{k})_{\Gal(\tilde{k}/\kcyc)} \surj X(\tilde{k})_{\Gal(\tilde{k}/k)}
$
and the fact that $J$ acts trivially on $X(\tilde{k})_{\Gal(\tilde{k}/\kcyc)}$ 
(see (\ref{X wedge X to X_Gal(tilde{k}/kcyc)})), 
it follows that $J$ also acts trivially on $X(\tilde{k})_{\Gal(\tilde{k}/k)}$.
This implies that $\del=0$.
Thus, comparing the exact sequence (\ref{ex seq Gal(mathcal{C}_{tilde{k}/k}/T^{langle ioverline{i} rangle})}) with the upper row of the commutative diagram (\ref{Hochschild--Serre central tilde{k}/k}), we conclude that 
$$
\Gal(\mathcal{C}_{\tilde{k}/k}/T^{\langle i\overline{i} \rangle})^{\rm ab} = \Gal(\mathcal{C}_{\tilde{k}/k}/T^{\langle i\overline{i} \rangle}),
$$ 
which shows that the extension $\mathcal{C}_{\tilde{k}/k}/T^{\langle i\overline{i} \rangle}$ is abelian.

Since $X(T^{\langle i \rangle})=0$ by Proposition \ref{X(T^{1})=0} and $\overline{\p_i}$ is totally ramified in the extension $T^{\langle i \rangle}/T^{\langle i\overline{i} \rangle}$, it follows that $X(T^{\langle i\overline{i} \rangle})=0$.
By definition, the primes lying above $\p_i$ and $\overline{\p_i}$ are totally inert in the extension $T^{\langle i\overline{i} \rangle}/Z^{\langle i\overline{i} \rangle}$.
Therefore, by an argument similar to that in the proof of Lemma \ref{L_2/N_p is abelian in case non-Gal deg 4}, 
we conclude that the extension $\mathcal{C}_{\tilde{k}/k}/Z^{\langle i\overline{i} \rangle}$ is abelian.
\end{proof}

Recall (\ref{equiv (iii) by decomp gp}) and Lemma \ref{|K| not cong 0 iff}, and note that we assume $|K|=2(a+b)((a+b)^2+3c^2) \not\equiv 0$.
Now, consider 
$$
\mathcal{M} \leftrightarrow \left\langle \langle \D{1},\Do{1} \rangle \cap \langle \D{2},\Do{2} \rangle, \I{3} \right\rangle.
$$
Since $a+b \not\equiv 0$ and 
$Z^{\langle 1\overline{1} \rangle} \cap Z^{\langle 2\overline{2} \rangle} \leftrightarrow \langle \D{1},\Do{1}, \D{2},\Do{2}  \rangle
=\langle \gam,x,y,z^{a+b},z^c \rangle$,
we have that 
$k=Z^{\langle 1\overline{1} \rangle} \cap Z^{\langle 2\overline{2} \rangle}$ 
and that the compositum 
$Z^{\langle 1\overline{1} \rangle} Z^{\langle 2\overline{2} \rangle}$ 
is a ${\Zp}^2$-extension over $k$.
Applying again a technique similar to that in \S \ref{another proof in case non-Gal deg 4} and Lemma \ref{X(k^{(3)})=0 in case deg 6 dim=0,1 and a+b notequiv 0}, we obtain the following:

\begin{cor}
The extension $\mathcal{C}_{\tilde{k}/k}/\mathcal{M}$ is abelian. 
\end{cor}

\begin{proof}
By Lemmas \ref{mathcal{C}_{tilde{k}/k}/Z^{langle ioverline{i} rangle} is abelian} and \ref{G/H act on H}, 
the group 
$\Gal(Z^{\langle 1\overline{1} \rangle}Z^{\langle 2\overline{2} \rangle}/k)$,
in particular its subgroup $\Gal(Z^{\langle 1\overline{1} \rangle}Z^{\langle 2\overline{2} \rangle}/\mathcal{M}) \isom \Zp$, 
acts trivially on 
$\Gal(\mathcal{C}_{\tilde{k}}/Z^{\langle 1\overline{1} \rangle}Z^{\langle 2\overline{2} \rangle})$.
Hence, again by Lemma \ref{G/H act on H}, the extension $\mathcal{C}_{\tilde{k}/k}/\mathcal{M}$ is abelian.
\end{proof}

\begin{lem}\label{property of M}
\ 
\begin{itemize}
\setlength{\parskip}{0pt} 
\setlength{\itemsep}{0pt} 
\item[{\rm (i)}]
$\mathcal{M}/k$ is a $\Zp$-extension in which both $\p_3$ and $\overline{\p_3}$ are totally inert.
\item[{\rm (ii)}]
In the extension $\mathcal{M}/k$, the prime ideals $\p_1$, $\overline{\p_1}$, $\p_2$, and $\overline{\p_2}$ are ramified.
\item[{\rm (iii)}]
Suppose that $(a^2+bc)+(c^2+ab) \not\equiv 0$ {\rm (}resp. $(b^2-ac)+(c^2+ab) \not\equiv 0${\rm )}.
Then $\p_1$ and $\overline{\p_1}$ {\rm (}resp. $\p_2$ and $\overline{\p_2}${\rm )} do not split in $T^{\langle 3\overline{3} \rangle}/k$.
\end{itemize}
\end{lem}

\begin{figure}[H]
$$ 
\begin{xy}
(0,0) *{\text{$\langle \gam,x,y,z \rangle$}}="k",
(-36,24) *{\text{$\langle \D{1},\Do{1} \rangle$}}="D1",
(36,24) *{\text{$\langle \D{2},\Do{2} \rangle$}}="D2",
(0,24) *{\text{$\bullet$}}="M",
(36,48) *{\text{$T^{\langle 3\overline{3} \rangle}$}},="T3",
(0,48) *{\text{$\langle \D{1},\Do{1} \rangle \cap \langle \D{2},\Do{2} \rangle$}}="D1 cap D2",
(-51,25) *{\text{$Z^{\langle 1\overline{1} \rangle} \leftrightarrow$}},
(51,25) *{\text{$\leftrightarrow Z^{\langle 2\overline{2} \rangle}$}},
(12,0) *{\text{$\leftrightarrow k$}},
(6,24) *{\text{$\leftrightarrow \mathcal{M}$}},
\ar@{-} "k";"D1"
^(.5){\text{$\substack{\p_3, \overline{\p_3} : \text{ram}}$}}
\ar@{-} "k";"D2"
_(.5){\text{$\substack{\p_1, \overline{\p_1} : \text{ram} }$}}
\ar@{-} "k";"M"
^(.6){\text{$\substack{\p_1, \overline{\p_1} : \text{ram} }$}}
_(.6){\text{$\substack{\p_3, \overline{\p_3} : \\ \text{totally} \\ \text{inert}}$}}
\ar@{-} "D1";"D1 cap D2"
^(.4){\text{$\substack{\p_1, \overline{\p_1} : \text{ram} \\ \p_3, \overline{\p_3} : \text{unr} }$}}
\ar@{-} "D2";"D1 cap D2"
\ar@{-} "M";"D1 cap D2"
^(.4){\text{$\substack{\p_1, \overline{\p_1} : \text{unr} \\ \p_3, \overline{\p_3} : \text{ram}}$}}
\ar@{-} "M";"T3"
%

\end{xy}
$$
\caption{}
\label{mathcal{M}}
\end{figure}

\begin{proof}
(i) 
This follows directly from the equivalence between (\ref{equiv (iii) by decomp gp}) and condition (iii) of Theorem \ref{main thm}(III).
\\
(ii)(iii) (see Figure \ref{mathcal{M}})
It suffices to prove the assertions for $\p_1$ and $\overline{\p_1}$, 
since the cases of $\p_2$ and $\overline{\p_2}$ are analogous.
Assume, for the sake of contradiction, that $\p_1$ (and hence $\overline{\p_1}$) is unramified in the extension $Z^{\langle 2\overline{2} \rangle}/k$.
Then $Z^{\langle 2\overline{2} \rangle}$ would be an infinite extension of $k$ unramified outside $\{ \p_3, \overline{\p_3} \}$ and completely decomposed at $\p_2$.
This contradicts \cite[Lemma 3]{Fujii2015}.
Therefore, $\p_1$ and $\overline{\p_1}$ are ramified in the extension $Z^{\langle 2\overline{2} \rangle}/k$, which implies that 
all primes lying above $\p_1$ and $\overline{\p_1}$ are infinitely ramified in the extension $Z^{\langle 1\overline{1} \rangle}Z^{\langle 2\overline{2} \rangle}/Z^{\langle 1\overline{1} \rangle}$.
On the other hand, the equality 
$Z^{\langle 1\overline{1} \rangle} \mathcal{M}=Z^{\langle 1\overline{1} \rangle}Z^{\langle 2\overline{2} \rangle}$ holds.
Indeed, by Lemma \ref{(a^2+bc)+(c^2+ab) (b^2-ac)+(c^2+ab) not 0}, 
the group $\langle \D{1}, \Do{1}, \I{3} \rangle=\langle \gam, x, y^{-c}z^a, y^{a+b}z^{b+c} \rangle$ is a subgroup of $\langle \gam, x,y,z \rangle$ with finite index.
Hence, $\p_3$ is infinitely ramified in $Z^{\langle 1\overline{1} \rangle}/k$, which implies that $\p_3$ is unramified in 
$Z^{\langle 1\overline{1} \rangle}Z^{\langle 2\overline{2} \rangle}/Z^{\langle 1\overline{1} \rangle}$.
It follows that
$Z^{\langle 1\overline{1} \rangle} \mathcal{M}=Z^{\langle 1\overline{1} \rangle}Z^{\langle 2\overline{2} \rangle}$.
This means that all primes of $\mathcal{M}$ lying above $\p_1$ and $\overline{\p_1}$ are unramified in 
$Z^{\langle 1\overline{1} \rangle}Z^{\langle 2\overline{2} \rangle}/\mathcal{M}$.
Hence, $\p_1$ and $\overline{\p_1}$ are ramified in $\mathcal{M}/k$.

Suppose that $(a^2+bc)+(c^2+ab) \not\equiv 0$.
Then $\langle \I{3}, \Io{3}, \D{1} \rangle=\langle \gam, x, y^{a+b}z^{b+c}, y^{-c}z^a \rangle=\langle \gam, x,y,z \rangle$,
which implies that $\p_1$ (and hence also $\overline{\p_1}$) does not split in the extension $T^{\langle 3\overline{3} \rangle}/k$.
Similarly, if $(b^2-ac)+(c^2+ab) \not\equiv 0$, 
then $\p_2$ and $\overline{\p_2}$ do not split in $T^{\langle 3\overline{3} \rangle}/k$.
\end{proof}

Note that
$\mathcal{M} \subset T^{\langle 3\overline{3} \rangle}$
by the definitions of $\mathcal{M}$ and $T^{\langle 3\overline{3} \rangle}$.
Depending on whether $(a^2+bc)+(c^2+ab) \not\equiv 0$ or $(b^2-ac)+(c^2+ab) \not\equiv 0$, we fix
$$
j=
\begin{cases}
1 
&
\text{if $(a^2+bc)+(c^2+ab) \not\equiv 0$},
\\
2
&
\text{if $(b^2-ac)+(c^2+ab) \not\equiv 0$}. 
\end{cases}
$$
We observe that $\p_j$ does not split in the $\Zp$-extension $T^{\langle 3\overline{3} \rangle} /\mathcal{M}$.
In fact, let $L^{\{ \p_j \}}(\mathcal{M})$ denote the maximal unramified abelian $p$-extension of $\mathcal{M}$ in which $\p_j$ splits completely.
By Lemma \ref{property of M}(ii) and $T^{\langle 3\overline{3} \rangle} \leftrightarrow \langle \I{3},\Io{3} \rangle$, 
the extension $T^{\langle 3\overline{3} \rangle}/\mathcal{M}$ is an unramified $\Zp$-extension.
Since we proved that $X(T^{\langle 3\overline{3} \rangle})=0$ in the proof of Lemma \ref{mathcal{C}_{tilde{k}/k}/Z^{langle ioverline{i} rangle} is abelian}, it follows that  
$L(\mathcal{M})=T^{\langle 3\overline{3} \rangle}$, 
and hence, 
$$
\mathcal{M} \subset L^{\{ \p_j \}}(\mathcal{M}) \subset T^{\langle 3\overline{3} \rangle}.
$$
Combining this with Lemma \ref{property of M}(iii), we conclude that 
$$
\mathcal{M} = L^{\{ \p_j \}}(\mathcal{M}), 
$$
which implies that $\p_j$ does not split in the $\Zp$-extension $T^{\langle 3\overline{3} \rangle} /\mathcal{M}$.

Finally, we prove Proposition \ref{main thm in case deg 6 dim=0,1 and a+b notequiv 0},
by employing a technique similar to that of Minardi \cite{Minardi} again.
Let $\mathbb{I}_3$ (resp. $\mathbb{I}_{\overline{3}}$) denote the inertia subgroup of the abelian group $\Gal(\mathcal{C}_{\tilde{k}/k}/\mathcal{M})$ corresponding to the unique prime lying above $\p_3$ (resp. $\overline{\p_3}$).
Similarly, let $\mathbb{D}_{j}$ 
(resp. $\mathbb{D}_{\overline{j}}$) 
denote the decomposition subgroup of $\Gal(\mathcal{C}_{\tilde{k}/k}/\mathcal{M})$ corresponding to the unique prime lying above $\p_j$
 (resp. $\overline{\p_j}$).
Since $\mathcal{M} = L^{\{ \p_j \}}(\mathcal{M})$, the subgroups $\mathbb{I}_3$, $\mathbb{I}_{\overline{3}}$, $\mathbb{D}_{j}$, and $\mathbb{D}_{\overline{j}}$ generate $\Gal(\mathcal{C}_{\tilde{k}/k}/\mathcal{M})$:
$$
\Gal(\mathcal{C}_{\tilde{k}/k}/\mathcal{M})
=
\langle
\mathbb{I}_3, 
\mathbb{I}_{\overline{3}}, 
\mathbb{D}_{j}, 
\mathbb{D}_{\overline{j}}
\rangle.
$$ 
Moreover, by Lemma \ref{property of M}(i) and (iii), 
the Galois group $\Gal(\mathcal{M}/k)$ acts on these four subgroups.
We know that $\Gal(\mathcal{M}/k) \isom \Zp$ acts naturally on $\Gal(\tilde{k}/\mathcal{M})$, 
and that this action is trivial by Lemma \ref{G/H act on H}.
On the other hand, since $\mathcal{C}_{\tilde{k}/k}/\tilde{k}$ is unramified, 
the natural projections
$$
\mathbb{I}_3 \to \Gal(\tilde{k}/\mathcal{M}),
\ \ 
\mathbb{I}_{\overline{3}} \to \Gal(\tilde{k}/\mathcal{M})
$$
are injective homomorphisms of $\Gal(\mathcal{M}/k)$-modules.
This implies that the action of $\Gal(\mathcal{M}/k)$ on both $\mathbb{I}_3$ and $\mathbb{I}_{\overline{3}}$ is trivial.
By Lemma \ref{property of M}(ii), $\p_j$ is unramified in $\mathcal{C}_{\tilde{k}/k}/\mathcal{M}$, 
and hence $\mathbb{D}_{j} \isom \Zp$.
Since $\p_j$ is infinitely inert in $\tilde{k}/\mathcal{M}$ by Lemma \ref{property of M}(ii) and (iii), 
the kernel $\mathbb{D}_{j} \cap \Gal(\mathcal{C}_{\tilde{k}/k}/\tilde{k})$ of the projection $\mathbb{D}_{j} \to \Gal(\tilde{k}/\mathcal{M})$ is trivial.
Hence, this projection is an injective homomorphism of $\Gal(\mathcal{M}/k)$-modules,
which implies that the action of $\Gal(\mathcal{M}/k)$ on $\mathbb{D}_{j}$ is trivial.
Similarly, the action of $\Gal(\mathcal{M}/k)$ on $\mathbb{D}_{\overline{j}}$ is also trivial.
Summarizing these arguments, we conclude that $\Gal(\mathcal{M}/k)$ acts trivially on $\Gal(\mathcal{C}_{\tilde{k}/k}/\mathcal{M})$.
This shows that $\mathcal{C}_{\tilde{k}/k}$ is abelian over $k$.
Hence, $\mathcal{C}_{\tilde{k}/k}=\tilde{k}$,
and thus $X(\tilde{k})=0$. 
This completes the proof of Proposition \ref{main thm in case deg 6 dim=0,1 and a+b notequiv 0}.

\section{Examples}\label{Examples}

Using \texttt{PARI/GP} \cite{PARI/GP}, we obtain the following examples.
Recall that, for a finite set $T$ of prime ideals of $k$ lying above $p$, $M_p^T(k)$ denotes the maximal abelian $p$-extension of $k$ that is unramified outside $p$ and in which every prime ideal in $T$ splits completely.

\subsection{Examples in degree 4}

We can check whether condition (iii) of Theorem \ref{main thm}(II) holds by computing ray class groups.

\begin{eg}
{\rm 
Let $p=3$.
Suppose that $k$ is the totally imaginary cyclic field of degree $4$ defined by the polynomial
$x^4 - x^3 + 2x^2 + 4x + 3$.
Then $3$ splits completely in $k$:
$(3)=\p \q \overline{\p} \overline{\q}$.
We see that $A(k)=0$ and that condition (ii) of Theorem \ref{main thm}(II) holds.
Define $T:=\{ \p, \overline{\p} \}$ and $\mathfrak{m}:=(\q \overline{\q})^2$.
Then the ray class group of $k$ modulo $\mathfrak{m}$ is isomorphic to $\Z/6 \Z$, and both $\p^2$ and $\overline{\p}^2$ are not principal in the ray class group.
This means that $M_p^T(k) = k$.
Therefore, condition (iii) of Theorem \ref{main thm}(II) holds, which yields $X(\tilde{k})=0$.

On the other hand, suppose that $k$ is the totally imaginary cyclic field of degree $4$ defined by the polynomial
$x^4 + 91x^2 + 637$.
Then $3$ splits completely in $k$:
$(3)=\p \q \overline{\p} \overline{\q}$.
We again see that $A(k)=0$ and that condition (ii) of Theorem \ref{main thm}(II) holds.
Define $T:=\{ \p, \overline{\p} \}$ and $\mathfrak{m}:=(\q \overline{\q})^2$.
Then the ray class group of $k$ modulo $\mathfrak{m}$ is isomorphic to $\Z/78 \Z \x \Z/2 \Z \x\Z/2 \Z$, and both $\p^{26}$ and $\overline{\p}^{26}$ are principal in the ray class group.
This means that $M_p^T(k) \neq k$.
Therefore, condition (iii) of Theorem \ref{main thm}(II) does not hold, which yields $X(\tilde{k}) \neq 0$.
}
\end{eg}

\begin{eg}
{\rm 
Let $p=3$.
Consider the totally imaginary $(2,2)$-extension $k=\Q(\sqrt{7}, \sqrt{-m})$,
where $m=11,17$.
Then $3$ splits completely in $k$:
$(3)=\p \q \overline{\p} \overline{\q}$.
We see that $A(k)=0$ and that condition (ii) of Theorem \ref{main thm}(II) holds.
Define $T:=\{ \p, \overline{\p} \}$ and $\mathfrak{m}:=(\q \overline{\q})^2$.

If $m=11$, then the ray class group of $k$ modulo $\mathfrak{m}$ is isomorphic to $\Z/24 \Z$, and both $\p^8$ and $\overline{\p}^8$ are principal in the ray class group.
This means that $M_p^T(k) = k$.
Therefore, condition (iii) of Theorem \ref{main thm}(II) holds, which yields $X(\tilde{k})=0$.

On the other hand, if $m=17$, then the ray class group of $k$ modulo $\mathfrak{m}$ is isomorphic to $\Z/120 \Z$, and both $\p^{40}$ and $\overline{\p}^{40}$ are not principal in the ray class group.
This means that $M_p^T(k) \neq k$.
Therefore, condition (iii) of Theorem \ref{main thm}(II) does not hold, which yields $X(\tilde{k}) \neq 0$.
}
\end{eg}

\begin{eg}
{\rm 
Recall the remark in Theorem \ref{main thm}(II); that is, condition (iii) can be omitted when $A(k)$ is cyclic.
Let $p=5$.
Suppose $k$ is the totally imaginary cyclic field of degree $4$ defined by the polynomial
$x^4 - x^3 + 13x^2 - 19x + 361$
or the totally imaginary $(2,2)$-extension $k=\Q(\sqrt{6}, \sqrt{-74})$.
Then $5$ splits completely in $k$.
We see that $A(k) \isom \Z/5 \Z$ and that condition (ii) of Theorem \ref{main thm}(II) holds.
Hence, $X(\tilde{k})=0$ in both cases.
}
\end{eg}

\subsection{Examples of degree 6}

If $p=3$ or $p \equiv 2 \pmod 3$,  we can check condition (iii) of Theorem \ref{main thm}(III) by computing certain ray class groups as follows.

\begin{lem}
Suppose that $p=3$ or $p \equiv 2 \pmod 3$.
With the notation of {\rm Theorem \ref{main thm}(III)}, assume that {\rm (i)} and {\rm (ii)} of  {\rm Theorem \ref{main thm}(III)} hold.
Then condition {\rm (iii)} of {\rm Theorem \ref{main thm}(III)} holds if and only if $M_p^{\{ \p_1,\p_2\}}(k)=k$.
\end{lem}

\begin{proof}
Recall that condition (iii) of Theorem \ref{main thm}(III) holds if and only if 
\begin{eqnarray}\label{eg |K|}
(a+b)((a+b)^2+3c^2) \not\equiv 0 \pmod p
\end{eqnarray}
by Lemma \ref{|K| not cong 0 iff}(i).
On the other hand, 
since $\langle \D{1}, \D{2} \rangle=\langle \gam, x,y,z^{a+b} \rangle$ by Lemma \ref{Z in case deg 6},
$M_p^{\{ \p_1,\p_2\}}(k)=k$ holds if and only if 
\begin{eqnarray}\label{eg a+b}
a+b \not\equiv 0 \pmod p.
\end{eqnarray}
Under the assumption that $p=3$ or $p \equiv 2 \pmod 3$, we see that conditions (\ref{eg |K|}) and (\ref{eg a+b}) are equivalent.
\end{proof}

\begin{eg}
{\rm 
(The same examples as those at the end of \cite{Okano2012-1}.)
Let $p=3$.
Suppose that $k$ is the compositum of $\Q(\sqrt{-m})$ and the subfield of the $61$st cyclotomic field of degree $3$,
where $m=2,11,23$.
Then $3$ splits completely in $k$:
$(3)=\p_1 \p_2 \p_3 \overline{\p_1}\; \overline{\p_2}\; \overline{\p_3}$.
Define $T:=\{ \p_1, \p_2 \}$ and $\mathfrak{m}:=(\p_3 \overline{\p_1}\; \overline{\p_2}\; \overline{\p_3})^2$.

For $m=2$, we see that $A(k)=0$ and that condition (ii) of Theorem \ref{main thm}(III) holds.
Then the ray class group of $k$ modulo $\mathfrak{m}$ is isomorphic to $\Z/258 \Z \x \Z/6 \Z \x\Z/2 \Z$.
We see that $\p_1^{86}$ and $\p_2^{86}$ generate the $3$-Sylow subgroup of the ray class group.
This means that $M_p^T(k) = k$.
Therefore, condition (iii) of Theorem \ref{main thm}(III) holds, which yields $X(\tilde{k}) = 0$.

For $m=11$, we again see that $A(k)=0$ and that condition (ii) of Theorem \ref{main thm}(III) holds.
Then the ray class group of $k$ modulo $\mathfrak{m}$ is isomorphic to $\Z/48 \Z \x \Z/48 \Z \x\Z/2 \Z$, and $\p_1^{16}$ and $\p_2^{16}$ become the same element of order $3$ in the ray class group.
This means that $M_p^T(k) \neq k$.
Therefore, condition (iii) of Theorem \ref{main thm}(III) does not hold, which yields $X(\tilde{k}) \neq 0$.

For $m=23$, we see that $A(k) \isom \Z/9 \Z$ and that condition (ii) of Theorem \ref{main thm}(III) holds.
Hence, $X(\tilde{k})=0$ (see Remark \ref{remark of main thm} or Remark \ref{in case deg 6 dim=1 p=3}).
}
\end{eg}

\begin{eg}
{\rm 
Let $p=5$.
Suppose that $k$ is the compositum of $\Q(\sqrt{-m})$ and the subfield of the $13$th cyclotomic field of degree $3$,
where $m=1, 74, 94, 321$.
Then $5$ splits completely in $k$:
$(5)=\p_1 \p_2 \p_3 \overline{\p_1}\; \overline{\p_2}\; \overline{\p_3}$.
Define $T:=\{ \p_1, \p_2 \}$ and $\mathfrak{m}:=(\p_3 \overline{\p_1}\; \overline{\p_2}\; \overline{\p_3})^2$.

For $m=1$, we see that $A(k)=0$ and that condition (ii) of Theorem \ref{main thm}(III) holds.
Then the ray class group of $k$ modulo $\mathfrak{m}$ is isomorphic to $\Z/60 \Z \x \Z/5 \Z$.
We see that $\p_1^{12}$ and $\p_2^{12}$ generate the $5$-Sylow subgroup of the ray class group.
This means that $M_p^T(k) = k$.
Therefore, condition (iii) of Theorem \ref{main thm}(III) holds, which yields $X(\tilde{k}) = 0$.

For $m=74$, we see that $A(k)$ is cyclic and that condition (ii) of Theorem \ref{main thm}(III) holds.
Then the ray class group of $k$ modulo $\mathfrak{m}$ is isomorphic to $\Z/12900 \Z \x \Z/20 \Z$.
We see that $\p_1^{516}$ and $\p_2^{516}$ generate the $5$-Sylow subgroup of the ray class group.
Hence, we again have $M_p^T(k) = k$.
Therefore, condition (iii) of Theorem \ref{main thm}(III) holds, which yields $X(\tilde{k}) = 0$.

For $m=94$, we see that $A(k)=0$ and that condition (ii) of Theorem \ref{main thm}(III) holds.
Then the ray class group of $k$ modulo $\mathfrak{m}$ is isomorphic to $\Z/3120 \Z \x \Z/20 \Z \x \Z/2 \Z \x \Z/2 \Z$, and $\p_1^{624}$ and $\p_2^{624}$ become the same element of order $5$ in the ray class group.
This means that $M_p^T(k) \neq k$.
Therefore, condition (iii) of Theorem \ref{main thm}(III) does not hold, which yields $X(\tilde{k}) \neq 0$.

For $m=321$, we see that $A(k)$ is cyclic and that condition (ii) of Theorem \ref{main thm}(III) holds.
Then the ray class group of $k$ modulo $\mathfrak{m}$ is isomorphic to $\Z/48900 \Z \x \Z/20 \Z \x \Z/2 \Z$, and $\p_1^{1956}$ and $\p_2^{1956}$ become the same element of order $25$ in the ray class group.
Hence, we again have $M_p^T(k) \neq k$.
Therefore, condition (iii) of Theorem \ref{main thm}(III) does not hold, which yields $X(\tilde{k}) \neq 0$.
}
\end{eg}





\vspace*{15pt}

{\bf  Questions.}
(1)
From the proof of Theorem \ref{main thm}, 
we see that the converse of the claim in Theorem \ref{main prop} holds if $[k:\Q] =4$ or $6$ 
(see Figure \ref{stragegy} in \S \ref{Organization of the paper}).
We can also check that it holds if $[k:\Q] =2$.
This naturally raises the question of whether the converse of the claim in Theorem \ref{main prop} holds if $[k:\Q] > 6$.
Since $X(\tilde{k}) \neq 0$ if $[k:\Q] > 6$, we are led to the following question: if $[k:\Q] > 6$ and $X'(\kcyc) =0$, is there always an anticyclotomic $\Zp$-extension $\ki$ of $k$ in which every prime ideal of $k$ lying above $p$ is non-split and ramified in $\ki/k$ such that $X'(\ki) \neq 0$?
\\
(2)
Under the conditions that $[k :\Q] >2$ and $X(\tilde{k}) \neq 0$,
is $X(\tilde{k})$ finitely generated as a $\Zp$-module?
No examples satisfying these conditions have been found, nor have any counterexamples been discovered.
If there exists a CM-field $k$ with $[k :\Q] \ge 4$ that satisfies the following conditions,
then it would provide an answer to this question:
\begin{itemize}
\setlength{\parskip}{0pt} 
\setlength{\itemsep}{0pt} 
\item[{\rm (i)}]
The prime $p$ splits completely in $k$, and Leopoldt's conjecture holds for $k$ and $p$.
\item[{\rm (ii)}]
$\mu(\kcyc/k)=0$ (this is known to hold if $k$ is abelian over $\Q$, by Ferrero and Washington \cite{Ferrero-Washington}).
\item[{\rm (iii)}]
$X(\tilde{k}) \neq 0$, and $\kcyc$ has an abelian $p$-class field tower; that is, the maximal unramified $p$-extension of $\kcyc$ is abelian.
\end{itemize}
Indeed, in this case, $X(\tilde{k})$ coincides with $\Gal(L(\kcyc)/\tilde{k})$, which is finitely generated as a $\Zp$-module.
\\
(3)
The ideas of the alternative proofs in \S \ref{another proof in case non-Gal deg 4} and \ref{Proof in the case deg 6 |K| notequiv 0} are based on those of the Generalized Greenberg conjecture.
This naturally raises the question of whether the ideas developed in this paper can be applied to the Generalized Greenberg conjecture.

\vspace*{10pt}

{\bf  Acknowledgements.}
The author would like to express his sincere gratitude to Professor Satoshi Fujii, Professor Tsuyoshi Itoh, and Dr. John Minardi for their previous works \cite{Fujii2015}, \cite{Itoh2011}, and \cite{Minardi}, 
on which much of the present proofs are based.
Their insights have contributed greatly to this study. 
The author is also deeply grateful to Professor Satoshi Fujii for pointing out errors and providing valuable comments, and to Professor Tsuyoshi Itoh for his helpful suggestions. 
He further thanks Keiichi Komatsu, Takashi Fukuda, Yasushi Mizusawa, and Kazuaki Murakami for their long-term encouragement.

{\small


}

\end{document}